\documentclass{amsart}

\usepackage[utf8]{inputenc}
\usepackage{babel}
\usepackage[nomain,symbols]{glossaries}
\usepackage[author-year]{amsrefs}

\usepackage{graphicx,amssymb,amscd,amsxtra,amsmath,amsfonts,amsthm,footnpag,bm,fge}
\usepackage{url}
\usepackage{xcolor} 
\usepackage{dirtytalk}
 
\usepackage[labelfont=bf]{caption}
\DeclareCaptionLabelSeparator{period}{.  }
\renewcommand{\thefootnote}%
{\fnsymbol{footnote}}

\theoremstyle{definition}

\newtheorem*{definition*}{Definition}
\newtheorem*{notation*}{Notation}

\swapnumbers
\theoremstyle{plain}
\newtheorem{Thm}{Theorem}[section]
\swapnumbers
\newtheorem{theorem}[Thm]{Theorem}

\newtheorem{lemma}[Thm]{Lemma}

\newtheorem{corollary}[Thm]{Corollary}

\newtheorem{proposition}[Thm]{Proposition}

\newtheorem*{Thm*}{Theorem}
\newtheorem*{theorem*}{Theorem}
\newtheorem*{Lemm*}{Lemma}
\newtheorem*{lemma*}{Lemma}
\newtheorem*{Propn*}{Proposition}
\newtheorem*{proposition*}{Proposition}
\newtheorem*{corollary*}{Corollary}
\newtheorem*{main*}{Main Theorem}
\newtheorem*{IMT*}{Invariant Manifold Theorem}

\newtheorem*{HGT*}{Hartman-Grobman Theorem}

\theoremstyle{remark}
\newtheorem*{question}{\bf Question}

\newtheorem*{remark*}{\bf Remark}

\newtheorem*{Ups}{\bf Upshot}

\newtheorem*{Example*}{\bf Example}
\newtheorem*{example*}{\bf Example}

\newtheoremstyle{vThm}%
{}{}%
{\itshape}%
{0pt}{\bfseries}%
{}{ }%
{}%
\theoremstyle{vThm}

\newtheoremstyle{vThm*}%
{}{}%
{\itshape}%
{-3pt}{\bfseries}%
{}{ }%
{\thmnote{#3}}%
\theoremstyle{vThm*}
\newtheorem*{nThm*}{}

\newcommand{\abs}[1]{\lvert#1\rvert}

\DeclareMathSymbol
{\rightrightarrows}
{\mathrel}{AMSa}{"13}

\DeclareSymbolFont{Csc}{U}{zcsc}{m}{n}
\DeclareMathSymbol{\bigsquareunion}{\mathop}{Csc}{"7C}
\DeclareMathSymbol{\bigintersection}{\mathop}{Csc}{"3C}
\DeclareMathSymbol{\bigunion}{\mathop}{Csc}{"3E}
\DeclareMathOperator{\Cos}{\mathrm{Cos}}
\DeclareMathOperator{\Sin}{\mathrm{Sin}}

\newcommand{\defun}[5]{\ensuremath{\begin{array}{lrcl}
#1:&#2 & \longrightarrow & #3\\&#4 & \longmapsto & #5\end{array}}}

\newcommand{\cell}[1]{\text{\em \small{#1}}}
\newcommand{\binomial}[2]{\ensuremath{\left( \begin{matrix}#1 \\ #2 \end{matrix} \right)}}

\newcommand{\sing}{\mathrm{sing}}
\newcommand{\bundle}[4]{\ensuremath{#1 \rightarrow #2 \stackrel{#3}{\rightarrow}{#4}}}

\newcommand{\high}{\mathrm{high}}
\newcommand{\low}{\mathrm{low}}
\renewcommand{\mid}{\mathrm{mid}}

\newcommand{\dd}{\,\mathrm{d}}

\makenoidxglossaries
\newglossaryentry{V}{
	sort=a010,
	name={\ensuremath{V(\Theta)}},
	description={Kuramoto potential}}
\newglossaryentry{K}{
	sort=a020,
	name={\ensuremath{K(\Theta)}},
	description={Kuramoto vector field $K(\Theta)=-\nabla V(\Theta)$}}
\newglossaryentry{phi}{
	sort=a030,
	name={\ensuremath{\varphi_t(\Theta)}},
	description={Time $t$ flow for Kuramoto vector}}
\newglossaryentry{DTheta}{
	sort=a040,
	name={\ensuremath{D_{\Theta}}},
	description={Diagonal passing through $\Theta \in \mathbb T^m$}}
\newglossaryentry{D}{
	sort=a050,
	name={\ensuremath{D}},
	description={Main diagonal $D=D_{\mathbf 0}$}}
\newglossaryentry{M}{
	sort=a060,
	name={\ensuremath{M(\Theta)}},
	description={Perfect Morse potential $M(\Theta)=\sum_{k=1}^m 1-\cos(\theta_k)$}}
\newglossaryentry{Q}{
	sort=a070,
	name={\ensuremath{\mathcal Q}},
	description={Quotient space of $\mathbb T^m$ by the diagonal action}} 
\newglossaryentry{classTheta}{
	sort=a080,
	name={\ensuremath{[\Theta]}},
	description={Coset $[\Theta] = \Theta+D$ in $\mathcal Q$}}
\newglossaryentry{psi}{
	sort=a090,
	name={\ensuremath{\psi_t([\Theta])}},
	description={Quotient Kuramoto flow}}
\newglossaryentry{VQ}{
	sort=a100,
    name={\ensuremath{V_{\mathcal Q}}},
    description={Kuramoto potential in quotient space $\mathcal Q$} }
\newglossaryentry{Z}{
	sort=a110,
	name={\ensuremath{Z}},
	description={The centroid in $\mathbb C$}}
\newglossaryentry{H}{
	sort=a120,
	name={\ensuremath{H}},
	description={The Hessian of $-V(\Theta)$, that is $DK(\Theta)$}}
\newglossaryentry{TI}{
	sort=a130,
	name={\ensuremath{\mathbb T^I}},
	description={Subtorus of $\mathbb T^m$ with coordinates in $I \subset [m]$}}
\newglossaryentry{QI}{
	sort=a140,
	name={\ensuremath{\mathcal Q^I}},
	description={Projection of $\mathbb T^I$ under diagonal action}}
\newglossaryentry{Kskeleton}{
	sort=a150,
	name={\ensuremath{\mathcal K}},
	description={The lower half of $\mathcal Q$, a CW complex}}
\newglossaryentry{pI}{
	sort=a160,
	name={\ensuremath{p_I}},
	description={Saddle $[\Theta] \in \mathcal Q$ with $\theta_i=0$ for all $i \not \in I$, $\theta_i=\pi$ otherwise.}}
\newglossaryentry{Vmax}{
	sort=a170,
	name={\ensuremath{\mathcal V^{\max}}},
	description={The set of points with $V_{\mathcal Q}$ maximal in $\mathcal Q$}}
\newglossaryentry{Vsing}{
	sort=a180,
	name={\ensuremath{\mathcal V^{\sing}}},
	description={The locus of non-smooth points of $\mathcal V^{\max}$}}
\newglossaryentry{Vhigh}{
	sort=a190,
	name={\ensuremath{\mathcal V^{\high}}},
	description={High potential region as a subset of $\mathcal Q$}}
\newglossaryentry{Vmid}{
	sort=a200,
	name={\ensuremath{\mathcal V^{\mid}}},
	description={Mid potential region as a subset of $\mathcal Q$}}
\newglossaryentry{Vlow}{
	sort=a210,
	name={\ensuremath{\mathcal V^{\low}}},
	description={Low potential region as a subset of $\mathcal Q$}}

\setglossarystyle{long3col}
    {\begin{longtable}{lp{\glsdescwidth}r}}%
    {\end{longtable}}%

\title{Global analysis of the Kuramoto flow}
\thanks{This work is supported by the Air Force Office of Scientific Research under award No. FA9550-22-1-0215  and a gift from Smale Institute (IR) 
}
\author[Burns]{Daniel Burns}
\address{Daniel Burns: Department of Mathematics, University of Michigan, Ann Arbor}
\email{dburns@umich.edu}
\author[Malajovich]{Gregorio Malajovich}
\address{Gregorio Malajovich: Departamento de Matem\'atica Aplicada, Instituto de Matem\'atica, Universidade Federal do Rio de Janeiro}
\email{gregorio@im.ufrj.br}
\urladdr{https://dma.im.ufrj.br/gregorio/}
\author[Pugh]{Charles Pugh}
\address{Charles Pugh: Department of Mathematics, University of California at Berkeley}
\email{pugh@math.berkeley.edu}
\author[Rajapakse]{Indika Rajapakse}
\address{Indika Rajapakse: 
Department of Computational Medicine \& Bioinformatics, Department of Mathematics and Department of Biomedical Engineering,
University of Michigan, Ann Arbor}
\email{indikar@umich.edu}
\urladdr{https://rajapakse.lab.medicine.umich.edu/}
\author[Smale]{Steve Smale}
\address{Steve Smale: Department of Mathematics, University of California at Berkeley}
\email{smale@berkeley.edu}
\date{\today}

\begin{document}  
\maketitle

\begin{abstract}
Kuramoto's differential equation describes a synchronization process between several harmonic oscillators. It has been used to model biological phenomena such as the synchronization of heart cells, the circadian rhythm, or brain waves. It is also used in power system control.

The simplest possible model assumes that all oscillators are identical and connected to each other with equal pairwise attraction. In this paper, we give a full geometric description of its global dynamics in terms of Morse theory and dynamical systems. Most of this description is stable in the sense that it is topologically preserved under small perturbations of the parameters.
\end{abstract}

\tableofcontents

\renewcommand{\glsglossarymark}[1]{}
\printnoidxglossary[type={symbols}, sort=def, title={List of Notations}]

\section{Introduction}
\label{s.intro}

The $m$-torus is  $\mathbb{T}^m$  where $\mathbb{T} = S^1$ is the unit circle.   It is an Abelian group with respect to componentwise addition modulo $2\pi $. We  write $\Theta = (\theta _1, \dots , \theta _m)$ for an element of $\mathbb{T}^m$.  
The standard \textbf{Kuramoto potential} $V: \mathbb{T}^m \rightarrow \mathbb{R}$ is 
$$
\gls{V} \;=\; \frac{1}{2}\sum_{\ell, k = 1}^m (1 -  \cos(\theta _\ell - \theta _k)) \;=\; \sum_{\ell,k = 1}^m  \sin^2((\theta _\ell - \theta _k)/2).
$$
 The \textbf{Kuramoto vector field} $K$ on $\mathbb{T}^m$ is the negative gradient of $V$.  
 $$
 \gls{K}\;=\; -\Big(\frac{\partial V}{\partial \theta _1}, \dots , \frac{\partial V}{\partial \theta _m}\Big)^T \;=\; \Big(\sum_k \sin(\theta _k - \theta _1), \dots , \sum_k \sin(\theta _k - \theta _m)\Big)^T.
 $$
$K$ generates the \textbf{Kuramoto flow} $\gls{phi}$.  Altogether, $V, K$, and $\varphi $ constitute the \textbf{Kuramoto dynamics}.

 \medskip
\noindent\textbf{Main classical references:} The Kuramoto model was introduced by \ocite{Kuramoto75}. See also
\ocite{Kuramoto},
\ocite{Strogatz},
\ocite{IzhikevichKuramoto}, 
\ocite{DorflerBullo},
\ocite{Kuramoto2026}
and references therein.
\medskip

\noindent\textbf{The Goal of this Paper:}
\emph{A geometrical description of   the   Kuramoto dynamics in the spirit of Morse theory and global dynamical systems.}

 \medskip
 
The first thing to notice is that $\varphi $ is actually a  flow of circles in $\mathbb{T}^m =\mathbb{R}^m \mod 2 \pi \mathbb Z^m$, not just a flow of points.  The torus $\mathbb{T}^m $ is foliated by \textbf{diagonals}
$$
\gls{DTheta} \;=\;  \{\Theta + \alpha\boldsymbol{1} : \alpha \in \mathbb R \mod 2 \pi \mathbb Z\}
$$
where  $\Theta  \in \mathbb{T}^m$, $\boldsymbol{1}$ is the $m$-tuple of ones, $(1,\dots , 1)$, and addition is performed in $\mathbb{T}^m$.  \emph{Note that each diagonal is a circle} in $\mathbb{T}^m$. Diagonals are either equal or disjoint.  All are parallel to the \textbf{main diagonal} $\gls{D} = D_{\boldsymbol{0}}$ where $\boldsymbol 0 = (0, \dots , 0)$.  


Since $\left< \boldsymbol{1}, K\right> = \sum_{k,\ell} \sin(\theta _k - \theta _{\ell}) = 0$, $\boldsymbol{1} \perp K$ .  Clearly, $K$ is constant along each diagonal, so $\varphi $ moves one diagonal to another and the motion is  normal to the diagonals. Thus, the set of fixed points of $\varphi $ is the union of the $\varphi $-invariant diagonals. That is, $\varphi $ fixes one point of $D_{\Theta }$ if and only if it fixes all of them.   
Figure~\ref{f.T2} shows the effect of $\varphi $ on the diagonals  when $m =2$.

\begin{figure}[ht]

\centering
\includegraphics[scale=.55]{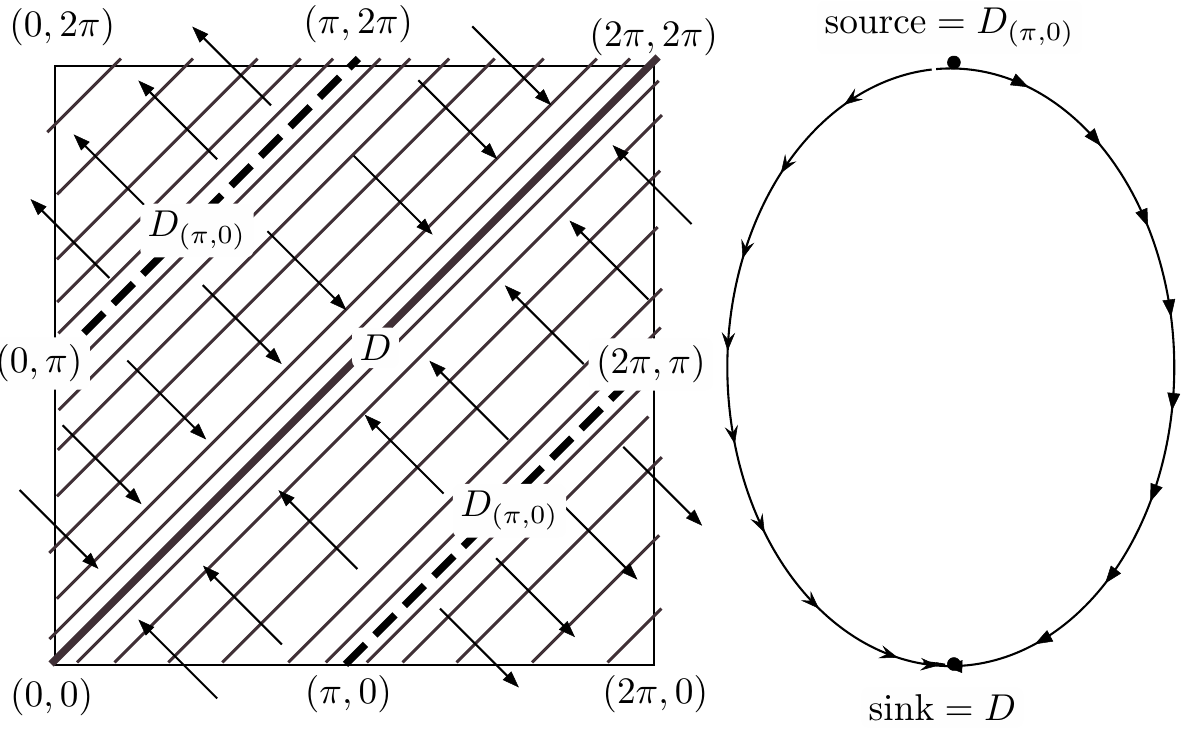}
\caption{When $m=2$, $\varphi $ fixes the diagonals $D_{(\pi ,0 )}$ and $D$.  Under $\varphi $, $D_{(\pi,0) }$  is a repeller and $D$ is an attractor.  This source/sink phenomenon becomes clearer when we represent diagonals as points on $S^1$, as in the righthand figure.}
\label{f.T2}
\end{figure}

We will use some standard hyperbolic concepts to go beyond $m=2$.  If an $m\times m$  matrix has $u$ eigenvalues with positive real part then its  \textbf{unstable index} is $u$.  If it has $s$ eigenvalues with negative real part, its \textbf{stable index} is $s$, and if $u + s =m$ then it is \textbf{hyperbolic}.  \emph{When we speak of \say{the} index, we refer to the unstable index.}

In the Kuramoto case, the derivative matrix $DK$ is constant along diagonals, so it makes sense to say a diagonal  has  index $u$ if its points do.  The matrix $DK$ is never hyperbolic since it always has the eigenvalue $0$ corresponding to the diagonal eigenvector $\bf{1}$.

\begin{definition*}
	A $\varphi $-invariant diagonal  is a \textbf{saddle diagonal} (resp. \textbf{sink diagonal}) if  $u + s = m-1$ and $s,u > 0$ (resp. $s=m-1$ and $u=0$). 
\end{definition*}

\begin{definition*}
The {\bf diagonal action} on $\mathbb T^m$ is the group of maps $T_{\alpha}: \mathbb T^m \rightarrow \mathbb T^m$ that send $\Theta$  to  $\Theta + \alpha \boldsymbol{1}$  where $\alpha \in \mathbb R  \mod  2 \pi\mathbb Z$.
\end{definition*}

From ODE theory, points $\Theta $ of a saddle diagonal have unstable and stable manifolds of dimensions $u$ and $s$.  They are denoted $W^u(\Theta )$ and $W^s(\Theta )$.  
Since $\varphi$ is the gradient flow of $V$, its fixed points (its {\bf equilibria}) are critical points of  $V$, and $V(\varphi_t(p))$ as a function of  $t$  is strictly decreasing for each non-critical point   $p$.

\begin{main*}  The Kuramoto flow $\varphi$ has the following properties:
\begin{itemize}
\item[(a)] Each trajectory $\varphi_t(p)$  has a unique $\alpha$-limit point  $\alpha(p)$ and a unique $\omega$-limit point  $\omega(p)$.
\item[(b)]  The minimum of $V$  is zero and occurs at the main diagonal  $D$.  The maximum of  $V$  occurs at a compact set  $V^{\max}$ and equals $m^2/2$.  The other critical values arise from saddle diagonals.   
There are precisely $\binom{m}{u}$ saddle diagonals of index $u$, $1 \le u < m/2$. Index $u$ saddles have $V$-value  $2u(m-u) < m/2$.

\item[(c)] $D$  is a sink in the sense that  $\omega(p) \in  D$  for all $p$  near  $D$, while  $V^{\max}$ is a source in the sense that $\alpha(p) \in V^{\max}$  for all $p$  near  $V^{\max}$.  The set of points $p$ with  $\alpha(p) \in V^{\max}$ and  $\omega(p)  \in  D$  is open-dense in $\mathbb T^m$.  That is, most Kumamoto  trajectories run directly from  source to sink.

\item[(d)]  $V^{\max}$ has an interesting topology.  It’s a product  $S^1 \times \mathcal V^{\max}$ where $\mathcal V^{\max}$ is a codimension 2 submanifold of $\mathbb T^{m-1}$,  which is smooth when $m$  is odd, and has only finitely many singularities when  $m$  is even.  See Sections \ref{s.high}--\ref{sec-imprints}.

\item[(e)]   Modulo diagonal action the saddle diagonals and their unstable manifolds form a CW complex of\, {\bf templates} on each of which  $\varphi$  is topologically equivalent to a Perfect Morse Flow, as described below.   
\end{itemize}
\end{main*}

The \textbf{Perfect Morse potential} $\gls{M} : \mathbb{T}^u \rightarrow  \mathbb{R}$ over-simplifies the Kuramoto potential.  It is defined as 
$$
M(\Theta ) \;=\;  \sum_{k=1}^u 1 - \cos(\theta _k).
$$
Its gradient is $(\sin\theta _1, \dots , \sin \theta _u)^T$, and  the gradient flow of $-M$ is the \textbf{Perfect Morse flow}.  It is the Cartesian product of $u$ source/sink flows on the circle,  for the solutions of $\theta ^{\prime} = -\sin \theta $ are $2\arctan(e^{-t + c})$ where $c$ is  a constant.

We can summarize the Main Theorem as follows.  $\mathbb{T}^m$ is divided into three regions or regimes:

\begin{itemize}

\item[(a)]
The {\bf low potential regime} where $0 \leq  V \leq L$ and $L$ is the second largest critical value of $V$. Specifically, $L=2u(m-u)$ when $u=\left \lceil \frac{m-1}{2} \right \rceil$.
\item[(b)]
  The {\bf mid potential regime} where $L < V < m^2/2 - \epsilon $,
\item[(c)]
The {\bf high potential regime} where $m^2/2 - \epsilon \leq V \leq m^2/2$.
\end{itemize}

We describe the Kuramoto dynamics fully in the low potential region; the dynamics are trivial in the mid potential region; and  the high potential region retracts to $V^{\textrm{max}}$ under the reverse $\varphi $-flow.  See Figure~\ref{f.FlowDiagram}.

\begin{figure}[ht]
\centering
\includegraphics[scale=.65]{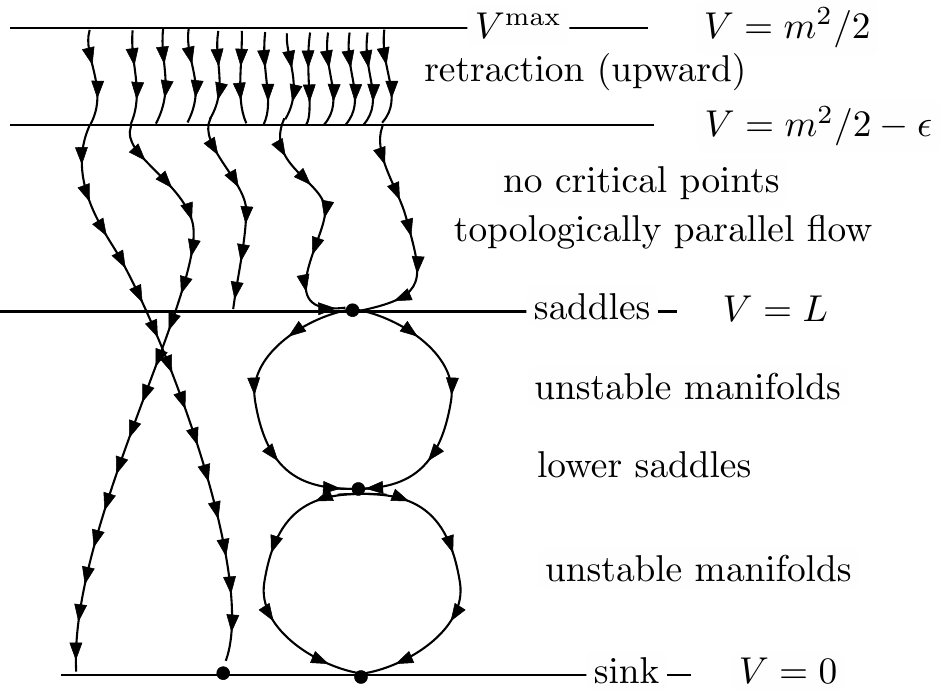}
\caption{The three potential regimes of the Kuramoto flow.}
\label{f.FlowDiagram}
\end{figure}

Finally, Figure~\ref{f.4handled} gives an example that indicates the complexity possible for $V^{\textrm{max}}$.
It's the product of $S^1$ with a 4-handled torus, and occurs when $m=5$. In higher dimension, its topology is more complicated and involves  the parity of $m$. It is described in Sections \ref{sec-cell}--\ref{sec-homology}.

\begin{figure}[ht]
\centering
\includegraphics[scale=.65]{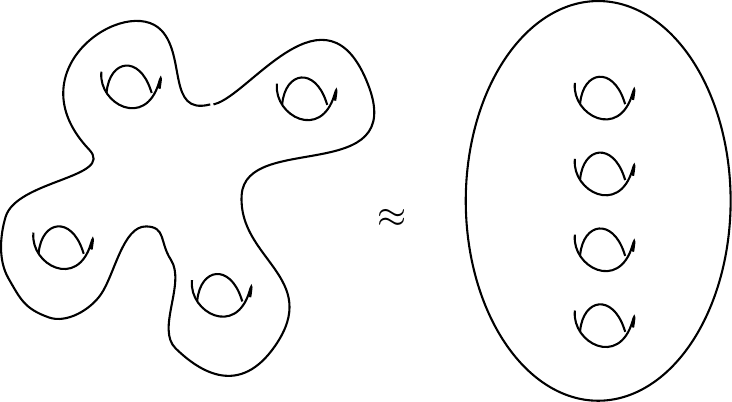}
\caption{When $m=5$, $V^{\textrm{max}}$ is 3-dimensional.  It's the product of the circle and a 4-handled torus.}
\label{f.4handled}
\end{figure}
\medskip

\paragraph{\bf Outline of the paper}
The Kuramoto vector field is invariant through diagonal action, and this suggests considering the Kuramoto vector field in the quotient of $\mathbb T^m$ by the diagonal action. This is explained in Section \ref{sec-quotient}, and is used in Section~\ref{s.picturing} to produce the phase portraits when $m=2$ and $m=3$. We also establish the {\bf equality principle}, to be used in the rest of the paper, that if two of the coordinates, say $\theta_1$ and $\theta_2$, are equal for some finite $t$, then $\theta_1\equiv \theta_2$ for all $t$, that is, on the whole orbit through that point. The equation for $V^{\max}$ in terms of the {\bf centroid} is deduced in Section~\ref{s.centroid}. The points of equilibrium outside $V^{\max}$ are characterized in Section \ref{s.antipodality}, and may be represented by an {\bf exemplar} (Section \ref{s.exemplars}). The Hessian at those exemplars is computed in Section~\ref{s.eigenstructure}. 

In Section~\ref{s.invariantsubtori} we investigate an interesting class of invariant subtori in quotient space, they are obtained as the projection of the principal subtori of $\mathbb T^m$. They are called {\bf templates}, each of them is the closure of the unstable manifold of a saddle point. In Section~\ref{s.templates} we describe the Kuramoto dynamics inside those templates. Then in Section~\ref{s.conjugacy} we show that the Kuramoto vector field, when restricted to a template, is topologically equivalent to the Perfect Morse potential. More examples of invariant subtori appear in Section~\ref{s.skew}. Section \ref{s.proof} is the proof
of the Main Theorem.

To complete the picture of the Kuramoto flow, we need to consider the high potential regime, that is the dynamics
near the source $V^{\max}$. In Section~\ref{s.high} we investigate the normal bundle of the projection $\mathcal V^{\max}$ of $V^{\max}$
to the quotient space. It turns out that for odd values of $m \ge 5$ this normal bundle is trivial, and the source is
a normally hyperbolic compact manifold. When $m$ is even, there is a finite number of singularities. Section~\ref{s.alpha}
is more technical, as we show that the $\alpha$-limit induces a retraction of the high potential region onto $\mathcal V^{\max}$. A cell decomposition of $\mathcal V^{\max}$ is produced in Section \ref{sec-cell}. Singularities of $\mathcal V^{\max}$ when $m$ is odd are described in Section \ref{sec-sing}. The homology of $V^{\max}$ for all values of $m$ is computed in section \ref{sec-homology}. The global description of the Kuramoto flow is completed in Section~\ref{sec-imprints} where we explain how the stable manifolds of the hyperbolic saddles connect to
the source set $\mathcal V^{\max}$, at least in dimension 5. In the conclusion we explain how
our picture is stable under small perturbations. 
This includes a more general model of the Kuramoto flow, 
for certain values of the parameters.
\medskip

\paragraph{\bf Acknowledgements.} The authors wish to thank Frederick Leve for the useful remarks and comments.

\section{The  Kuramoto Quotient}
\label{sec-quotient}

Standard Morse Theory  deals with critical points, not critical circles, so  we ``divide out by the diagonals.'' 
It's most natural to do so  algebraically.
The main diagonal $D$  is a subgroup of $\mathbb{T}^m$.  As a subgroup, it's  $D = \{\tau \boldsymbol{1}: \tau  \in \mathbb R \mod 2 \pi \mathbb Z\}$.
The quotient group $\gls{Q} = \mathbb{T}^m / D$ is diffeomorphic to the $(m-1)$-torus and consists of cosets. 
The coset containing $\Theta  \in \mathbb{T}^m$ is denoted $\gls{classTheta}$ and 
$$
[\Theta ] \;=\; \Theta +D \;=\;  \{\Theta + \tau  \boldsymbol{1} : \tau   \in S^1\}.
$$ 
Since $\varphi $ just permutes  the diagonals, it descends to a  flow $\psi $ on $\mathcal Q$

\begin{equation}\label{commdiag}
\begin{CD}
 \mathbb{T}^m @>\text{\normalsize
$\qquad 
\varphi _t\qquad$}>> 
\mathbb{T}^m
\\
@V\text{\normalsize$
q 
$}VV @VV\text{\normalsize$
q 
$}V
\\
\mathcal Q  @>\text{\normalsize$\qquad 
\psi _t  
\qquad$}>>
\mathcal Q
\end{CD}
\end{equation}
where $q : \Theta \mapsto [\Theta ]$ is the quotient map and $\gls{psi} 
= [\varphi _t(\Theta )]$.  We call $\psi $ the \textbf{quotient Kuramoto flow}.   Just as in  the $m=2$ case, this  captures the essential dynamics of $\varphi $.   

\medskip

\emph{Our overall strategy is to understand $\varphi $ by studying $\psi $.}

\medskip

  Algebraically, the coset $[\Theta ]$ is  an element of $\mathcal Q = \mathbb{T}^m/D$, but it can also be viewed as a subset of $\mathbb{T}^m$.  For example,  the  main diagonal  $D=D_{\boldsymbol 0} \subset \mathbb{T}^m$
is also the zero coset in $\mathcal Q$. The other elements of $\mathcal Q$ are, as subsets of $\mathbb{T}^m$, the other diagonals $D_{\Theta }$. 

\begin{notation*}
We will write $[\Theta ]$ for the element of $\mathcal Q$ and $D_{\Theta }$ for the corresponding diagonal, although too often we fail to distinguish them properly.  We also write $[m]$  to indicate the set $\{1, 2, \dots , m\}$, and $\abs{S}$ for the cardinality of a set $S$.  
\end{notation*}

The commutative diagram \eqref{commdiag} defines the Kuramoto quotient flow $\psi $   on the quotient $\mathcal Q$.  It's purely algebraic.  Here are two similar coordinate ways to picture it.

\begin{itemize}

	\item[(a)] \emph{Cube face coordinates.}\label{cube-face-coords}
$\mathbb{T}^m$ can be pictured as the $m$-cube $[0,2\pi ]^m$ with opposite faces identified.  Fix some $(m-1)$-dimensional face of $\mathbb{T}^m$, say $F_m = \{\Theta \in \mathbb{T}^m: \theta _m = 0 \}$.  The $(m-1)$-tuple $(\theta _1, \dots  , \theta _{m-1})$ gives $F_m$ a coordinate system, and since each diagonal  meets $F_m$ at a unique point, we get a coordinate system on $\mathcal Q$ in which to express $\psi  $.

\item[(b)]
\emph{Counterdiagonal coordinates.}
The \textbf{counterdiagonal} or {\bf antidiagonal} of $\mathbb{T}^m$  is the set
$$
A  \;=\;  \{ \Theta : \Theta \perp \textbf{1}\} = \{ \Theta : \sum \theta _j = 0\mod 2\pi \}.
$$
The counterdiagonal is an $(m-1)$-dimensional subtorus of $\mathbb{T}^m$ and is invariant under the Kuramoto flow since 
$$
\left< K(\Theta ), \textbf{1} \right> \;=\;  \sum _{j, k}  \sin (\theta _j - \theta _k) = 0
$$
implies that the Kuramoto gradient $K$ is everywhere tangent to $A$.  Even though $A$ is  an $(m-1)$-dimensional torus, it does not support a canonical coordinate system.  
\end{itemize}
Those two systems are related as follows. 
Let  $\mathrm e_1, \dots , \mathrm e_m$  be the standard basis vectors of $\mathbb{R}^m$.  The vectors $\mathrm f_j = \mathrm e_j - \frac{1}{m}\textbf{1}$, $ 1 \leq  j \leq  m$, span the counterdiagonal. They are not independent, since $\sum \mathrm f_j = 0$.
Any choice of $m-1$ of the $\mathrm f_j$ gives a coordinate basis for $A$ and hence a coordinate system for $\mathcal Q$.
For instance, if we pick the vectors $\mathrm f_j$ for $j \le m-1$, we obtain
cube face coordinates 
\[
	(\theta_1, \dots, \theta_{m-1}) \mapsto [\theta_1 \mathrm f_1, \dots,
	\theta_{m-1} \mathrm f_{m-1}].
\]

The Kuramoto potential on $\mathcal Q$ is well defined by commutativity of
 $$
\begin{CD}
 \mathbb{T}^m @>\text{\normalsize
$\qquad 
V\qquad$}>> 
\mathbb{R}
\\
@V\text{\normalsize$
q 
$}VV @VV\text{\normalsize$
=
$}V
\\
\mathcal Q  @>\text{\normalsize$\qquad 
	\gls{VQ}
\qquad$}>>
\mathbb{R}
\end{CD}
$$ 
since $V$ is constant along each diagonal.

\begin{theorem}
\label{t.grads}
The quotient flow $\psi $ is the gradient flow of $-V_{\mathcal Q}$.  \end{theorem}


\begin{proof}[\bf Proof]  
	Specifying the gradient of $V_{\mathcal Q}$ on $\mathcal Q$ requires a  ``Riemann structure'' on $\mathcal Q$, i.e., we need preferred  inner products on the tangent spaces of $\mathcal Q$.  We get them from $\mathbb{T}^m$ as follows.

The tangent bundle of $\mathbb{T}^m$ is trivial: it is a product
$
T(\mathbb{T}^m) =  \mathbb{T}^m \times  \mathbb{R}^m 
$.  In the same sense, we give it  a trivial Riemann structure by taking the standard inner product of $\mathbb{R}^m$ on each $T_{\Theta }(\mathbb{T}^m)$.
 Let $E$ be the subbundle of $T(\mathbb{T}^m)$ orthogonal  to the diagonals.  The tangent to the quotient map $q$, when restricted to $E$, gives a bundle  isomorphism   $Tq: E \rightarrow  T\mathcal Q$.  We choose the Riemann structure on $\mathcal Q$ to make this a bundle isometry.

	When we compute the gradient of $V$ at $\Theta $, we can use any convenient orthonormal coordinate system.  We choose one, say $x = (x_1, \dots , x_m)$, such  that $\partial /\partial x_m$ is tangent to the diagonal.  The $m^{\textrm{th}}$ component of the gradient is zero because $V$ is constant on each diagonal.  The calculation of the gradient of $V_{\mathcal Q}$ is exactly the same because we can use the orthonormal coordinate system $q(x_1, \dots , x_{m-1},0)$ on $\mathcal Q$  at $[\Theta ]$. 
\end{proof}

\begin{corollary*}

$q$ converts saddle diagonals to saddle points without changing the index.
\end{corollary*}


\begin{remark*}
It is tempting to say the whole proof of Theorem~\ref{t.grads}  is trivial.  However, the gradient of $V$, when computed relative to a cube face's coordinates, does not generate $\psi $.  That's because the face's  Riemann structure does not match  the quotient Riemann structure.
\end{remark*}


It may be useful to express the Kuramoto vector field in terms of cube face coordinates. The computation below will be needed in Section~\ref{s.conjugacy}.
\begin{proposition}\label{l.cfcoords} Consider the cube face coordinates $(\theta_1, \dots, \theta_{m-1})$ in $\mathcal Q$. The negative gradient of $V_{\mathcal Q}$ has coordinates 
\[
	-(\nabla V_{\mathcal Q})_i = \sum_{k=1}^{m-1} \sin (\theta_k - \theta_i)
	- \sin \theta_i - \sum_{k=1}^{m-1} \sin \theta_k .
\]
\end{proposition}

\begin{proof}
	Parameterize the counterdiagonal $A$ of $\mathbb R^m$ by
$X: \mathbb R^{m-1} \rightarrow A$,
	\[
		X: \begin{pmatrix} \theta_1 \\ \vdots \\ \theta_{m-1} \end{pmatrix}
			\mapsto \theta_1 \mathrm f_1 + \dots + 
			\theta_{m-1} \mathrm f_{m-1} = 
		\begin{pmatrix} \theta_1 - \bar \theta \\ \vdots \\ \theta_{m-1} - \bar \theta \\ - \bar \theta\end{pmatrix},
	\]
with $\bar \theta = \frac{1}{m} \sum_{j=1}^{m-1} \theta_j$.
	This map induces a parameterization $X: \mathbb T^{m-1} \rightarrow \mathcal Q$,
	which is clearly a diffeomorphism.
	The potential $V_{\mathcal Q}=V \circ X$ on the quotient set is precisely 
\[
	V_{\mathcal Q}(\theta_1, \dots,
\theta_{m-1})=
	\frac{1}{2}\left(\sum_{k,l=1}^{m-1} (1-\cos(\theta_k-\theta_l))\right)
	+ \sum_{k=1}^{m-1} (1-\cos (\theta_k)).
\]
Partial differentials are
\[
	\frac{\partial V_{\mathcal Q}}{\partial \theta_j} = - \sum_{k=1}^{m-1} \sin(\theta_k - \theta_j) + \sin(\theta_j)
.\]
Computing the gradient (a vector) requires 
	``raising the indices'' of the derivative (a covector or ``row vector'').
The Riemannian metric of $\mathcal Q$ is by construction the Riemannian metric of the counterdiagonal as a subset of $\mathbb R^m$. Hence, the Riemannian metric in cube face coordinates is its pull-back $g_{ij} = DX^T DX$. Let $\boldsymbol 1 = (1, \dots, 1)^T \in \mathbb R^m$ and
	$\boldsymbol u = (1, \dots, 1)^T \in \mathbb R^{m-1}$. The derivative of $X$ is 
\[
	DX = \begin{pmatrix} I \\
		0 \dots 0
	\end{pmatrix} - \frac{1}{m}\boldsymbol 1 \boldsymbol u^T
.
\]
Hence,
\[
g_{ij} = I - \frac{1}{m}\, \boldsymbol u \,\boldsymbol u^T
	\hspace{2em}
	\text{and}
	\hspace{2em}
g^{ij} = (g_{ij})^{-1} =  I + \boldsymbol u \boldsymbol u^T
\]

	The gradient of $V_{\mathcal Q}$ can now be computed, and the quotient Kuramoto vector
field is:
\[
	- (\nabla V_{\mathcal Q})_i = - \sum_j g^{ij} \frac{\partial V_{\mathcal Q}}{\partial \theta_j} 
	=
	\sum_{k=1}^{m-1} \sin (\theta_k - \theta_i)
	- \sin \theta_i - \sum_{k=1}^{m-1} \sin \theta_k .
\]
\end{proof}

\section{Picturing $\mathbb{T}^m$, $\varphi $, and $\psi $ when $m$ is small}

\label{s.picturing}

A common  mental picture of $\mathbb{T}^m$ is that when $m=2$, it looks like the surface of a donut.  The donut picture  is of little use for understanding  Kuramoto.  Instead, as in Section~\ref{s.intro}, we can picture $\mathbb{T}^2$ as the square   $[0, 2\pi ]^2$ with opposite edges identified, i.e., just interpret the square's coordinates modulo $2\pi $.

   Each diagonal is a line that meets the vertical edge $0 \times [0, 2\pi ]$ once.  It also meets the horizontal edge $[0, 2\pi ] \times  0$ once.  Because everything is modulo $2\pi $, there is only one vertical edge and only one horizontal edge.  Writing the Kuramoto vector field  as an ODE gives
\begin{equation*}
\begin{split}
\theta _1^{\prime} & \;=\; \sin (\theta _1 - \theta _1) + \sin(\theta _2 - \theta _1) \;=\; \sin (\theta _2 - \theta _1)
\\
\theta _2^{\prime} & \;=\; \sin (\theta _1 - \theta _2) + \sin(\theta _2 - \theta _2) \;=\; \sin(\theta _1 - \theta _2).
\end{split}
\end{equation*}
The main diagonal consists of equilibria, and so does its parallel through the point $(0, \pi ) = (\pi , 2\pi ) = (\pi , 0) = (2\pi , \pi )$.  (These equalities occur  in $[0,2\pi ]^2$ modulo $2\pi $.)  In fact, the main diagonal is a sink and the other diagonal is a source.  The $K$-trajectories are perpendicular to the diagonals.  See Figure~\ref{f.T2}.

\medskip

If   $m =3$, $\mathbb{T}^3$ is the  $3$-cube modulo $2\pi $.   
Its main diagonal is the line segment from $(0, 0, 0)$ to $(2\pi , 2\pi , 2\pi )$.  It is a circle, i.e., a sink-diagonal.   Three of the diagonals pass through the centerpoints of the three edges of the cube modulo $2\pi $.  They consist of equilibria and are saddle-diagonals.  As elements of the quotient, they are $[\pi, 0, 0], [0, \pi, 0], [0, 0, \pi ]$.  Finally, there are two diagonals through $(2\pi /3, 4\pi /3, 0)$ and $(4\pi /3, 2\pi /3, 0)$.  They are source-diagonals, and their union is $V^{\max}$.  In terms of the quotient flow, we have two source-points, three saddle-points, and one sink-point.  Diagonals in $\mathbb{T}^3$ become points in the quotient.  See Figure~\ref{f.KV2}.

\begin{figure}[ht]
\centering
\includegraphics[scale=.65]{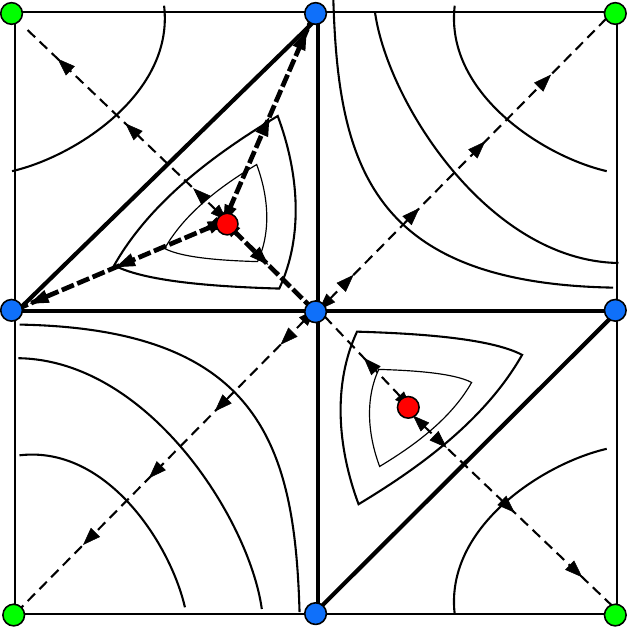}
\caption{The $\psi $-flow on $\mathbb{T}^2$, drawn on the square: red = source = $V^{\textrm{max}}$, blue = saddle, green = sink, black = constant $V$, dashed black with arrows = trajectory.}
\label{f.KV2}
\end{figure}

A further way to picture a Kuramoto trajectory $\varphi _t(\Theta) $ is to draw $m$ dots on the unit circle  indicating each angle $\theta _j$ of $\Theta $ as $e^{i\theta _j}$.  Think of  the dots  signaling to each other.    The signal between distant dots is weak, and so is the signal between  nearby  dots.  Antipodal dots or equal dots make no signal to each other.   
Each dot reads the signals from the other dots, adds its incoming signals up, and moves accordingly.

Here is a basic, useful fact about the Kuramoto flow, consistent with the signaling picture.

\begin{theorem} \textbf{\normalfont{\textbf{(Equality Principle)}}}
\label{t.equality}
If two angles of $\Theta \in \mathbb{T}^m$ are equal then they stay equal under $\varphi$.  If they are unequal then they stay unequal.
\end{theorem}

\begin{proof}[\bf Proof]
Suppose $\Theta  = (\theta _1, \dots , \theta _m)$ and  without loss of generality the equal angles are $\theta _1, \theta _2$.  Consider the manifold $M = \{\Theta: \theta_1 = \theta_2\}$. Its normal is $n = (1, -1, 0,\ldots, 0)$ and

\[ \langle n, - \nabla V \rangle\, =\, - \sum_{\ell} \sin(\theta_{\ell} - \theta_1) + \sum_{\ell} \sin(\theta_{\ell} - \theta_2).\]

But on $M$ the right-hand side of this equation is $\equiv 0$, so $-\nabla V$ is tangent to $M$ and the flow of any point in $M$ must stay in $M$.


\emph{The same principle applies to more than two angles, and to more than one set of equal angles.}
\end{proof} 
After all, if $\theta_1$ and $\theta_2$ are equal then they receive the same signals from the other angles, and this prevents divergence.


\section{The Centroid and $V^{\max}$}   
\label{s.centroid}

A useful tool for understanding   Kuramoto dynamics is the \textbf{centroid} of $\Theta  \in \mathbb{T}^m$.  It is the average of $e^{i\theta _1}, \dots , e^{i\theta _m}$, 
$$
\gls{Z} \;=\; \frac{1}{m}\sum _{k=1}^m e^{i \theta _k} \in \mathbb{C}. 
$$

The centroid is also known in the Kuramoto literature as the {\bf phase order parameter}.  
The following result is well known, see e.g. \ocite{DorflerBullo}*{Sec 4.2}.
\begin{theorem}
\label{t.centroid}
\begin{equation}
\label{eq.centroid}
V(\Theta )  = (1-\abs{Z}^2) \dfrac{m^2}{2}.
\end{equation}
\end{theorem}

\begin{proof}[\bf Proof]
We have $Z = re^{i\eta }$ for some $\eta $ and $r = \abs{Z}$.  Rotating  angles lets us assume $\eta  = 0$ without changing $\abs{Z}$ or $V$.  Express $Z$ in terms of the rotated angles,
$$
Z \;=\; \frac{1}{m} \sum e^{i\theta _k} \;=\; \frac{1}{m}\sum \cos(\theta _k) +  \frac{i}{m}\sum \sin(\theta _k)  \;=\;   r
$$
This implies that the sine sum of the  $\theta _{k}$ is zero, $\sum \sin (\theta _{k}) = 0$.

Consider the complex conjugates $ e^{-i\theta _{k}}$.  Their centroid is the complex conjugate of the centroid of the $e^{i\theta _k}$, which is therefore also  $r$.  This implies that 
\begin{equation*}
\begin{split}
r^2 &\;=\; \frac{1}{m^2}\Big(\sum_{\ell} e^{i\theta _{\ell}}\Big)\Big(\sum_k e^{-i \theta_k}\Big)\;=\; \frac{1}{m^2}\sum _{{\ell},k}e^{i(\theta _{\ell} - \theta _k)}
\\
&\;=\; \frac{1}{m^2}\Big( \sum_{{\ell},k} \cos(\theta _{\ell} - \theta _k) +
\sum_{{\ell},k} i \sin (\theta _{\ell} - \theta _k)\Big) \;=\; \frac{1}{m^2}\sum_{{\ell},k} \cos(\theta _{\ell} - \theta _k)
\end{split}
\end{equation*}
because $\sin(\theta _{\ell} - \theta _k) = - \sin (\theta _k - \theta _{\ell})$ makes the  sum of the sine terms vanish.  Thus $m^2r^2 = \sum \cos(\theta _{\ell} - \theta _k)$.

There are  $m^2$ terms in the sum
$$
V(\Theta ) \;=\; \frac{1}{2}\sum_{\ell, k = 1}^m (1 -  \cos(\theta _\ell - \theta _k))
$$
 so it's
$$
V(\Theta ) \;=\; \frac{m^2}{2} - \frac{1}{2}\sum_{{\ell},k} \cos(\theta _{\ell}-\theta _k) \;=\; \frac{m^2}{2} - \frac{m^2r^2}{2} \;=\; (1-r^2)\frac{m^2}{2}.
$$
\end{proof}

\begin{corollary}
\label{c.centroid} The minimum of $V$ is $0$, and occurs if and only if $|Z|=1$, that is, $\theta_i=\theta_j$ for all $i,j$.
The maximum of $V$   is  $m^2/2$, and occurs if and only if $\abs{Z}=0$, i.e., if and only if the centroid is the origin $O$ of $\mathbb{C}$. 
\end{corollary}

\begin{proof}[\bf Proof]
Obvious from  Theorem~\ref{t.centroid}.
\end{proof}

\begin{definition*}
$V^{\max}$ is the set of points at which $V$ has a maximum.
\end{definition*}
Every maximum of $V$ is a critical point of $V$.  (This is true for any potential, not just the Kuramoto potential, of course.)

\begin{theorem}
\label{t.Vmax}
	$V^{\max}$ is  nonempty and compact.  If $m$ is odd then it is a codimension $2$ submanifold of $\mathbb{T}^m$. 
	If $m$ is even then it is a singular codimension $2$ submanifold, with $\frac{1}{2} \binomial{m}{m/2}$ isolated
singularities.
\end{theorem}

\begin{proof}[\bf Proof]
Compactness of $\mathbb{T}^m$ implies that $V^{\max}$ is nonempty and compact.
Theorem~\ref{t.centroid} implies that $V^{\max}$ is the pre-image of the origin under the centroid map $Z : \mathbb{T}^m \rightarrow  \mathbb{C}$, and it    remains to show that the origin is a regular value of $Z$ when $m$ is odd.  Assume $Z(\Theta ) = O$.

We view $\mathbb{C}$ as $\mathbb{R}^2$, and write $mZ = (Z_1, Z_2) = (\sum \cos\theta _j, \sum \sin\theta _j)$.  The derivative of $mZ$ at $\Theta $ is the $2 \times  m$ matrix, 
$$
\begin{bmatrix}
-\sin\theta _1 &-\sin\theta _2 & \dots & -\sin\theta _m
\\
\cos\theta _1 & \cos\theta _2 & \dots & \cos\theta _m
\end{bmatrix}.
$$
 It is nonsingular at $\Theta $  if and only if the determinant of some $2\times  2$ submatrix 
 $$
\det \begin{bmatrix}
-\sin\theta _j & -\sin\theta _k
\\
\cos\theta _j &  \cos\theta _k 
\end{bmatrix}
\;=\; \sin\theta _k\cos\theta _j - \cos\theta _k\sin\theta _j \;=\; \sin(\theta _k - \theta _j)
$$
	is nonzero.  The only way all these submatrices are singular is that every angle difference  $\theta _k - \theta _j$ equals zero modulo $\pi $, which means there is an angle $\alpha $ such that each $\theta _j$ equals $\frac{\pi}{2}+\alpha $ or equals $-\frac{\pi}{2}+\alpha $. By diagonal action, assume $\alpha=0$. Now there are an odd number of angles $\theta _j$, so either there are more that equal $\pi/2 $, or there are more that equal $-\pi/2$.  Say there are more that equal $\pi/2$.  This shows that $Z(\Theta) $ is nearer to $i$ than to $-i$, so $Z(\Theta ) \not= O$, contrary to    $\Theta $ being in $V^{\max} $.  Hence, at least one of the $2 \times 2$ submatrices is nonsingular,   $O$ is a regular value $Z$, and its pre-image is a smooth submanifold of codimension $2$.

Now we consider the case where $m$ is even. The only possible singularities are the ones with half of the $\theta_j=\pi/2$ and the other half equal to $-\pi/2$.
\end{proof}

\begin{figure}
\centerline{
\includegraphics[width=0.5\textwidth]{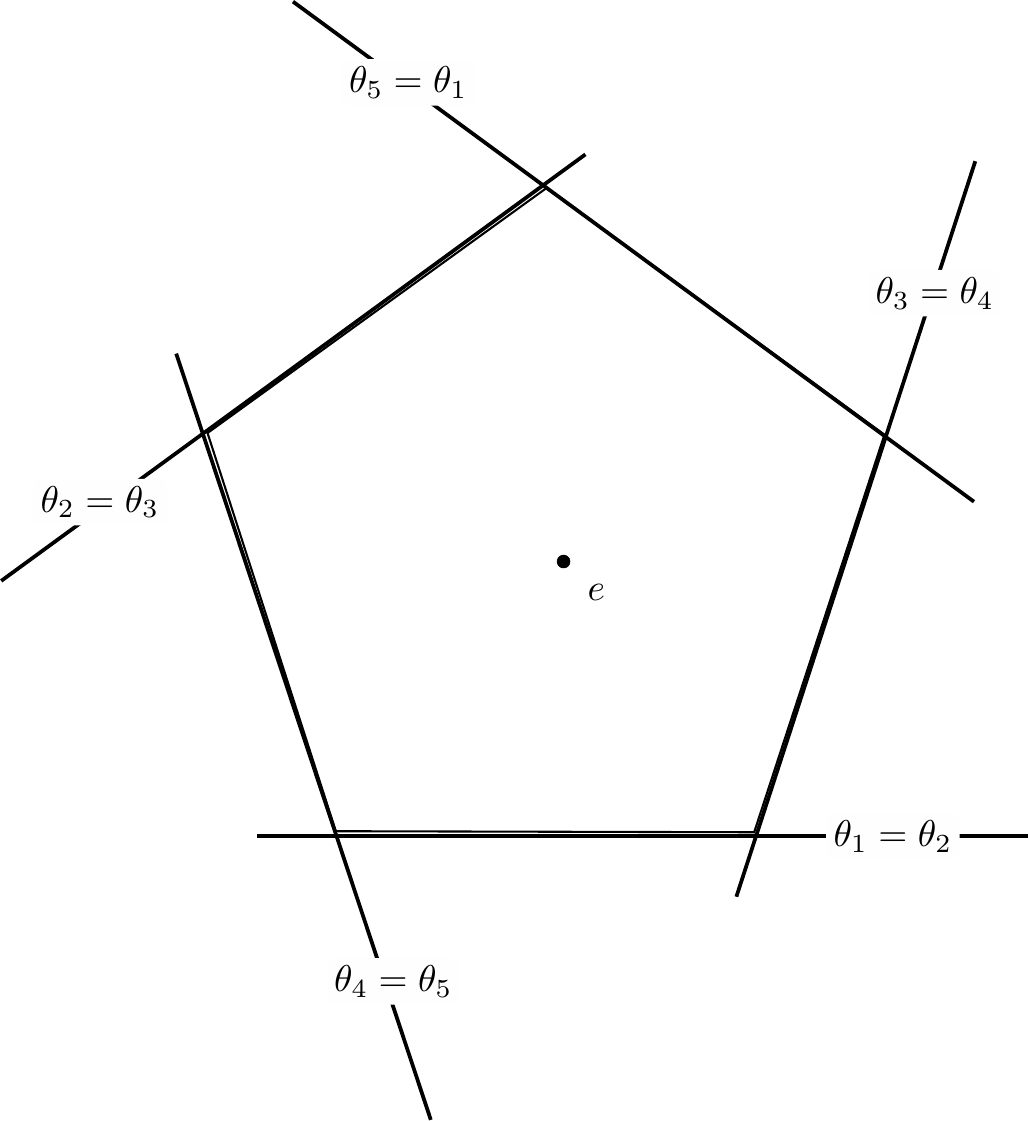}
}
\caption{When $m=5$, $\mathcal V^{\max}$ is the union of 24 pentagons like this one.
	\label{pentagon}}
\end{figure}

Here's a sketch of why $\mathcal{V}^{\textrm{max}} = \{ \Theta \in \mathbb{T}^m : \theta _1 = 0 \textrm{ and } Z(\Theta ) = O\}$ is a  4-handled torus when $m = 5$.  There are five $4$-planes in $\mathbb{T}^5$ defined by $\theta _1 = \theta _2$, \dots , $\theta _4 = \theta _5$, $\theta _5 = \theta _1$.  They intersect $\mathcal{V}^{\textrm{max}}$ in five curves, defining a pentagon centered at $e = (0, 2\pi /5, 4\pi /5, 6\pi /5, 8\pi /5)$ (Figure~\ref{pentagon}).   There are 4!=24 permutations of the non-zero angles $\theta_2 \dots, \theta_5$ defining $e$, so we get 24 pentagons that cover $\mathcal{V}^{\textrm{max}}$.  Each edge is shared by two adjacent pentagons, and each vertex is shared by four pentagons.  Thus the Euler characteristic of $\mathcal{V}^{\textrm{max}}$ is
$$
\chi  \;=\; F-E+V  \;=\; 24 - (24 \cdot 5) \div 2 \;+\;   (24 \cdot 5) \div 4 \;=\; -6,
$$
and modulo a few things like connectedness of $\mathcal{V}^{\textrm{max}}$ we see that it's a 4-handled torus.

\begin{question}
   Is the curvature of $\mathcal V^{\max}$ constant?
\end{question}

\section{Antipodality} 
\label{s.antipodality}

The critical points of $V$ are the points where $\operatorname{grad}V =0$.  They are exactly the equilibria of the system, and they are exactly the fixed points of $\varphi $.
$$
\textrm{zeros of $\operatorname{grad}V$ = critical points of $V$ = equilibria = fixed points of $\varphi $.}
$$

\begin{definition*}
$\Theta = (\theta _1, \dots , \theta _m) \in \mathbb{T}^m$ is \textbf{antipodal} if every difference $\theta _j - \theta _k$ equals zero modulo $\pi $.  (Note: $\pi $, not $2\pi $.)
\end{definition*}

\begin{theorem}
\label{t.antipodal}
All non-maximal critical points $\Theta $   are antipodal, and conversely, all antipodal $\Theta $ are critical. 
\end{theorem}

\begin{proof}[\bf Proof]
First assume $\Theta $ is antipodal.   The Kuramoto ODE is
$$
\theta _j{}^{\prime} \;=\;  \sum_k \sin(\theta _k - \theta _j) \qquad 1 \; \leq  \; j \; \leq  \; m,
$$
which makes  every term in the sum  zero, so $\Theta $ is fixed by $\varphi $ and critical for $V$.

Conversely, suppose that some critical   $\Theta \in \mathbb{T}^m \setminus V^{\max}$   has  non-antipodal angles.  We will derive a contradiction.

The centroid is $Z = re^{i\psi }$ with $0 \leq  r \leq  1$.  Since $\Theta  \notin V^{\max}$, $r > 0$.  On the other hand, if $r=1$ then $V = 0$ and all the angles of $\Theta $ are equal, which is an extreme form of antipodality.  Thus, we can assume $0 < r < 1$.  By   angle rotation  we can also assume that $Z$ is real, and
$$
Z \;=\;  \frac{1}{m}\sum \cos\theta _j \qquad \textrm{while} \qquad \sum \sin\theta _j = 0.
$$ 
Since the angles of $\Theta $ are not all antipodal, for at least one, say $\theta _1$,   $e^{i\theta _1}$ is not real, so we can assume that
$$
0 < r = \abs{Z} < 1 \qquad 0 < \theta _1 < \pi .
$$  

 Set $\Theta (s) = (\theta _1 + s, \theta _2, \dots , \theta _m)$.  It is a smooth curve through $\Theta $. Theorem~\ref{t.centroid}  states that $$
V(\Theta (s)) \;=\;  (1-r^2(s))\frac{m^2}{2}
$$
where 
$$
r^2(s) \;=\;  \Big(\frac{1}{m} \sum_j \cos(\theta _j(s)) \Big)^2 \;\; +\;\; \Big(\frac{1}{m} \sum_j \sin(\theta _j(s)) \Big)^2.
$$
The derivative with respect to $s$  of the first squared sum is   
$$
2\Big(\frac{1}{m} \sum_j \cos\theta _j(s) \Big)\Big(\frac{1}{m}\sum_j - \sin\theta _j(s)\frac{d\theta _j}{ds}\Big)
$$
which at $s=0$ is
$$
\frac{-2 Z \sin \theta _1}{m} \not= 0 .
$$
The derivative with respect to $s$ of the second squared sum is 
$$
2\Big(\frac{1}{m} \sum \sin\theta _j(s) \Big)\Big(\frac{1}{m}\sum \cos\theta _j(s)\frac{d\theta _j}{ds}\Big),
$$
which equals zero at $s=0$ since $\sum \sin\theta _j = 0$.  Hence, the derivative of  $r^2(s)$ with respect to $s$ at $s =0$ is nonzero.  This implies that  $V$ varies at first order in the $\theta _1$-direction  at $\Theta $, which contradicts the assumption that  $\Theta $ is  critical.  
\end{proof}

\section{Exemplars} 
\label{s.exemplars}

\begin{definition*}
If $\Theta $ is a non-maximal critical point then the diagonal  through $\Theta $ is a \textbf{critical diagonal}.  (All its points are critical.) 
An \textbf{exemplar} of  a critical   diagonal is a point $\Theta ^{\displaystyle \ast}$ in the critical diagonal whose angles are either zero or $\pi $, and which has fewer angles $\pi $ than angles zero.  It is a ``best example'' of points in the critical diagonal. 
\end{definition*}

Note that since $\Theta $ is non-maximal its centroid is nonzero, so an exemplar can't have equally many angles $\pi $ and zero.

\begin{theorem}
\label{t.exemplar} 
The critical set of $V$ consists of $V^{\max}$ and   finitely many disjoint critical diagonals, each of which has index $u < m/2$ and contains a unique exemplar.  
\end{theorem}

\begin{proof}[\bf Proof]  Let $\Theta $ be a non-maximal critical point of $V$.  
Theorem~\ref{t.antipodal} states that $\Theta $ has antipodal angles; i.e.,   there is some $\alpha  \in \mathbb R \mod 2 \pi $ such that each angle $\theta _i$ of $\Theta $ equals $\alpha $ or equals  $\alpha  + \pi $.  It is clear that angle rotation has no effect on antipodality, so diagonals consist either wholly of critical points or have none at all.  

As remarked above, since the centroid of $\Theta $ is nonzero, there can't be just as many $\theta _i$ that equal $\alpha $ as equal $\alpha + \pi $.   So suppose that there are more $\theta _i$ equal to $\alpha $ than equal to $\alpha  +\pi $.  Rotate by an angle of $-\alpha$ to get  an exemplar  $\Theta^{\displaystyle \ast} $.   Uniqueness is clear.  And so is the fact that there are only finitely many non-maximal  critical diagonals, since there are only $\binom{m}{d}$    $m$-tuples having $d$ entries $\pi $ and the rest zero. 
\end{proof}

\begin{corollary}
\label{c.psifixedpoints}
	The quotient Kuramoto flow has $2^{m-1}$ non-maximal fixed points when $m$ is odd, and $2^{m-1}-{\frac{1}{2}}\binom{m}{m/2} $ when $m$ is even.
\end{corollary}

\begin{proof}[\bf Proof] Use  $\binom{m}{d} = \binom{m}{m-d}$ and
 the Binomial Theorem $\sum_{d=0}^m \binom{m}{d}r^d = (1+r)^m$ with $r=1$.  When $m$ is even we delete the middle term $\binom{m}{m/2}$ because it counts $m$-tuples $\Theta $ with equally many angles $\pi $ and zero, all of which are in $V^{\max}$.  When $m$ is odd, there is no middle term $\binom{m}{m/2}$ to delete.  
  \end{proof}

\section{Critical Eigenstructure} 
\label{s.eigenstructure}

 Consider a critical diagonal and its exemplar $\Theta ^{\displaystyle \ast}$.  After permuting   angles, we can assume  $\Theta ^{\displaystyle \ast} = (\pi , \dots , \pi , 0, \dots , 0)$ where there are $d$ angles $\pi $,  $z$ angles $0$,  $d < z$, and $d + z = m$.  See Theorem~\ref{t.exemplar}. The local behavior of the Kuramoto dynamics at the critical diagonal is determined by the eigenvectors and eigenvalues of   the Hessian matrix  of $-V$ there.  We have
 $$
V(\Theta ) \;=\; \frac{1}{2}\sum_{\ell, k = 1}^m (1 -  \cos(\theta _\ell - \theta _k)) .
$$
The Hessian is the symmetric   $m \times  m$ matrix 
\begin{equation*}
\begin{split}
\frac{\partial ^2(-V)}{\partial \theta _i\partial \theta  _j} &\;=\;  \sum_{k = 1}^m \frac{\partial \sin (\theta _k - \theta _j)}{\partial \theta _i}
\\
& \;=\; \cos(\theta _i - \theta _j) - \delta _{ij}\sum_{k=1}^m  \cos(\theta _k - \theta _i).
\end{split}
\end{equation*}

\begin{remark*}
Since we are dealing with gradients, the Hessian of $-V$ and the   Jacobian of $K =  \operatorname{grad}(-V)$ are the same matrix.  On the other hand,  the Jacobian of the $K$-flow $ \varphi _t$ at a critical point is    $e^{tH}$.
\end{remark*}

To fill in the entries of the Hessian, we just need to check the cases $i=j\leq d$, $i < j \leq d$, $i\leq d <  j$, $d < i =j $, and $d < i <j$.  The others follow from symmetry of the Hessian.  Set 
$$
a \;=\; 1+(z-d)  \quad \textrm{and} \quad b = 1 - (z-d).
$$
Then the Hessian is  
$$
\gls{H} \;=\;  
\begin{bmatrix}
a  & 1 & \dots &1 & &-1 &\dots &\dots & \dots &-1
\\
1 &  a  &\dots & 1 &&-1 &\dots & \dots &\dots &-1
\\
\vdots &\ddots&\ddots&\vdots &&\vdots &&&&\vdots 
\\
1 & 1    & \dots & a & &-1  & \dots  & \dots &\dots   & -1 
\\
{\,}
\\
-1 & \dots    & \dots  & -1  & & b & 1 & 1 & \dots &1
\\
-1 & \dots    & \dots  & -1  & & 1 & b & 1 & \dots &1 
\\
-1   & \dots  & \dots  & -1   & & 1 &  1 & b &\dots &1
\\
\vdots  & &  &\vdots  & & \vdots&&\ddots&\ddots&\vdots 
\\
-1   & \dots  & \dots  & -1   & & 1 &  \dots & \dots & 1 &b 
\end{bmatrix}
$$
Let's check $H_{ii}$ for $i \leq d$.  It's
\begin{equation*}
\begin{split}
&\cos(\theta _i - \theta _i) - \sum_{k=1}^m\cos(\theta _k - \theta _i) 
\\
& \;=\;  1 - \sum_{k=1}^d \cos(\theta _k -\theta _i) - \sum_{k=d+1}^m \cos(\theta _k - \theta _i) 
\\
& \;=\; 1 - d\cos(\pi -\pi ) -z\cos(0 - \pi ) \;=\; 1 + (z-d) \;=\; a
\end{split}
\end{equation*}
as claimed.

\begin{theorem}
\label{t.eigenstructure}
The eigenstructure of $H$ at $\Theta^{\displaystyle \ast}  = (\pi , \dots , \pi , 0, \dots , 0)$ is as follows.
\begin{itemize}

\item[(a)]
The diagonal vector $(1, \dots , 1)^T$ is an eigenvector with eigenvalue $0$.
\item[(b)]
  The $dz$-vector $(z, \dots , z, -d, \dots , -d)^T$, with $d$ entries $z$ and $z$ entries $-d$ is an eigenvector with eigenvalue $m$.  
\item[(c)] 
The $d-1$ vectors 
$$
(1, -1, 0, \dots , 0)^T, \dots  , (1, 1, \dots , 1, -(d-1), 0, \dots , 0)^T
$$ are eigenvectors with eigenvalue $z-d$.  (Each has a tail of $z $ zeros.)
\item[(d)]
The $z-1$ vectors 
$$
(0, \dots  , 0, 1, -1, 0, \dots , 0)^T, \dots , (0, \dots , 0, 1, \dots 1, -(z-1))^T
$$ 
are eigenvectors with eigenvalue $d-z$. (Each has a head of $d$ zeros.)
\item[(e)] These eigenvectors are mutually orthogonal. 
\end{itemize}
\end{theorem} 

\begin{remark*}
If $\Theta $ is a permutation or a rotation  of $\Theta ^{\displaystyle \ast}$ then the corresponding eigenstructure   assertions hold at $\Theta $.
\end{remark*}

\begin{remark*}
We have written $(x_1, \dots , x_m)^T$ to indicate the transpose in order to make the formulas  look right.
\end{remark*}

\begin{remark*}  Consider the edge cases $d=0, 1$.  

\underline{Case 1.} $d=0$. Assertion (b) is true but trivial: The $dz$-vector is the zero vector.  (c) is vacuous since there are no vectors of this form.  (d) needs no special proof when $d=0$.

\underline{Case 2.}  $d = 1$. (b) and (d) need no special proof.  (c) is vacuous since there are no vectors of this form.   
\end{remark*}

\begin{remark*}
(c) means that $e^H$ has an eigenspace of dimension $d-1$  on which each vector is expanded by the factor $e^{z-d} > 1$, while (d) means that $e^H$ has an eigenspace of dimension $z-1$ on which each vector is contracted by the factor $e^{d-z} < 1$.   The expanding eigenspace is $E^u$, the contracting eigenspace is $E^s$, and the diagonal eigenspace is $E^c$. The  orthogonal splitting
$$
T_{\Theta ^{\displaystyle \ast}}(\mathbb{T}^m) \;=\;  E^u \oplus E^c  \oplus E^s
$$
	extends to a $T\varphi $-invariant orthogonal splitting over the critical diagonal  with exemplar $\Theta ^{\displaystyle \ast}$.  This demonstrates that  non-maximal critical diagonals are normally hyperbolic under the Kuramoto flow, so  it makes sense to refer to them as \textbf{saddle diagonals}. A reference for Normal Hyperbolicity is the book by \ocite{HirschPughShub}.
\end{remark*}

\begin{remark*}
	The dimension of the expanding space of $e^H$ is $d$ since it is spanned by the $d-1$  eigenvectors appearing in (d) and the $dz$-vector.  No other tangent vectors are expanded by $e^H$.  Therefore the  index is $d$.  (By definition \emph{the index is the dimension of the  unstable manifold.} Rotational invariance implies that we have the same index at all points of a saddle circle.)  For example, the main diagonal $D_{\boldsymbol 0}$ containing $\boldsymbol{0}= (0, \dots , 0)$  is a sink  for the Kuramoto flow. Its unstable manifold at each  $\Theta \in D_{\boldsymbol 0}$ is just the point $\Theta $ itself, and the  index is zero.
\end{remark*}

\begin{proof}[\bf Proof of  Theorem~\ref{t.eigenstructure}]
Each row of the matrix $H$ has sum zero.  Hence it kills the diagonal vector $(1, \dots  , 1)^T$, verifying (a).

The first row of $H$ times the $dz$-vector is
\begin{equation*}
\begin{split}
\begin{pmatrix}
a & 1 & \dots &1&-1&\dots &-1
\end{pmatrix}
\begin{pmatrix}
z
\\
z
\\
\vdots 
\\
z
\\
-d
\\
\vdots
\\
-d
\end{pmatrix} &=
az+(d-1)(z) +zd 
\\
&=\; (a + d-1 +d)(z) 
\\
&=\; (1 + z-d +d -1 +d)(z) \;=\; mz.
\end{split}
\end{equation*}
The same holds for the other rows and we see that $H$ multiplies the $dz$-vector by $m$, verifying (b).

Consider the first row of  $H$ times the column  $(1,-1, 0, \dots , 0)^T$,
\begin{equation*}
\begin{pmatrix}
a & 1 & \dots &1&-1&\dots &-1
\end{pmatrix}
\begin{pmatrix}
1
\\
-1
\\
\vdots 
\\
0
\\
0
\\
\vdots
\\
0
\end{pmatrix} 
\;=\; a-1 =  (z-d)(1).
\end{equation*}
The product of the second row times $(1, -1, 0 \dots , 0)^T$ is  
$
-1 + a  = z-d,
$ 
while the product of the 
third, fourth, and $m^{\textrm{th}}$ row times $(1, -1, 0, \dots , 0)^T$ is zero.  All together, $(1,-1, 0, \dots , 0)^T$ is an eigenvector with eigenvalue $z-d$, which verifies (c) for the first  eigenvector with eigenvalue $z-d$. Verification that the other vectors in (c) are eigenvectors with eigenvalue $z-d$ is similar, and so are (d), (e).
\end{proof}

\begin{Ups}
Theorem~\ref{t.eigenstructure} implies  that the set   of equilibria of the Kuramoto flow consists of   $V^{\max}$ and finitely many normally hyperbolic critical diagonals.  \emph{$\varphi $  is like a Morse flow, except that the critical points are diagonals and there is the presence of a set $V^{\max}$
which is smooth and normally hyperbolic for $m$ odd, and contains
isolated singularities for $m$ even.}
\end{Ups}


\begin{remark*}
The three eigenspaces are called unstable, center, and stable. At the sink exemplar $(0, \dots , 0)$ the unstable eigenspace is zero.   At points of $V^{\max}$ there is a similar splitting of $T(\mathbb{T}^m)$ in which the stable eigenspace is zero, while the center eigenspace is tangent to $V^{\max}$. 
\end{remark*}

\begin{remark*}
Here's   a  presentation of the  Kuramoto eigenvectors with an alternate notation.   We're looking at $I = \{1, \dots , d\} \subset [m]$ with $ d < m/2$.  All    should be column vectors, so mentally add a $T$ where needed. The unstable eigenvectors are 
\begin{equation*}
\begin{split}
(1,-1,0, \dots , 0&\;|\; 0, \dots , 0)
\\
(1,1,-2, 0, \dots , 0&\;|\;0, \dots , 0)
\\
(1,1,1, -3, 0, \dots , 0&\;|\;0, \dots , 0)
\\
\dots 
\\
(1, 1, \dots , 1, -(d-1)&\;|\;0, \dots , 0)
\end{split}
\end{equation*}

The vertical bar has $d$ entries to its left and $z = m-d$ entries to its right.  Each vector is an $m$-tuple,
\begin{equation*}
\begin{split}
 d \textrm{ entries} &\;|\;  z \textrm{ entries}.
\end{split}
\end{equation*}
There are  $d-1$ of these eigenvectors, each with eigenvalue  $m-2d > 0$.

Then there's the  $dz$-eigenvector
\begin{equation*}
\begin{split}
( z, \dots , z &\;|\; -d , \dots , -d )
\end{split}
\end{equation*}
with  eigenvalue $m$.  

Finally there are the $z-1$ stable eigenvectors
\begin{equation*}
\begin{split}
(0, \dots , 0) &\;|\; (1, -1, 0, \dots , 0)
\\
(0, \dots , 0) &\;|\; (1, 1, -2,0, \dots , 0)
\\
(0, \dots , 0) &\;|\; (1, 1, 1,-3,0, \dots , 0)
\\
\dots 
\\
(0, \dots , 0) &\;|\; (1, 1, \dots , 1, -(z-1)).
\end{split}
\end{equation*}
Each has eigenvalue $d-z < 0$.  
\end{remark*}

\section{Invariant Subtori} 
\label{s.invariantsubtori}

We want to show that in some sense, Kuramoto and  Perfect Morse are topologically equivalent.  
To some extent we can break down the Kuramoto flow using the subtori of $\mathbb{T}^m$.

\begin{definition*}
If $I \subset  [m]$  then the \textbf{principal $I$-torus} is
$$
	\gls{TI} \;=\; \{\Theta \in \mathbb{T}^m : \textrm{ all $\theta _j$ with $ j \notin I$ are zero}\}.$$
\end{definition*}
Since it's a product, the exponent $I$ in $\mathbb{T}^I$ is appropriate.  For $\Theta \in \mathbb{T}^I$, the components $\theta _i $ with $i \in I$ can vary freely, while the other components $\theta _j$ are constrained to equal zero.  By ``varying freely'' we mean that the $\theta _i$ with $i \in I$ can be any angles in $[0, 2\pi ]$ without  constraint.  Thus, the dimension of $\mathbb{T}^I$ equals the cardinality of $I$. 

Thus, $\mathbb{T}^I$ is a smooth, $\abs{I}$-dimensional   subtorus of $\mathbb{T}^m$, and the collection $\mathbb{T}^I$ with $I \subset [m]$ is a filtration of $\mathbb{T}^m$ in the sense that 
\begin{itemize}

\item[(a)]
 $J\subset I$ if and only if $\mathbb{T}^J \subset  \mathbb{T}^I$.

\item[(b)]
Each $\Theta \in \mathbb{T}^m$ belongs to a unique smallest $\mathbb{T}^I$.  

\item[(c)]  $\mathbb{T}^{[m]} = \mathbb{T}^m$ and $\mathbb{T}^{\emptyset} = (0, \dots , 0)$.
\end{itemize}

Moreover, $\mathbb T^I \cap \mathbb T^{J} = \mathbb T^{I \cap J}$.

\begin{definition*}
The (cellular)  \textbf{boundary} and (cellular)  \textbf {interior} of $\mathbb{T}^I$ are
\begin{equation*}
\begin{split}
\partial \mathbb{T}^I &\;=\;  \{\Theta  \in \mathbb{T}^I : \textrm{ at least one $ \theta _i$ with $i \in I$ is zero}\},
\\
\mathring{\mathbb{T}}^I &\;=\; \{\Theta  \in \mathbb{T}^I : \textrm{ all $\theta _i$ with $i \in I$ are nonzero}\}.
\end{split}
\end{equation*}
\end{definition*}

Clearly, $\mathbb{T}^I$ is $\abs{I}$-dimensional, $\mathbb{T}^I = \mathring{\mathbb{T}}^I \sqcup \partial \mathbb{T}^I$, and $\partial \mathbb{T}^I$ is the union of the $\mathbb{T}^J$ with $J \subsetneqq I$. 
Thinking  of $\mathbb{T}^m$ as the unit $m$-cube modulo $2\pi$, see Section \ref{s.intro}, $\mathring{\mathbb{T}}^I$ corresponds to its open $I$-face.  This gives $\mathbb{T}^m$ the structure of a \emph{cell complex} whose   (open) cells are the open faces, $\mathring{\mathbb{T}}^I$, and whose closed cells are the principal subtori $\mathbb{T}^I$.  
%

 The principal subtori $\mathbb{T}^I$ with $I \subsetneqq \mathbb{T}^m$  are not invariant under $\varphi $ since $\varphi _t$ immediately destroys the equality  $\theta _i = 0$ for $i \notin I$.  On the other hand, we can look at the \textbf{oblique subtorus} 
 $$
\operatorname{Diag}(\mathbb{T}^I )
 $$
 where we saturate $\mathbb{T}^I$ by including all the diagonals through its points.  $\operatorname{Diag}(\mathbb{T}^I)$ has dimension $ \abs{I} + 1$. Let $\gls{QI} \subset Q$ be the projection of $\operatorname{Diag}(\mathbb{T}^I)$ under the quotient map 
 $q$. 
The $\mathcal Q^I$ are also a filtration of $\mathcal Q$, but the property 
$\mathcal Q^{I} \cap \mathcal Q^J = \mathcal Q^{I \cap J}$ does not hold in general:
\[
	\mathcal Q^{\{3,4,5\}} \cap \mathcal Q^{\{1,2,3\}}= \{ (\alpha, \alpha, \theta_3, \beta, \beta) \in \mathbb T^5 \} \supsetneqq \mathcal Q^{\{3\}} .
\]

\begin{lemma}\label{l.intersection} If $\abs{I}, \abs{J} < m/2$, then 
$\mathcal Q^{I} \cap \mathcal Q^J = \mathcal Q^{I \cap J}$.
\end{lemma}
\begin{proof}
There is some $k \not \in I \cup J$, so $\theta_i = \theta_j = \theta_k$
for all $i \not \in I, j \not \in J$.
\end{proof}

\begin{remark*}
	Let $\mathring{\mathcal Q}^I$ be the interior of $\mathcal Q^I$. 
Each  $\mathring{\mathcal Q}^I$ consists of cosets $[\Theta ]$ 
such that if $i \in I$ and $j, k \notin I$ then 
$$
\theta _i \; \not= \; \theta _j \;=\;  \theta _k.
$$
	The Equality Principle implies that each cell $\mathring{\mathcal Q}^I$ is $\psi $-invariant.

\end{remark*}

\begin{definition*}
	The \textbf{lower half \gls{Kskeleton} of $\mathcal Q$} is the union of the $\mathcal Q$-cells of dimension $< m/2 $.  
\end{definition*}

The remarks above let us view $\mathcal K$ as a {\bf CW-complex}. The $k$-cells, $k<m/2$,
will be the $\mathcal Q^I$ for all $|I|<m/2$.
The definition of CW complex requires each $k$-cell of a CW-complex to be attached to the $k-1$-skeleton through a boundary map $f_I$ defined on $S^{k-1}$. For the $k$-cell $I=\{i_1, \dots, i_k\}$, define this boundary map as the composition $f_I = g_I \circ P$ where 
\[
	\defun{P}{S^{k-1}}{\partial [0,2\pi]^{k}}{x}{ \pi \left(\frac{1}{\|x\|_{\infty}} x + \boldsymbol{1}\right)} 
\]
is the trivial map from the ($k-1$)-sphere to the boundary of the $k$-cube,
and 
$g_I(y) = q(\Theta)$ where $\theta_{i_j}=y_j$ and the other coordinates vanish. By construction, at least one of the $\theta_{i_j}$ is equal to $0$ or $2\pi$ so $\Theta \in \cup_{J \subsetneqq I} \mathbb T^{J}$. It follows that $q(\Theta)$ belongs to the ($k-1$)-skeleton of $\mathcal K$.

\begin{theorem}
	\label{t.QsubI} Let $d = \lfloor (m-1)/2 \rfloor$.
	The quotient map $q : \Theta \mapsto [\Theta ]$ is a \say{\textbf{cellular diffeomorphism}} 
	from the $d$-skeleton of $ \mathbb{T}^m$ to $\mathcal K$,i.e.,
when restricted to $\mathbb{T}^I$, the quotient map $q$  is a diffeomorphism $\mathbb{T}^I \rightarrow  \mathcal Q^I$ where 
$$
	\mathcal Q^I \;=\; \{[\Theta] \in \mathcal Q : \textrm{ all $\theta _j$ with $ j \notin I$ are equal}\}.
$$
	The set $\mathcal Q^I$ is $\psi$-invariant.
\end{theorem}

\begin{proof}[\bf Proof]
Given $p \in \mathcal K$, Lemma~\ref{l.intersection} above implies that 
there is a unique minimal set
	$I \subset [m]$, $|I|<m/2$ so that $p \in \mathcal Q^I$. Furthermore, $p \in \mathring {\mathcal Q}^I=q(\mathring{\mathbb{T}}_I)$. Since the sets $\mathring {\mathbb T}^I$ are disjoint, $I$ is the unique subset, $|I|<m/2$, with that property. It is now sufficient to show that its restrictions $\mathbb T^I \rightarrow \mathcal Q^I$ are
diffeomorphisms.

	Let $I \subset [m]$ have cardinality $r \le d$. After permuting variables,
assume without loss of generality that $I=\{1, \dots, r\}$. Then $\mathbb T^I$
	is the set of $ \Theta$ with $\theta_{r+1}= \dots = \theta_m=0$
	and, assuming cube face coordinates with respect to the face $\theta_m=0$, 
\[
	q(\theta_1, \dots, \theta_r, 0, \dots, 0) = [(\theta_1, \dots, \theta_r, 0 \dots, 0)].
\]
There are $m-r$ tailing zeros in the left-hand term and $m-r-1$ in the right-hand term. The preimage by $q$ of the right-hand term is
	\[
		\{ (\theta_1+\alpha, \dots, \theta_r+\alpha, \alpha,
		\dots, \alpha): \alpha \in \mathbb R\mod 2\pi \mathbb Z \}\subset \mathbb T^m.
	\]
Thus, $q: \mathbb T^I \rightarrow \mathcal Q^I$ is one-to-one, onto and a trivial
diffeomorphism.
	Invariance of $\mathcal Q^I$ follows directly from the Equality Principle 
	(Theorem \ref{t.equality} p.\pageref{t.equality}).
\end{proof}

\begin{theorem}
\label{t.lowerhalf}
Every $\mathbb{T}^I$ with $\abs{I} < m/2$ is disjoint from $V^{\textrm{max}}$ and every $\mathbb{T}^I$ with $\abs{I} \ge m/2$ meets $V^{\textrm{max}}$.
\end{theorem}
\begin{proof} 
Suppose that $\abs{I}< m/2$. Then $\mathbb T^I$ contains an exemplar
	with $|I|$ coordinates equal to $\pi$. The centroid of Theorem~\ref{t.centroid} is $Z= 1-\frac{2 |I|}{m}$ and attains the global minimum for $|Z|$ on
	$\mathbb T^I$. Hence, this exemplar is a global maximum for $V$ on $\mathbb T^I$ with $V= (1-|Z|^2)\frac{m^2}{2}= 2 |I| (m-|I|) < m^2/2$. Hence
	$\mathbb T^I \cap V^{\max}=\emptyset$.

	If $\abs{I}=m/2$, pick $\Theta$ with $\theta_i=\pi$ for $i \in I$ and $\theta_i=0$ for $i \not \in I$. By construction, $\Theta \in \mathbb T^I$.
Its centroid is $Z=0$, so $\Theta \in \mathbb T^{I} \cap V^{\max}$.

	If $\abs{I} > m/2$, we distinguish two cases. If $\abs{I}$ is even, pick
	$\Theta \in \mathbb T^{I}$ with $\abs{I}/2$ coordinates equal to $\alpha$ and $|I|/2$ coordinates equal to $-\alpha$, 
	\[
		\cos(\alpha) = - \frac{m-\abs{I}}{\abs{I}}
	\]
	The remaining coordinates are $\theta_i=0$ for all $i \not \in I$.
	The choice of $\alpha$ guarantees that
	\[
		\sum \sin(\theta) = \sum \cos(\theta) = 0
	\]
	and hence $\Theta \in V^{\max}$.

	If $\abs{I} > m/2$ and $|I|$ odd, pick $\Theta \in \mathbb T^{I}$
	with $\frac{\abs{I}-1}{2}$ coordinates equal to $\alpha$, $\frac{\abs{I}-1}{2}$ coordinates equal to $-\alpha$ and one coordinate equal to $\pi$. In order to guarantee that $\sum \sin(\theta) = \sum \cos(\theta) = 0$, it is enough
	to choose $\alpha$ with
	\[
		\cos(\alpha) = - \frac{m-\abs{I}-1}{\abs{I}-1}.
	\]

\end{proof}

\section{Dynamics on $\mathcal Q^I$}
\label{s.templates}

Let $d<m/2$ and let $I \subset [m]$ have cardinality $d$. To the set $I$ we
associate the exemplar $\Theta_I=(\theta_1, \dots, \theta_m) \in \mathbb T^m$
with $\theta_i = \pi$ for $i \in I$ and $\theta_i=0$ for $i \not \in I$.
Let $\gls{pI} = q(\Theta_I) \in \mathcal Q^I \subset \mathcal Q$, its projection in the quotient.

\begin{theorem}
\label{t.QI}
\begin{enumerate}
		Let $I \subset [m]$ with $\abs{I}=d<m/2$. Then
\item
The restriction of $V_\mathcal{Q}$ to $\mathcal Q^I$ is a Morse function, with critical point 
		$p_J$ of index $\abs{J}$ 
for all $J \subset I$. 
\item
The Kuramoto flow $\psi$ on $\mathcal Q$ restricts to a flow, also denoted $\psi$,
on $\mathcal Q^I$. 
\item 
For every $J \subsetneqq I$ and $i \in I \setminus J$, there are  two disjoint {\bf saddle connections} $[\Theta(t)] = \psi_t([\Theta_0])$ with $\lim_{t \rightarrow -\infty} [\Theta(t)] = p_{J \cup \{i\}}$ and $\lim_{t \rightarrow \infty} [\Theta(t)] = p_J$.
\item The unstable manifold $W^u(p_I)$ in $\mathcal Q$ is precisely the cellular interior
	$\mathring {\mathcal Q}^I$, that is $\mathcal Q^I \setminus \cup_{J \subsetneqq I} \mathcal Q^J$.
\end{enumerate}
\end{theorem}

\begin{proof}[Proof]
\begin{enumerate}
	\item A critical point of the restriction of $V_{\mathcal Q}$ to $\mathcal Q^I$ is necessarily a critical diagonal of $V$ in $\mathbb T^m$, see Theorem~\ref{t.QsubI}. According to
Theorem~\ref{t.exemplar} each of them contains one and only one exemplar. The
exemplars that project onto $\mathcal Q^I$ are precisely the $\Theta_J$, $J \subset I$, 
with $\theta_j = \pi$ for $j \in J$ and $\theta_j=0$ otherwise. The projection
of $\Theta_J$ is precisely $p_J$. Since the critical diagonals are non-degenerate critical sets except along the diagonal direction, the
$p_J$ are non-degenerate critical points. 

Let $H$ be the Hessian of $-V$ at $p_J$.
Consider now the eigendecomposition of $H$ 
for $|J|= k \le d$. Without loss
of generality, we permute variables so that $J=[k]$ and $I=[d]$. 
Also, let $(\mathrm e_1, \dots, \mathrm e_m)$ be the
standard canonical basis of $\mathbb R^m$.

We decompose the tangent space $T_{p_J} \mathbb R^d = E^u \oplus E^s$,
with $E^u$ the span of $\mathrm e_1, \dots, \mathrm e_k$ and $E^s$ the
		span of $\mathrm e_{k+1}, \dots \mathrm e_d$. We claim that those are respectively the positive and negative eigenspace of $H$. 

The Hessian $H$ was computed in Theorem \ref{t.eigenstructure}. The restriction
of $H$ to $E^u$ is 
\[
\begin{bmatrix}
	a & 1 & \dots & 1 \\
	1 & a & \dots & 1 \\
	\vdots & \vdots & \ddots & \vdots \\
	1 & 1 & \dots & a
\end{bmatrix} 
\]
with $a = m - 2k + 1 > 1$. In particular, $E^u$  
is contained in the positive eigenspace of $H$.

If $k=d$ the space $E^s$ is zero-dimensional and there is no
negative eigenspace. If $k<d<m/2$, the restriction of $H$ at $E^s$ is
\[
\begin{bmatrix}
	b & 1 & \dots & 1 \\
	1 & b & \dots & 1 \\
	\vdots & \vdots & \ddots & \vdots \\
	1 & 1 & \dots & b
\end{bmatrix} 
\]
with $b=1-m+2k\le -1 -m + 2d < -1$. This matrix is negative definite.

From the considerations above, the index of $p_J$ is precisely the dimension
of the positive definite eigenspace, which is precisely $k$.
\item 
The equality principle implies that the flow is always tangent to $\mathcal Q^I$.
\item 
	Assume that $\abs{J}=\abs{I}-1$, and consider the flow in $\mathcal Q^I$. The only
		source is $p_I$ and $p_J$ is an index $\abs{I}-1$ saddle,
		so it has a stable eigenspace $E^s$ of dimension $1$ in $T_{p_J} \mathcal Q^I$.
Therefore, its stable manifold in $\mathcal Q^I$ 
		is $1$-dimensional, and $(W^s(p_J) \cap \mathcal Q^I) \setminus
		\{p_J\}$ is the union of two disjoint open curves, that are
		orbits of $\phi$. If $t \mapsto \phi_t(u)$ is one of those
		orbits, $V_{\mathcal Q}(u) > V_{\mathcal Q}(p_J)$. The limit $\lim_{t \rightarrow -\infty}
		V_{\mathcal Q}(\phi_t(u))$ exists. Since $V_{\mathcal Q}$ is bounded above and $\mathcal Q^I$
		is compact, the limit is a critical value. There is only
		one critical value $V_{\mathcal Q}(p_I) > V_{\mathcal Q}(p_J)$ so 
		$\lim_{t \rightarrow -\infty} \phi_t(u) = p_I$.

\item
We claim that $E^u$ is also the unstable eigenspace of $p_I$ in $T_{p_I}\mathcal Q$.
The dimension of $E^u$ is precisely the dimension of the positive eigenspace
of $p_I$ in $T_{p_I}\mathbb T^m$ so it is enough to prove that each of the
$\mathrm e_j$, $1 \le j \le k$, is the projection of a vector in 
the unstable eigenspace of $p_I$ with respect to the diagonal action.

Let $\mathbf 1$ denote the vector with all coordinates 1, and $\mathbf k$
the vector with $k$ coordinates $1$ and $m-k$ zero coordinates.
Recall from Theorem~\ref{t.eigenstructure} that the null eigenspace of $-H=-HV(p_J)$ is spanned by $\mathbf 1$. The positive eigenspace
by the vectors $\mathbf x$ with $\sum_{j=1}^k x_j = 0$ and by the 
`$dz$-vector' $m \mathbf k - k \mathbf 1$.

For $j \le k$, we just write 
\[
	\mathrm e_j = \mathbf x + \frac{1}{km} (m \mathbf k - k \mathbf 1) 
		+ \frac{1}{m} \mathbf 1
\]
with $\mathbf x = \mathrm e_j - \frac{1}{k} \mathbf k$.

\end{enumerate}
\end{proof}

\section{Topological conjugacy}
\label{s.conjugacy}

\begin{figure}
\centerline{
\includegraphics[width=0.5\textwidth,trim=1in 0in 1in 0in,clip]{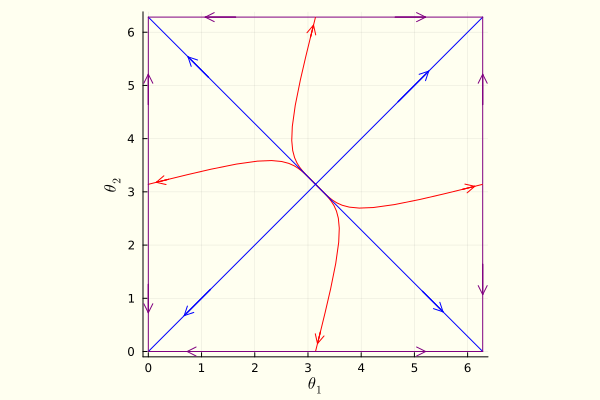}
\includegraphics[width=0.5\textwidth,trim=1in 0in 1in 0in,clip]{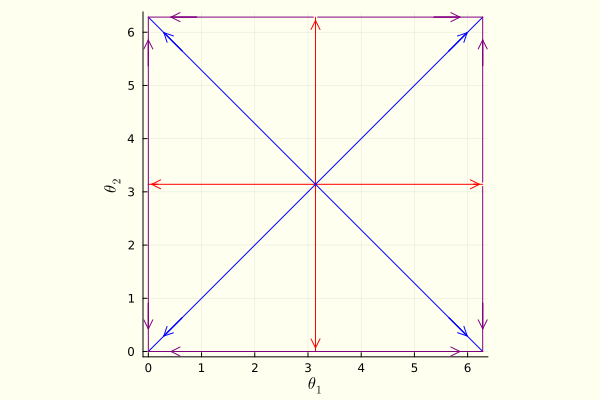}
}
\caption{Left: Quotient Kuramoto flow for $m=5$ on the template
$\theta_3=\theta_4=\theta_5=0$. Right: Perfect Morse flow in dimension
two. Those two flows are topologically equivalent.
	\label{fig-equivalence}}
\end{figure}

The Perfect Morse potential is defined on $\mathbb T^d$ as 
$M(\Theta) = \sum_{k=1}^d 1 - \cos(\theta_k)$. The vector field is
$F_0(\Theta) = -\nabla M(\Theta) = (-\sin(\theta_1), \dots, -\sin(\theta_d))^T$.
(See Figure~\ref{fig-equivalence}).

\begin{theorem}
	\label{t.equivalence} Let $I \subset [m]$ with $\abs{I}=d<m/2$. Then the flow
of $\psi$ in $\mathcal Q^I$ is topologically equivalent to the flow of the
Perfect Morse potential $M$ on $\mathbb T^d$.
\end{theorem}

Consider the flow $\psi$ `in coordinates', that is as the flow of an ODE on 
$\mathbb T^d$.  The Perfect Morse flow has exactly the same critical points, 
with the same index. In order to establish topological equivalence, it 
suffices to produce a continuous deformation taking one vector field 
into the other and preserving the critical points and indices.


An explicit expression for the Kuramoto vector field in cube face coordinates was
given in Proposition~\ref{l.cfcoords},
\[
	-(\nabla V_{\mathcal Q})_i = \sum_{k=1}^{m-1} \sin (\theta_k - \theta_i)
	- \sin \theta_i - \sum_{k=1}^{m-1} \sin \theta_k .
\]
When $I=[d]$ and $\theta_{d+1}= \dots = \theta_m=0$, this vector field becomes
\[
F_1(\Theta) := 
\begin{pmatrix}
	\sum_{k=1}^d \sin(\theta_k - \theta_1) - (m-d) \sin(\theta_1) 
	- \sum_{k=1}^d \sin(\theta_k) \\
	\\
	\vdots\\
	\sum_{k=1}^d \sin(\theta_k - \theta_d) - (m-d) \sin(\theta_d) 
	- \sum_{k=1}^d \sin(\theta_k) \\
\end{pmatrix}.
\]
Consider the convex linear combination $F_s(\Theta) = (1-s) F_0(\Theta) + s F_1(\Theta)$. The ODE $\dot \Theta = F_s(\Theta)$ is equal to the Perfect Morse flow when $s=0$, and to the reduced Kuramoto flow when $s=1$.

\begin{lemma}
For $0 \le s \le 1$, the fixed points of this equation are precisely the points $p_I$ from Theorem~\ref{t.QI}.
\end{lemma}
\begin{proof}[Proof] Indeed,
\[
F_s(\Theta) = s
\begin{pmatrix}
	\sum_{k=1}^d \sin(\theta_k - \theta_1) - \left(m-d+\frac{1-s}{s}\right) \sin(\theta_1) 
	- \sum_{k=1}^d \sin(\theta_k) \\
	\\
	\vdots\\
	\sum_{k=1}^d \sin(\theta_k - \theta_d) - \left(m-d+\frac{1-s}{s}\right) \sin(\theta_d)
	- \sum_{k=1}^d \sin(\theta_k) \\
\end{pmatrix}.
\]

Assume that $0 < s < 1$ and that $F_s(\Theta) = 0$. Expanding 
the $\sin (\theta_k - \theta_i)$ above, we retrieve that for each
$1 \le i \le d$,
	\begin{equation}\label{coordinate}
0 = \sum_{k=1}^d \sin(\theta_k) (\cos(\theta_i)-1)
	-\left(m-d+\frac{1-s}{s}+ \sum_{k=1}^d \cos(\theta_k) \right) \sin(\theta_i).
\end{equation}
	Write $S=\sum_{k=1}^d \sin(\theta_k)$ and
	$C=\sum_{k=1}^d \cos(\theta_k)$.
	After adding \eqref{coordinate} for $1 \le i \le d$, we obtain:
	\[
		0 = S(C-d) - \left(m-d+\frac{1-s}{s}+ C \right) S .
	\]
Two cases arise. If $S \ne 0$, The equation above means that
	$0=-m-\frac{1-s}{s}$ which is clearly impossible. Therefore
	$S=0$, and \eqref{coordinate} becomes
\[
0 =
	-\left(m-d+\frac{1-s}{s}+ C \right) \sin(\theta_i).
\]
Since $-d \le C \le d$, $m-d+C+\frac{1-s}{s} \ge m-2d + \frac{1-s}{s} >
\frac{1-s}{s} > 0$. Hence, $\sin(\theta_i)=0$ for all $i$, and
$\Theta$ has coordinates $0$ or $\pi$ modulo $2\pi$, 
	so $\Theta=p_I$ for some $I \subset [d]$.
\end{proof}

\begin{lemma} Assume $0<s<1$.
Let $W = \frac{1}{2} \sum_{l,k=1}^d (1-\cos(\theta_l - \theta_k))$. Then the function
$\Lambda(\Theta) =  W(\Theta) + \left(m-d+\frac{1-s}{s}\right) M(\Theta)$ is strictly decreasing along non-trivial orbits
of $F_s$. In particular, there is no nontrivial periodic orbit for $F_s$.
\end{lemma}
\begin{proof} We compute
\[
\frac{\partial \Lambda(\Theta)}{\partial \theta_i} = 
\left(C \sin(\theta_i) - S \cos(\theta_i) + \left(m-d+\frac{1-s}{s}\right) \sin(\theta_i) \right).
\]
It follows that
\[
\langle \nabla \Lambda(\Theta), F_s(\Theta) \rangle=
-s\|\nabla \Lambda(\Theta)\|^2 - s
\left(m-d-\frac{1-s}{s}\right) S^2 
\le 0,
\]    
and equality is possible only for all $\theta_i \equiv 0 \mod \pi$.
\end{proof}

\begin{lemma}
Let $I \subset [d]$. The fixed point $p_I$ of the ODE 
$\dot \Theta= F_s(\Theta)$ has the same unstable eigenspace 
regardless of $s$.
\end{lemma}

\begin{proof}
Without loss of generality, assume that $I=[k]$ for some $k \le d$.
Theorem~\ref{t.QI} implies that the  positive eigenspace
for $s=1$ is the subspace $E^u \subset R^d$ with vanishing last
$d-k$ coordinates. 
The derivative of $F_s$ at $p_I$ is
	\[
		DF_s(p_I) = s DF_1(p_I) + (1-s) \begin{bmatrix}
		I_k & \\ & -I_{d-k} \end{bmatrix}		 
	\]
which has the same positive eigenspace.
\end{proof}

\begin{lemma}
For all $s$, $W^u (p_J)$ is a $|J|$-dimensional submanifold of $\mathcal Q^J$, and its closure is $\mathcal Q^J$.
\end{lemma}

\begin{proof}
The equality principle (Theorem~\ref{t.equality}) extends to $F_s$: if two angles of $\Theta$ are equal then they stay equal under the flow. The proof is the same. The flow of $F_s$ restricts to a flow of $Q^J$, with unique source $p_J$. Since the unstable eigenspace of $p_J$ is $|J|$-dimensional, so is $W^u(p_J)$ and it is contained in $Q^J$. Since there is a finite number of
fixed points and no periodic orbit, the closure of the unstable manifold is all of $Q^J$.
\end{proof}

\begin{proof}[Proof of Theorem~\ref{t.equivalence}]
	The dynamical systems defined by the ODE $\dot \Theta = F_s(\Theta)$
	are Morse-Smale \cite{Smale1960}: Indeed, they 
have finitely many hyperbolic fixed points, and no nontrivial periodic orbit. 
It remains to check the transversality condition.
	Let $J, K \subset I$. $\dim W^s(p_J)=d-|J|$ while $\dim W^u(K)=|K|$.
If the invariant manifolds
	$W^s(p_J)$ and $W^u(p_K)$ intersect then $p_J \in \overline{W^u(p_K)}=\mathcal Q^K$ and
	$J \subset K$. 
    
    In that situation $\dim (W^s(p_J) \cap W^u(p_K)) =\dim (W^s(p_J) \cap \mathcal Q^K
	= |K|-|J|$. 
Dimensionwise,
\[
	\dim W^s(p_J) + \dim W^u(p_K) - \dim (W^s(p_J) \cap W^u(p_K))
	= d.
\]
	The condition above would establish transversality for affine spaces, but $W^s(p_J)$ is not affine. So assume that $r \in W^s(p_J) \cap  W^u(p_K)$ is such that 
\[
	\dim T_r W^s(p_J) + \dim T_r W^u(p_K) - \dim (T_rW^s(p_J) \cap T_rW^u(p_K))
	< d.
\]
	The same inequality holds at all points of the forward orbit of $r$ by the flow associated to $F_s$. Therefore, we can assume that $r$ belongs to a small enough neighborhood of $p_J$, contained in $\mathbb T^K = \overline{W^u(p_K)}$, with $\epsilon$ small enough. But $T_r W^{u}(p_K) = T_{p_J} \mathbb T^K$ and 
$T_r W^{s}(p_K) \rightarrow T_{p_J} W^{s}(p_K)$ which is the negative eigenspace of the Hessian at $p_J$. 
Therefore,
	$W^s(p_J)$ and $W^u(p_K)$ intersect transversally.

Those vector fields $F_s$ are therefore Morse-Smale and hence, structurally
stable \cite{Robinson74} for all $0 \le s \le 1$. 
	There is a finite cover $\{(a_j, b_j)\}_{j=1, \dots, r}$ of $[0,1]$
	so that for all $s, s' \in (a_j, b_j)$, $F_s$ and $F_{s'}$ are
	topologically equivalent. Assuming the $a_j$ and $b_j$ are increasing,
	there are $0= c_1 < \dots < c_r=1$, $a_{j+1} < c_j < b_{j}$
	with $F_{c_j}$ topologically equivalent to $F_{c_{j+1}}$.
	It follows that $F_0$ and $F_1$ are topologically equivalent.
\end{proof}

\section{Skew invariant subtori}
\label{s.skew}
In this section, we show by example that the topological conjugacy in Theorem~\ref{t.equivalence} is not a difeomorphism in general. We also explain why saddle connections in the Kuramoto flow are so hard to find numerically.

\begin{figure}
\centerline{
\includegraphics[width=0.5\textwidth,trim=1in 0in 1in 0in,clip]{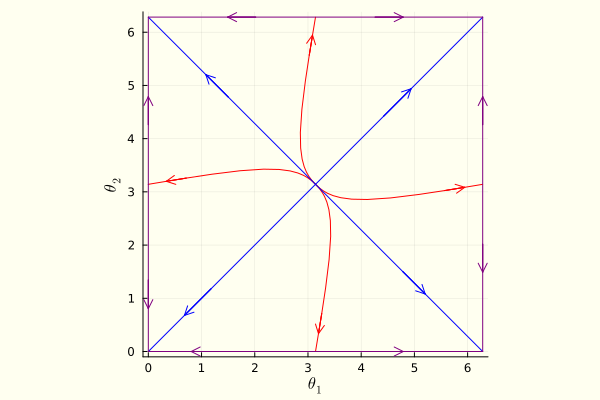}
\includegraphics[width=0.5\textwidth,trim=1in 0in 1in 0in,clip]{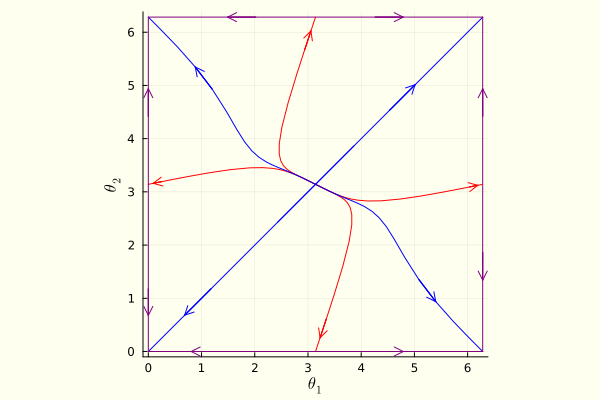}
}
	\vspace{0.5ex}
\centerline{
\includegraphics[width=0.5\textwidth,trim=1in 0in 1in 0in,clip]{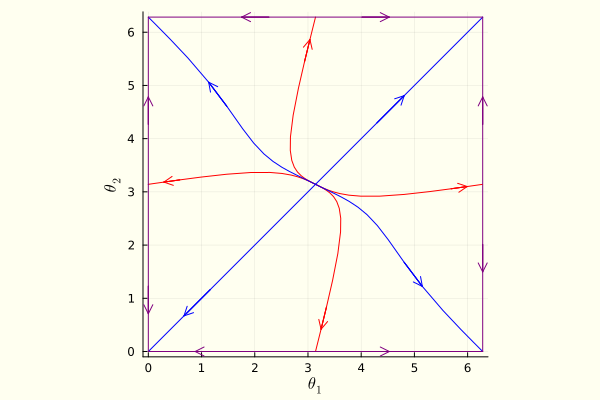}
\includegraphics[width=0.5\textwidth,trim=1in 0in 1in 0in,clip]{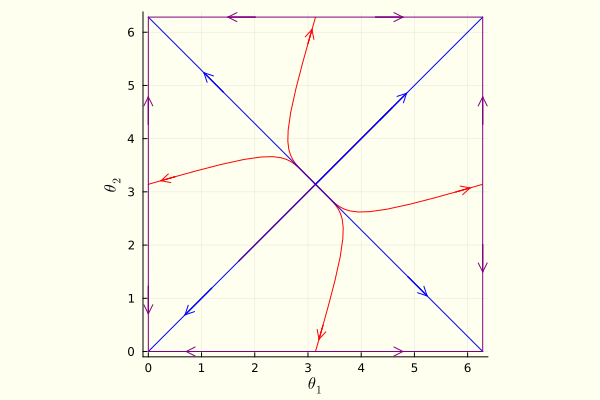}
}
	\caption{Kuramoto flow on skew subtori for partitions $(I,J,K)$
	for $(\abs{I},\abs{J},\abs{K})$ respectively: $(1,1,5)$ and $m=7$, $(1,2,4)$ and $m=7$, $(1,2,6)$ and $m=9$,
	$(2,2,5)$ and $m=9$. \label{fig-eq-2}}
\end{figure}

Let $\mathcal I=(I_1, \dots, I_r)$ be a partition of $[m]$. The equality
principle implies that the {\bf skew subtorus} $\mathcal Q^{\mathcal I}=\{ \Theta \in \mathbb T^m: i,j \in I_k \Rightarrow \theta_i = \theta_j \}$ is invariant with respect to Kuramoto flow $\phi$. 
The planar subtori of the previous two sections correspond to the particular case $\abs{I_1}=\dots=\abs{I_d}=1$, $\abs{I_{d+1}}=m-d>1$ projected into cube face coordinates where $\theta_j=0$ for $j \in I_{d+1}$. In this section we consider another important example.
By choosing one representative $i_k \in I_k$ and setting $\alpha_k = \theta_{i_k}$, we obtain an induced flow on $\mathbb T^r$. A more general formulation of
Theorem~\ref{t.equivalence} is:

\begin{theorem}\label{t.equiv2} Let $\mathcal I=(I_1, \dots, I_r)$ be a partition of $[m]$ with $\abs{I_r}\ge m/2$.  Then the flow
	of $\psi$ in $\mathcal Q^{\mathcal I}$ is topologically equivalent to the 
	flow of the
	Perfect Morse potential $M$ on $\mathbb T^{r-1}$.
\end{theorem}

\begin{example*}
	Assume that $m > 5$ and consider the partition
	$\mathcal I=(\{1\}, \{2\}, \{3,4, \dots, m\})$, see for instance
    Fig. \ref{fig-eq-2} top left. The induced
	vector field on $\mathbb T^2$ is 
\begin{eqnarray*}
	\frac{\partial \alpha}{\partial t} &=&
	\sin(\beta-\alpha)  -(m-1) \sin(\alpha) - \sin(\beta)
	\\
	\frac{\partial \beta}{\partial t} &=&
	- \sin(\beta-\alpha) - (m-1) \sin(\beta) - \sin(\alpha)
\end{eqnarray*}
        Both diagonals $\alpha=\beta \mod 2\pi$ and $\alpha=-\beta \mod 2\pi$ 
	are invariant. The derivative of the vector field at the source
    $(\alpha, \beta)=(\pi,\pi)$ has eigenvectors $\begin{pmatrix}1\\1\end{pmatrix}$ and 
    $\begin{pmatrix}1\\-1\end{pmatrix}$ with respective eigenvalues $1$ and $m-4 > 1$. It follows
    that orbits not in the main diagonal are all tangent to the diagonal $\alpha=-\beta \mod 2\pi$ at
    the source. Thus,
	  there
	cannot possibly be a \em{diffeomorphism} of $\mathbb T^2$ sending
	orbits of the induced vector field into orbits of the Morse vector
	field, and in this sense topological conjugacy in Theorem~\ref{t.equivalence} cannot be replaced by a difeomorphism in general.
\end{example*}

\section{Proof of the Main Theorem}\label{s.proof}

For clarity, the Main Theorem is proved in terms of the quotient flow on $\mathcal Q$. Item (a) follows from the fact that
$\psi$ is a smooth gradient flow on a compact manifold. Item (b) is Corollary~\ref{c.centroid}. From Theorem~\ref{t.antipodal}, $\psi$ admits $\binomial{m}{u}$ fixed points of the form $[\pi, \dots, \pi, 0, \dots, 0]$ with $u$ copies of $\pi$ for all $1 \le u<m/2$. Form Theorem~\ref{t.eigenstructure} those are precisely saddles of index $u$ and item (d) follows, using again Theorem~\ref{t.centroid} for the computation of $V$. Those are all the fixed points for $0<V<m^2/2$. Their stable and unstable manifolds have codimension $\ge 1$, hence item (d). Item (e) is Theorem~\ref{t.equivalence}

\section{The high potential regime}
\label{s.high}

The dynamics of the reduced Kuramoto flow may be decomposed into a low-potential regime
with $V_{\mathcal Q} \le (m^2-c)/2$, $c=1$ for $m$ odd and $c=4$ otherwise, containing all the
isolated fixed points; a trivial region with
$(m^2-c)/2 < V_{\mathcal Q} < m^2/2-\epsilon$, 
and a high potential regime $V_{\mathcal Q} \ge m^2/2-\epsilon$
containing a neighborhood of $\gls{Vmax}:= q(V^{\max})$, the projection of $V^{\max} \subset
\mathbb T^m$ into the quotient set $\mathcal Q$. In the next few sections we 
describe the set $\mathcal V^{\max}$, its neighborhood and how it connects to
the stable manifolds $W^s(p_I)$, $\abs{I}\le d$.
Let $\gls{Vsing}$ be the set of singular points of $V^{\max}$. If
$m$ is odd, $V^{\sing} = \emptyset$. If $m$ is even, 
$V^{\sing}$ is the union of diagonals containing a point with the
same number of coordinates equal to $\pi$ and to $0$. It follows that in this case\, 
$\mathcal V^{\sing}=q(V^{\sing})$ contains
$\frac{1}{2}\binomial{m}{m/2}$ isolated points.

\begin{theorem}\label{t.trivial}
	\begin{enumerate}
		\item The normal bundle of $\mathcal V^{\max} \setminus \mathcal V^{\sing}$ is trivial.
		\item $\mathcal V^{\max} \setminus \mathcal V^{\sing}$ is locally normally hyperbolic.
	\end{enumerate}
\end{theorem}
\begin{proof}[Proof of Theorem \ref{t.trivial}]
{\bf Part 1.} 
Through this argument, we will assume counter-diagonal coordinates. This
means that each time that we write coordinates $\theta_1, \dots,
\theta_m$ for some point $q(\Theta) \in \mathcal Q$, we also assume
that $\theta_1+\dots+\theta_m=0$.

From Theorem \ref{t.centroid},
	the set $\mathcal V^{\max}$ is defined by the equations
\begin{eqnarray*}
C(\Theta)&:=& 
\sum_{i=1}^m \cos(\theta_i) = 0 \\
S(\Theta)&:=&
\sum_{i=1}^m \sin(\theta_i) = 0 \\
\end{eqnarray*}
For every $\Theta$ in $\mathbb T^m$, write
\[
\Cos(\Theta) := 
\begin{pmatrix}
\cos(\theta_1) \\
\vdots \\
\cos(\theta_m) \\
\end{pmatrix}=
\nabla S(\Theta)
\hspace{1em}
\text{and}
\hspace{1em}
\Sin(\Theta) := 
\begin{pmatrix}
\sin(\theta_1) \\
\vdots \\
\sin(\theta_m) \\
\end{pmatrix} =
-\nabla C(\Theta)  
\]
and notice that for $\Theta \in V^{\max}$,
$\Cos(\Theta)$ and $\Sin(\Theta)$ are orthogonal to the diagonal.
By construction, they are also normal to the level surface
	$\mathcal V^{\max}$.
	It remains to prove that for all $\Theta \in V^{\max} \setminus V^{\sing}$,
$\Cos(\Theta)$ and $\mathbf \Sin(\Theta)$ are linearly
independent.

	Suppose for the sake of  contradiction that for some $\Theta \in V^{\max} \setminus V^{\sing}$, we have
\[
a \, \Cos(\Theta) + b \, \Sin(\Theta) = 0
\]
where $(a,b) \ne (0,0)$. We can normalize $a^2+b^2=1$ and write
$a=\sin(\alpha)$, $b=\cos(\alpha)$ so that
\[
\sin(\alpha) \, \Cos(\Theta) + \cos(\alpha) \, \Sin(\Theta) = 0
\]
For all $i$,
\[
\sin(\alpha+\theta_i)=
\sin(\alpha) \cos(\theta_i) + \cos(\alpha) \sin(\theta_i) = 0
\]
In particular, $\theta_i + \alpha \equiv 0 \mod \pi$. In particular
$\Theta$ is an exemplar so it belongs to $V^{\sing}$.

	Denote by $B=\mathcal V^{\max} \setminus \mathcal V^{\sing}$ the base space of the normal
bundle $\bundle{\mathbb R^2}{NB}{}{B}$. This
bundle admits
the global trivialization 
\[
\defun{\varphi}{B \times \mathbb R^2}{NB}
{(q(\Theta), \mathbf u)}{(q(\Theta), u_1 \, \Cos(\Theta) + u_2 \, \Sin(\Theta)).}
\]
{\bf Part 2:}
Under the notations above, $-\nabla V = S \Cos - C \Sin$ and hence the Hessian is
\begin{equation}\label{eq-Hessian}
	H= \Cos \Cos^T + \Sin \Sin^T - S\, \mathrm{diag}(\Sin) - C\, \mathrm{diag}(\Cos)
.
\end{equation}
After specialization to $\Theta \in V^{\max} \setminus V^{\sing}$,
$C=S=0$ and $\Cos, \Sin$ are linearly independent. Hence,
\[
H = \Cos \Cos^T + \Sin \Sin^T.
\]
The Hessian is a symmetric, rank 2 matrix as expected. To check positive-definiteness, it suffices to write $H$ in the (non-orthonormal) basis $(\Cos,\Sin)$. 
\[
	H_{(\Cos,\Sin)} = \begin{pmatrix} \Cos^T \Cos & \Cos^T\Sin \\
		        \Sin^T\Cos & \Sin^T \Sin
	\end{pmatrix}
\]
We recovered a Gramian matrix, that factors out as
\[
	H_{(\Cos,\Sin)} = \begin{pmatrix} \Cos  & \Sin \end{pmatrix}^T \begin{pmatrix} \Cos  & \Sin \end{pmatrix}
\]
and since $\Cos$ and $\Sin$ are linearly independent, the kernel vanishes and $H_{(\Cos,\Sin)}$ is positive definite. 
\end{proof}

The $\bf alpha$ and $\bf omega$ limit sets of $r$ by a flow $\psi_t$ are usually
defined as the set of accumulation points of $\psi_t(r)$ when $t \rightarrow -\infty$, resp. $+\infty$. Since in this paper $\psi_t$ is the flow of a Gradient vector field, there is more that can be said of the limit points.

\begin{theorem}\label{t.limits}
	For all $r \in \mathcal Q$, $\alpha(r)$ and $\omega(r)$ are points.
\end{theorem}

This allows to write the {\bf alpha} and the {\bf omega} limits as ordinary limits.
\[
	\alpha(r) = \lim_{t \rightarrow -\infty} \psi_t(r)
	\hspace{1em} 
	\text{and}
	\hspace{1em} 
	\omega(r) = \lim_{t \rightarrow \infty} \psi_t(r)
.
\]

\begin{proof}
	{\bf Part 1}. The quotient space $\mathcal Q$ is compact, and 
	$V(\psi_t)$ is strictly decreasing except on fixed points.
	Therefore, the omega-limit (resp alpha-limit)
	of any point $r \in \mathcal Q$ is necessarily
	nonempty, compact and  connected. We claim it
	consists entirely of fixed points. 
Indeed, suppose that $y \in \omega(r)$ is not a fixed point.
By definition, there is $(t_i)_{i \in \mathbb N}$ with
$\lim \psi_{t_i}(r) = y$ and we can assume that $t_i < t_i+1 < t_{i+1}$.
	This implies that $V(\psi_{t_i}(r)) > V(\psi_{t_i + 1}(r)) > V(\psi_{t_{i+1} }(r))$. Passing to limits,
	\[
		\lim_{i \rightarrow \infty} V(\psi_{t_i}(r))
		=
		\lim_{i \rightarrow \infty} V(\psi_{t_{i+1} }(r)) = V(y)
	\]
while
	\[
		\lim_{i \rightarrow \infty} V(\psi_{t_i + 1}(r)) = V( \psi_1(y)) = V(y),
	\]
contradiction.

	{\bf Part 2:} The fixed points in $\mathcal Q$ are the points $s_I$ for $|I|<m/2$ and $\mathcal V^{\max}$. If $r \in \mathcal V^{\max}$, then $\alpha(r) = \omega(r) = q$ so we assume $r \not \in \mathcal V^{\max}$. Necessarily, $\omega(r)$ contains one of the $p_I$, and those points are isolated. Thus the omega limit is always a point.

	The alpha-limit of $r$ can contain a point in $\mathcal V^{\max} \setminus \mathcal V^{\sing}$. Since $\mathcal V^{\max} \setminus \mathcal V^{\sing}$ is a normally hyperbolic repeller, that point is the limit when $t \rightarrow -\infty$. Suppose that the alpha limit of $r$ is not in $\mathcal V^{\max} \setminus \mathcal V^{\sing}$. Then it must contain either one of the $p_I$ or a point in $\mathcal V^{\sing}$. Those are all isolated, so the alpha limit is always one point. 
\end{proof}


\begin{theorem}\label{t.limits2}
	\begin{enumerate}
		\item If $\alpha(p) \in \mathcal V^{\max} \setminus \mathcal V^{\sing}$, the
			$\psi_t(p)$ is normal to $\mathcal V^{\max}$ as $t \rightarrow -\infty$.
        \item
		Let $d=m/2$. All the points $p=q(\theta)$ with exactly $d$ coordinates $\theta_i=0$ have limit $\alpha(p) \in \mathcal V^{\sing}$. 
	The orbits $\psi_t(p)$ are orthogonal to the vector $\Sin(\Theta)$, 
	$\sin(\theta_i)=1$ for $i \in I$ and $\sin(\theta_i)=-1$ for $i \not \in I$.
	\end{enumerate}
\end{theorem}

\begin{proof}[Proof of Theorem~\ref{t.limits2}]

{\bf Part 1.} 
Because of normal hyperbolicity, the unstable manifold
$W^u(q(\Theta))$ is tangent to the normal bundle.

{\bf Part 2.} 
The equality principle guarantees that $p_I$ is the only possible
$\alpha$-limit point.
\end{proof}

\section{Continuity of the $\alpha$-limit}
\label{s.alpha}

\begin{theorem}\label{t.alpha-Lipschitz} The $\alpha$-limit, when restricted to the border
	$\partial \mathcal V^{\high}$ of the high potential region
\[
	\gls{Vhigh} :=\{ [\Theta]\in \mathcal Q: V_{\mathcal Q}(\Theta) \ge \frac{m^2}{2}-\epsilon \}.
\]
	is Lipschitz continuous.
\end{theorem}

\begin{corollary}\label{c.alpha-continuous} Let $c=1$ for $m$ odd and $c=4$ for $m$ even.
	The $\alpha$-limit, when restricted to the mid potential region,
\[
	\gls{Vmid}:=\{ [\Theta]\in \mathcal Q: (m^2-c)/2 < V_{\mathcal Q}(\Theta) <\frac{m^2}{2}-\epsilon \}
\]
	is continuous.
\end{corollary} 

\begin{proof}
	Since there are no fixed points in $\mathcal V^{\mid}$, we can
	define a map $h$ that maps every $\Theta \in \mathcal V^{\mid}$ to the intersection of its
	backward Kuramoto orbit 
	with the level set $V=\frac{m^2}{2}-\epsilon$. By continuity of the flow, the map $h$ 
	is continuous. Theorem~\ref{t.alpha-Lipschitz} implies that $\alpha$ is continuous on
	$\mathcal V^{\high}$. Since $\alpha = \alpha \circ h$ on $\mathcal V^{\mid}$, the $\alpha$-limit
	is the composition of continuous functions, hence it is continuous.
\end{proof}

\begin{corollary}\label{c.alpha-retracts}
	The $\alpha$-limit induces a retraction of the high potential region $\mathcal V^{\high}$ to 
	$\mathcal V^{\max}$.
\end{corollary} 

\begin{proof}
	It is enough to prove the continuity of $\alpha$. 
	The argument of the previous Corollary still holds at any $p \in \mathcal V^{\high} \setminus \mathcal V^{max}$. We are left with the continuity at $\mathcal V^{\max}$. Suppose first that $v \in \mathcal V^{\max} \setminus \mathcal V^{\sing}$.
In particular $\alpha(v)=v$. Let $B$ be an open ball around $v$, and we assume that $\bar B$ is disjoint to $\mathcal V^{\sing}$.
The ball $B$ is a (noncompact) normally hyperbolic invariant manifold for the Kuramoto vector field. 
Indeed, its tangent bundle splits at each $b \in \bar B$ as $T_b{\mathcal Q}=T_b \mathcal V^{\max} \oplus E_b^u$ with $E_b^u$ two-dimensional.  The derivative of the time-$t$ flow $\psi_t$ is expanding at $E_b^u$ and null on $T_b \mathcal V^{\max}$.

	By Normal Hyperbolicity, there is a local fibration of $Q$ by unstable manifolds of the flow, covering $B$.
	This local fibration is actually a disk bundle, see for instance Theorem~1 in \ocite{ElderingKvalheimRevzen}. 
	If we restrict this fibration to $\mathcal V^{\high}$, we obtain an invariant 
	fibration $\alpha: E \rightarrow B$, $E \subset \mathcal V^{\high}$ with $E$ open containing $v$. 
	Since this holds for arbitrarily small balls $B$ around $v$, the map $\alpha$ is continuous at $v$.

	Finally, suppose that $w \in \mathcal V^{\sing}$. Let $B_0$ be a small open 
	neighborhood of $w \in \mathcal V^{\sing}$, $\bar B_0 \cap \mathcal V^{\sing}=\{w\}$.
	Assume by contradiction that there is $p_i \rightarrow w$ with $\alpha(p_i) \not \in B_0$, $p_i \in \mathcal V^{\high}$.
	By compactness, the $\alpha(p_i)$ admit an accumulation point $v$ in $\mathcal V^{\max} \setminus B_0$.
	If $v \not \in \mathcal V^{\sing}$, then there is a neighborhood $B \ni v$ and a fiver bundle $\alpha: E \rightarrow B$
	as in the previous case, and $w \in B$, contradiction. 

	We are left with the possibility that $v, w \in \mathcal V^{\sing}$. Then they must be equal. For instance, assume that
	$w=[0,0,0,\pi,\pi,\pi]$ and say $v = [0,0,\pi,\pi,\pi,0]$. Assume coordinates $\theta_1\equiv 0$.
	Because $p_i \rightarrow w$, the coordinates of $p_i$
	must satisfy as $\theta_3(p_i) \rightarrow 0$ while, for $\alpha(p_i) \rightarrow w$, we need $\theta_3(p_i) \rightarrow \pi$.
	Thus, $v=w$ and $\alpha$ is also continuous near $\mathcal V^{\sing}$.

\end{proof}

The proof of Theorem~\ref{t.alpha-Lipschitz} requires an auxiliary vector field
and a sequence of Lemmas. The auxiliary vector field is
\[
	W(\Theta) = \frac{w(\Theta)}{\| \nabla V \|^2} \nabla V
\]
	with $w(\Theta) = \frac{m^2}{2}-V(\Theta)$. The orbits of the auxiliary vector field $W$ are the orbits of the Kuramoto vector field $K=-\nabla V$ with inverted direction.

\begin{lemma}\label{lemma-W1} If $\dot{\Theta}(t) = W(\Theta(t))$, then $w(\Theta(t))= w(\Theta(0))\, e^{-t}$. In particular, for
	$\Theta(0) \in V^{\high}$,
	$w(\Theta(t)) \le \epsilon \, e^{-t}$.
\end{lemma}
		\begin{proof}
\[
	\frac{\partial w(\Theta(t))}{\partial t} =
	-\langle \nabla V(\Theta(t)), W(\Theta(t)) \rangle=- w(\Theta(t)). 
\]
	It follows that $w(\Theta(t)) = w(\Theta(0))\, e^{-t} 
			\le \epsilon\, e^{-t} $.
\end{proof}

\begin{lemma}\label{lemma-W2}
	Let $\Theta \in \mathcal V^{\high}$ and $\|\mathbf u\|=1$. Then
	the Hessian of $V$ satisfies
	$-m-\sqrt{2 w(\Theta)} \le  \mathbf u^T  HV(\Theta) \mathbf u \le \sqrt{2 w(\Theta)}$.

\end{lemma}

\begin{proof}
	We will use the notations of the proof of Theorem~\ref{t.trivial}. Recall that $H$ is minus the
	Hessian of $V$. Equation~\eqref{eq-Hessian} yields
\[
		H= \Cos \Cos^T + \Sin \Sin^T - S\, \mathrm{diag}(\Sin) - C\, \mathrm{diag}(\Cos)
.
\]
We can add a constant angle to all of the $\theta_i$ to obtain $S=0$ and $C\ge 0$. At this point $\nabla V = C \Sin$
and therefore
\[
\mathbf u^T H \mathbf u = \langle \Cos, \mathbf u \rangle^2 + \langle \Sin, \mathbf u\rangle^2
	- C \sum \cos(\theta_i) u_i^2 \ge -C \|\mathbf u\|^2
\]
We can bound $0 \le \langle \Cos, \mathbf u \rangle^2 + \langle \Sin, \mathbf u\rangle^2
\le (\|\Cos\|^2 + \|\Sin\|^2) \|\mathbf u\|^2 = m$.
	Since $w(\Theta) = \frac{m^2}{2}-V(\Theta)=\frac{1}{2} C^2$,
	\[
		-\sqrt{2 w(\Theta)} \le \mathbf u^T H \mathbf u \le m+\sqrt{2 w(\Theta)}
	\]
\end{proof}

\begin{lemma}\label{lemma-W3}
	Define $Y(\Theta):=\frac{w(\Theta)}{\|\nabla V(\Theta)\|^2}$. Let $\Theta(t)$ be a trajectory of the
	vector field $W$, that is $\dot \Theta(t) = W(\Theta(t))$ and let $y(t)=Y(\Theta(t))$.
	Then for $t \ge 0$, $y(t) \le \max\left (y(0), \frac{1}{2(m+\sqrt{2\epsilon})}\right)$.
\end{lemma}


\begin{proof} Differentiating $y(t)$,
	\begin{eqnarray*}
		\dot y(t) &=& - \frac{ \langle \nabla V(\Theta(t)), W(\Theta(t)) \rangle}{\| \nabla V(\Theta(t))\|^2}
		-2 \frac{w(\Theta(t)) \langle \nabla V(\Theta(t), HV(\Theta(t))\ W(\Theta(t))\rangle }{\| \nabla V(\Theta(t))\|^4}
		\\
		&=& - y(t) - 2 y(t)^2 \frac{ \langle \nabla V(\Theta(t)), HV(\Theta(t))\ \nabla V(\Theta(t))\rangle}{\| \nabla V(\Theta(t))\|^2}  .
	\end{eqnarray*}
	From Lemma~\ref{lemma-W2},
	\[
-\frac{ \langle \nabla V(\Theta(t)), HV(\Theta(t))\ \nabla V(\Theta(t))\rangle}{\| \nabla V(\Theta(t))\|^2}  
\le
	m+\sqrt{2 w(\Theta(t))} 
\]
and hence
	\[
\dot y(t) \le -y(t) + 2(m+\sqrt{2\epsilon}) y(t)^2
	\]
The solution of $y(t)$ can be bounded above by the solution of
	\[
\dot z(t) = -z(t) + 2(m+\sqrt{2\epsilon}) z(t)^2
	\]
which is of the form
	\[
		z(t) = \frac{1}{ 2(m+\sqrt{2\epsilon}) \pm e^{c \pm t} } 
	\]
	for $z(0) \ne  \frac{1}{2(m+\sqrt{2\epsilon})}$, and constant in case of equality.
In any case, for $t \ge 0$,
	\[
		y(t) \le \max\left(y(0), \frac{1}{ 2(m+\sqrt{2\epsilon}) }\right)
	\]
\end{proof}

\begin{lemma}\label{lemma-W4}
	Let $\gamma:[a,b] \rightarrow \mathcal V^{\high}$ be a parameterized curve of class $\mathcal C^2$ 
	with $w(\gamma(\tau))$ constant
and let $L_0=\int_a^b \left\| \gamma'(\tau) \right\| \dd \tau$
	be its length. Let $L(t)$ be the length of the curve obtained by the time $t$ flow of $W$ applied to $\gamma$. 
	Suppose that there is a constant $y^T$ such that $Y(\gamma(\tau))\le y^T$ for all $\tau \in [a,b]$ and 
	$\frac{1}{2(m+\sqrt{2\epsilon})} \le y^T$.
	Then $L(t) \le e^{2\sqrt{2}y^T}\, L(0)$.
\end{lemma}


\begin{proof}
	Let $\Theta(t, \tau)$ satisfy $\frac{\partial}{\partial t} \Theta(t, \tau) = W(\Theta(t, \tau))$ with 
	$\Theta(0,\tau)=\gamma(\tau)$. 
	From Lemma~\ref{lemma-W1}, $w(\Theta(t,\tau))$ is constant for any fixed value of $t$
	and is bounded above by $\epsilon \, e^{-t}$. In the calculations below, $\Theta=\Theta(t,\tau)$ and
	$\Theta'$ stands for $\frac{\partial}{\partial \tau} \Theta(t,\tau)$. 
	Differentiating $\frac{\partial}{\partial t}\Theta(t,\tau)=W(\Theta(t,\tau))$ with respect
	to $\tau$ we obtain
	\begin{eqnarray*}
		\frac{\partial^2 \Theta(t, \tau)}{\partial \tau \partial t} &=&
	-\frac{ \langle \nabla V(\Theta), \Theta'\rangle }{\|\nabla V(\Theta)\|^2} \nabla V(\Theta)
	-2 \frac{ w(\Theta) 
	\langle \nabla V(\Theta), HV(\Theta) \Theta'\rangle 
	}{\|\nabla V(\Theta)\|^4} 
	\nabla V(\Theta) 
	\\
		&&+
		\frac{w(\Theta)}{\|\nabla V(\Theta)\|^2}HV(\Theta) \Theta'
.
	\end{eqnarray*}
	The first term vanishes because we are taking $\Theta' \perp \nabla V(\Theta)$. 
	The other two terms factor as a product
\[
	\frac{\partial^2 \Theta(t, \tau)}{\partial \tau \partial t} =
	\left( I - \frac{2}{\| \nabla V(\Theta)\|^2} \nabla V(\Theta) \nabla V(\Theta)^T \right)
	\frac{w(\Theta)}{\|\nabla V(\Theta)\|^2} HV(\Theta) \Theta'
.
\]
    The first term is a symmetry that fixes $\Theta'$. The upper bound in Lemma~\ref{lemma-W2} yields
	\[
	\left \langle \Theta', \frac{\partial^2 \Theta(t, \tau)}{\partial \tau \partial t} \right \rangle 
	\le 
	\frac{w(\Theta)\sqrt{2w(\Theta)}}{\|\nabla V(\Theta)\|^2} 
	\|\Theta'\|^2 .
	\]
	We can now plug in Lemma~\ref{lemma-W1} to obtain
	\begin{equation}\label{bound-accel}
	\left \langle \Theta', \frac{\partial^2 \Theta(t, \tau)}{\partial \tau \partial t} \right \rangle 
	\le 
		\sqrt{2}\, Y(\Theta(t,\tau))\, e^{-t/2}\, \|\Theta'\|^2 .
	\end{equation}
	We can now proceed to bound $L(T)$. Indeed,
\begin{eqnarray*}
	L(T) &=& \int_a^b \left\| \Theta'(T,\tau) \right\| \dd \tau \\
	     &\le& \int_a^b \left( \left\| \Theta'(0,\tau) \right\| 
	     + \int_0^T \frac{\partial}{\partial t} \left\| \Theta'(t,\tau) \right\| \dd t \right) \dd \tau \\
		&=& L(0) + \int_{a}^b \int_{0}^T \frac{  \left\langle \Theta'(t,\tau), 
		\frac{\partial^2}{\partial t \partial \tau}\Theta(t,\tau) \right\rangle}
		{\| \Theta'(t,\tau)\|} \dd t \dd \tau \\
		&=& L(0) + \int_{a}^b \int_{0}^T \frac{  \left\langle \Theta'(t,\tau), 
		\frac{\partial^2}{\partial \tau \partial t}\Theta(t,\tau) \right\rangle}
		{\| \Theta'(t,\tau)\|} \dd t \dd \tau \\
\end{eqnarray*}
using the smoothness of the flow and the fact that $\gamma$ is of class $\mathcal C^2$.
	The bound \eqref{bound-accel} now implies that
	\[
		L(T) \le L(0)+\sqrt{2}  \int_{a}^b \int_{0}^T Y(\Theta(t,\tau)) e^{-t/2} \|\Theta'(t,\tau)\| 
		\dd t \dd \tau
	\]

	We picked $y^T\ge \max( \sup_{a \le \tau \le b} Y(\gamma(\tau)), \frac{1}{2(m+\sqrt{\epsilon})})$. From Lemma \ref{lemma-W3},
	$Y(\Theta(t,\tau)) \le y^T$. Changing the order of integration,
	\begin{eqnarray*}
		L(T) &\le& L(0)+\sqrt{2}y^T  
		\int_{0}^T 
		\int_{a}^b 
		e^{-t/2} 
		\|\Theta'(t,\tau)\| 
		\dd t \dd \tau\\
		&=& L(0)+\sqrt{2}y^T \int_{0}^T e^{-t/2} L(t) \dd t .
	\end{eqnarray*}
	We can bound $L(T) \le l(T)$ where $l(T)$ is the solution of the integral equation
\[
	l(T)=l(0)+\sqrt{2}y^T \int_0^T e^{-t/2}\, l(t) \dd t
\]
	with $l(0)=L(0)$. Differentiating and dividing by $l(t)$,
\[
	\frac{\partial}{\partial T} \log (l(T) ) = \sqrt{2}y^T e^{-T/2}
\]
	hence
\[
	\log (l(T))  = \log l(0) + 2\sqrt{2}y^T (1-e^{-T/2}).
\]
	Thus,
\[
	L(T) \le l(T) \le e^{2 \sqrt{2} y^T}\, L(0) .
\]
\end{proof}

\begin{proof}[Proof of Theorem~\ref{t.alpha-Lipschitz}]
	This is a particular case of Lemma~\ref{lemma-W4} applied to the level set
	$V(\Theta)=\frac{m^2}{2}-\epsilon$. 
Let $y^T$ be the maximum of 
	$\frac{1}{2(m+\sqrt{2\epsilon})}$ and the supremum of $Y(\Theta)$ on the level set $V(\Theta)=\frac{m^2}{2}-\epsilon$.
The Lipschitz constant is $\lambda=e^{2 \sqrt{2} y^T}$. 

	Fix $\Theta_1$, $\Theta_2$ at that level set
	and let $\gamma(\tau)$ be the minimizing geodesic connecting $\Theta_1$ and $\Theta_2$.
	For any fixed $T$, the length $L(T)$ of $\Theta(T, \tau)$ is at most $\lambda L(0)$.
	Since the curves $L(T)$, $T \in \mathbb N$ are uniformly bounded and share the same Lipschitz constant,
	Arzel\`a-Ascoli theorem guarantees that the limit is a curve 
	of length at most $\lambda L(0)$. We also know that the limit is in $\mathcal V^{\max}$.
	Therefore, the distance between $\alpha(\Theta_1)$ and $\alpha(\Theta_2)$ is no more than $\lambda L(0)$.
\end{proof}

\section{Cell decomposition of $\mathcal V^{\max}$}
\label{sec-cell}

In this section, we will use cube face coordinates $\theta_1 \equiv 0$
on $\mathcal V^{\max}$, with $0 \le \theta_1 < 2\pi$ for the other coordinates.
Each possible ordering of the coordinates $\theta_1=0, \theta_2,
\dots, \theta_m$ will define a {\bf cell} in $\mathcal V^{\max}$.
For instance,
the {\bf canonical (open) $m-2$-cell} is the set of all the points of $\mathcal V^{\max}$
that can be written as
\[
	\mathring{\mathcal C} = \{ [ \theta_1, \dots, \theta_{m}] \in \mathcal V^{\max}:
0 = \theta_1 < \theta_2 < \dots < \theta_{m} < 2\pi \}.
\]

\begin{lemma} The set $\mathring{\mathcal C}$ is open in $\mathcal V^{\max}$, and path
	connected. 
\end{lemma}

\begin{proof}
	By construction, $\mathring{\mathcal C}$ is an open subset of $\mathcal V^{\max}$.
	We can show path-connectedness as follows:
let $\omega=2\pi/m$ and $\Omega=(0,\omega,2\omega, \dots, (m-1)\omega)$.
	For any $\Theta$ with $q(\Theta) \in \mathring{\mathcal C}$ and $\theta_1=0$,
	we write $\Theta(t)=t \Theta + (1-t) \Omega$ and then
	$\alpha(\Theta(t))$ is a continuous path on $\mathring{\mathcal C}$.
\end{proof}

If $m=3$, the canonical cell is the point $\{[0,\pi/3, 2\pi/3]\}$.
If $m=4$, the canonical cell is the line $\{[0, \alpha, \pi, -\alpha]\}$
with $0 < \alpha < \pi$.

\begin{figure}
	\centerline{\includegraphics[width=\textwidth]{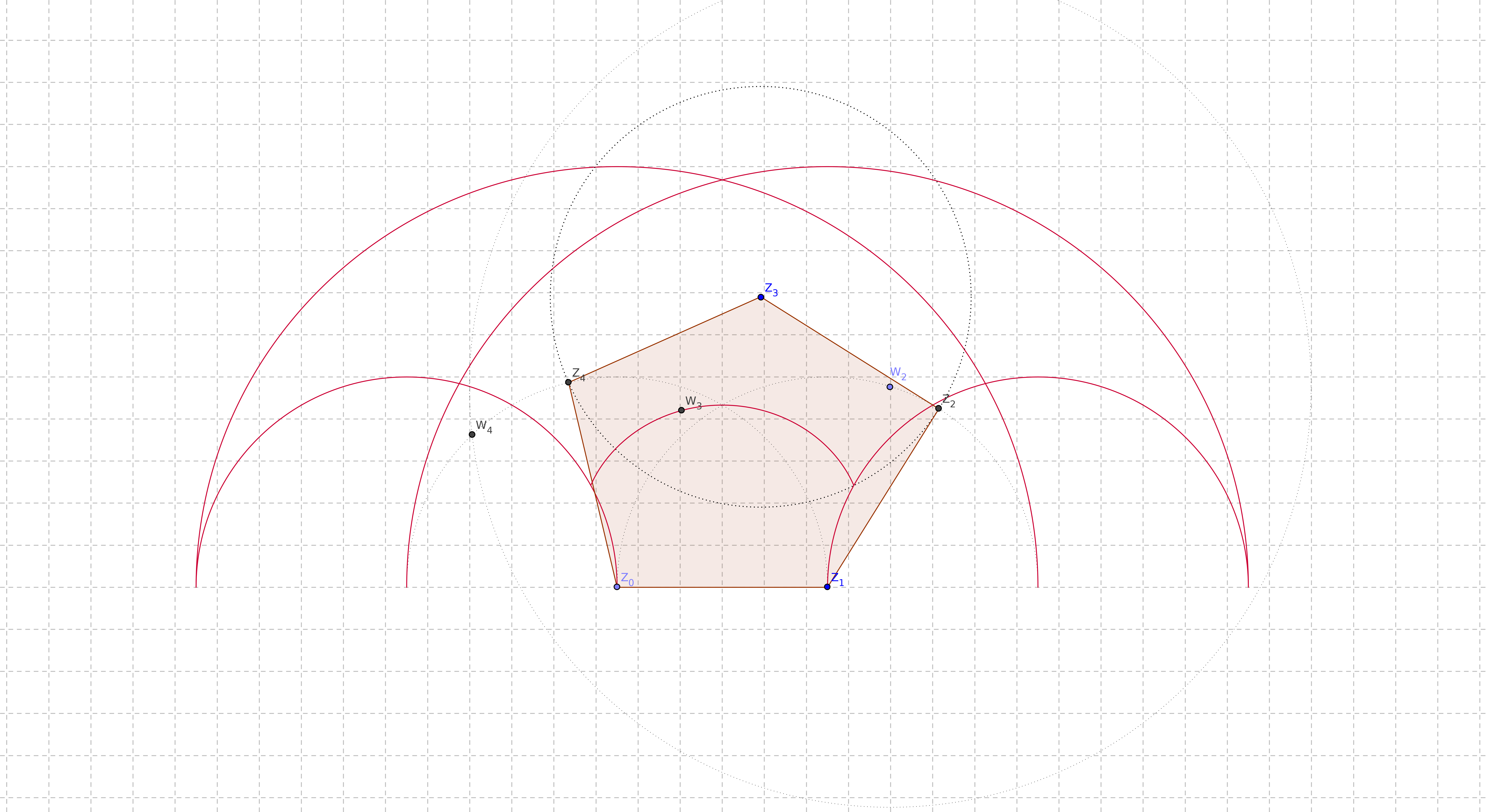}}
\caption{\label{geogebra1} The point $Z_3$ lies in the region delimited by
	the 5 red lines.}
\end{figure}

The canonical cell for $m=5$ can be constructed geometrically as
explained below (See Figure~\ref{geogebra1}).
Instead of representing each angle as a point in $\mathbb R \mod 2\pi \mathbb Z$, this time
we will represent it by a segment of unit length in the plane.
For instance, $[0=\theta_1, \theta_2, \dots, \theta_m]$ will be represented by the closed positively oriented polygon
\[
(Z_1=1, \dots, Z_m)
\]
where $Z_{k}=\sum_{j \le k} e^{i \theta_j}$.  The centroid condition is equivalent to the assertion $Z=\frac{1}{m}Z_m=0$, thus the last side $Z_m Z_1$ has length one.
Back to the case $m=5$, the point $Z_3$ must be at distance $\le 2$ from
$Z_m=0$ and from $Z_1$. So it belongs to the intersection of two radius 2 disks.
Since the polygon is positively oriented, it also belongs to the upper half-plane. There are three extra conditions for the angles at $Z_1$, $Z_3$ and $Z_4$ not to be reflex.

For every permutation $\sigma$ of $\{2, \dots, m\}$, we define
\[
	\mathring{\mathcal C}_{\sigma} = \{ [ 0=\theta_1, \theta_2, \dots, \theta_{m}] \in \mathcal V^{\max}:
0=\theta_1 < \theta_{\sigma(2)} < \dots < \theta_{\sigma(m)} < 2\pi \}.
\]
The $m-2$-cell $\mathcal C_{\sigma}$ is obtained by permuting the coordinates
of the canonical cell $\mathcal C$, but always fixing $\theta_1=0$. 
In the sequel we will consider {\bf closed} cells, 
\[
	\mathcal C_{\sigma} = \{ [ 0=\theta_1, \theta_2, \dots, \theta_{m}] \in \mathcal V^{\max}:
0=\theta_1 \le \theta_{\sigma(2)} \le \dots \le \theta_{\sigma(m)} \le 2\pi \}.
\]
For concision, each $m-2$-cell will be associated to a unique label. The
label will be a word in the alphabet $\{-, 0, a, b, \dots\}$. 
The symbol
$0$ is associated to $\theta_1=0$, the symbol $a$ to $\theta_2$,
the symbol $b$ to $\theta_3$ and so on. The label associated to
$\mathcal C_{\sigma}$ is the word starting with $0$ and followed by the 
the symbols for $\theta_{\sigma(2)}$, $\theta_{\sigma(3)}$, etc...

For instance, the label $\cell{0-a-b-d-c}$ represents the {\bf closed cell}
\[
0 =\theta_1 \le \theta_2 \le \theta_4 \le \theta_3 \le 2 \pi
.\]

Lower dimensional cells can be obtained by omitting a separator.
For instance,
$\cell{0a-b-d-c}$ will mean the $1$-cell
\[
0 =\theta_1 = \theta_2 \le \theta_4 \le \theta_3 \le 2 \pi
.
\]

This construction is formalized below:

\begin{definition*}
A {\bf sentence} is a list starting with $0$ and containing all the $m-1$
symbols $a, b, c, \dots$ without repetition. It contains one
or more instances of the separator. There cannot be adjacent separators
nor a separator at the end of the sentence.
\end{definition*}

To each sentence, we associate a system of equalities and inequalities
by the following algorithm: replace $0$ by the expression $0=\theta_1$,
$a$ by $\theta_2$, $b$ by $\theta_3$ and so on.
replace each separator by  the inequality symbol $\le$, whenever two $\theta_i$
are adjacent to each other introduce an equal symbol $=$, and always terminate by
$\le 2\pi$.

For instance,
$\cell{0af-bcd-e}$ becomes 
\[
0=\theta_1 =
\theta_2
=
\theta_7
\le
\theta_3
=
\theta_4
=
\theta_5
\le
\theta_6
\le 2 \pi
\]

Different sentences can lead to equivalent systems of equalities-inequalities,
viz. $\cell{0a-bc}$ and $\cell{0a-cb}$ become $0=\theta_1=\theta_2 \le \theta_3 = \theta_4 \le 2\pi$.

\begin{definition*}
	A {\bf word} in a sentence is a maximal sub-list without
	a separator. It cannot be empty. 
	It is said to be {\bf ordered} when the symbols appear in 
	increasing order.
\end{definition*}

\begin{definition*} A sentence is in {\bf canonical form} if and only if
	every word is ordered.
\end{definition*}

In the example above, $\cell{0a-bc}$ is the canonical form for $\cell{0a-cb}$. Different sentences in canonical form represent different systems of inequalities in $[0,2 \pi)^n$. But some of those systems may fail to have solutions in $\mathcal V^{\max}$.

\begin{example*}
	The sentence $\cell{0ab-c-d}$ corresponds to a system of equalities and inequalities without solution in $\mathcal V^{\max}$. 
The reason is that $[Z_5 Z_3]$ is a lenght 3 segment while $[Z_3 Z_4]$ and $[Z_4 Z_5]$ have both length 1. Since it fails the triangule inequality, the triangle $Z_5, Z_3, Z_4$ cannot exist.
\end{example*}

\begin{example*}
	The sentences $\cell{0ab-cde}$, $\cell{0a-b-cde}$ and $\cell{0-ab-cde}$ give rise to different systems of equalities-inequalities. Because of the triangular inequality, the solution set in $\mathcal V^{\max}$ is the same: $\{[0,0,0,\pi,\pi,\pi]\}$.
\end{example*}

This suggests the following restriction:

\begin{definition*}
	A sentence is {\bf valid} if and only if it satisfies the following
	conditions:
\begin{enumerate}
	\item \label{v1} Words have size $1$ to $m/2$.
	\item \label{v2} If a word has size $m/2$, then there are only
		two words in the sentence (with size $m/2$).
\end{enumerate}
\end{definition*}

This defines uniquely the
closed cells associated to all the valid sentences in canonical form. The intersection
of two of those cells is a lower dimensional cell or empty, and the union
of all $m-2$-dimensional cells is $\mathcal Q$.

It remains to define the {\bf border operator}.
The border of a cell $c$ is the union of the cells obtained by removing one separator from the valid sentence, as long as the sentence remains valid. Excepcionally, if a word of size $m/2$ appears, the other words collapse so that the sentence remains valid.

The {\bf orientation} is given by the following rule: all $m-2$ cells are
positively oriented. If $c' \in \partial c$ is such that the $i$-th
independent variable of $c$ collapses with the $i+1$-st (counting from $i=0$,
and modulo the dimension of $c$) then $c'$ counts in $\partial c$
with sign $(-1)^i$. 
Equivalently, removing the $k$-th separator makes the new sentence count
as $(-1)^{k-1}$ in the border of the original sentence.

\begin{example*}
\[
\partial\cell{0-ab-c-d-e} =
\cell{0-ab-cd-e} -\cell{0-ab-c-de} +\cell{0e-ab-c-d} 
\]
The $0$-cells $\partial\cell{0ab-cde}$ and 
$\partial\cell{0de-abc}$ do not show up in the sum above because
of the dimension. But notice that
\begin{eqnarray*}
	\partial \cell{0-ab-cd-e} &=&
\cell{0ab-cde} 
-\cell{0e-ab-cd} 
\\
\partial \cell{0-ab-c-de} &=&
\cell{0ab-cde} 
- \cell{0de-abc} 
\\
\partial \cell{0e-ab-c-d} &=&
\cell{0e-ab-cd} 
-\cell{0de-abc} 
\end{eqnarray*}
and therefore $\partial \partial \cell{0-ab-c-d-e}=0$ as expected.
\end{example*}

\begin{figure}
	\centerline{\includegraphics[clip, trim= 0 300 0 110]{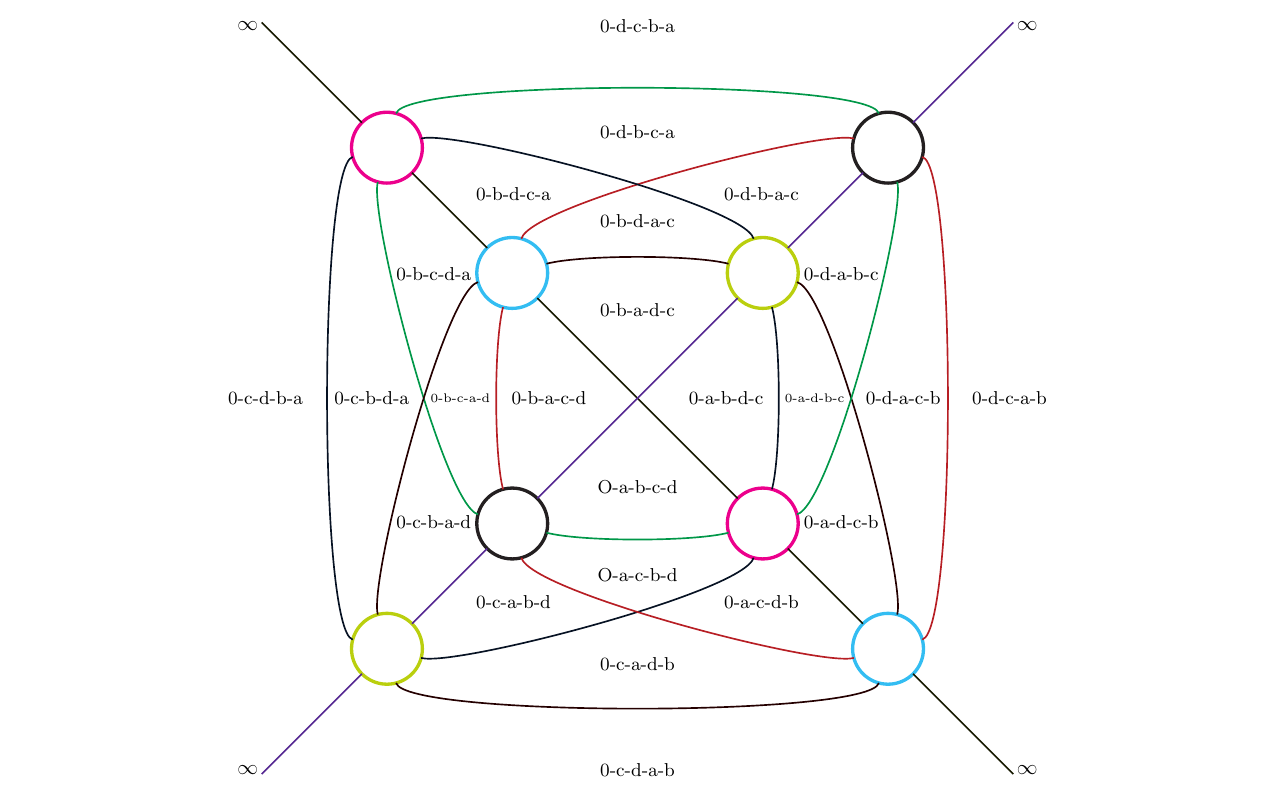}}
	\caption{Cell decomposition of $\mathcal V^{\max}$ when $m=5$. This picture
	shows a sphere with 4 handles,
	under stereographic projection, so the point at
	infinity is a regular point (and actually a $0$-cell). 
	The 4 handles (in different colors) are the union of 6 $1$-cells
	and 6 $0$-cells each.\label{atlas}}
\end{figure}
\begin{theorem}\label{t.cw} The set of cells in $\mathcal V^{\max}$ induced by the ordered valid sentences plus the empty sentence $\emptyset$ 
is a CW-complex for $\mathcal V^{\max}$.
\end{theorem}

\begin{example*} If $m=4$, the valid 1-cells are labelled as follows:
\begin{itemize}
	\item $\cell{0-a-b-c}$ with border $\cell{0a-bc} - \cell{0c-ab}$.
	\item $\cell{0-c-b-a}$ with border $\cell{0c-ab} - \cell{0a-bc}$.
	\item $\cell{0-b-c-a}$ with border $\cell{0b-ac} - \cell{0a-bc}$.
	\item $\cell{0-a-c-b}$ with border $\cell{0a-bc} - \cell{0b-ac}$.
	\item $\cell{0-c-a-b}$ with border $\cell{0c-ab} - \cell{0b-ac}$.
	\item $\cell{0-b-a-c}$ with border $\cell{0b-ac} - \cell{0c-ab}$.
\end{itemize}
	Each pair of consecutive $1$-cells is a topological circle. In box
	coordinates, it becomes a  
	line segment. The 
three topological circles intersect two by two on one point.
\end{example*}

\begin{example*}
If $m=5$, 
\[
	\partial \cell{0-a-b-c-d} =
\cell{0a-b-c-d} 
-\cell{0-ab-c-d} 
+\cell{0-a-bc-d} 
-\cell{0-a-b-cd} 
+
\cell{0d-a-b-c} 
\]
and the other ($m-2$)-cells are permutations hereof. 
	The map of Figure~\ref{atlas} was obtaining by patching 
those cells together. Also, notice that there are
precisely $4!=24$ two-cells. Each cell has $5$ edges, and each edge belongs to the border of exactly two
	two-cells. Hence, there are precisely $24\time 5/2=60$ one-cells.

	The one-cells are all of the form $\theta_i=\theta_j$ for a pair $i \ne j$. The permutation $(ij)$ acts
	on $\mathcal V^{\max}$ by swapping coordinates $\theta_i$ and $\theta_j$. It is an isometry, and fixes
	the line $\theta_i=\theta_j$ which is therefore a geodesic. Two such geodesics can only intersect at
	a right angle, so the number of $0$-cells is $24 \times 5 / 4 = 30$. Alternatively, $0$-cells are all
	defined by the choice of two distinct cardinality $2$ subsets in $[5]$, so there are $\frac{5!}{2! 2!}=
	30$ of them.

	The Euler characteristic of $\mathcal V^{\max}$
is therefore
\[
\chi = 24-60+30=-6=2-2g.
\]
so its genus is $4$. 
\end{example*}

\begin{proof}[Sketch of the proof of Theorem~\ref{t.cw}]
The hard part is to show that cells are homeomorphic to balls.
We start in dimension $0$.
	Labels with 
at most two separators correspond to a unique point in $\mathcal V^{\max}$. 
If $m$ is odd and there are words of size $k$, $l$ and $m-k-l$, there is
a unique closed triangle with a horizontal segment of size $k$ passing
through the origin, then (in trigonometrical order) 
a horizontal segment of size $l$ and another of size $m-k-l$.
The case $m$ even is simpler.

Cells with three separators are lines, corresponding to quadrilaterals
with side given by the length of the 4 words. The two extreme figures
(flat quadrilaterals) are 0-cells.

In general, $k$-cells are associated to a system of linear equalities
and inequalities in the nonlinear surface $\mathcal V^{\max}$. This means that
a $k$-cell $c$ is always the intersection of $\mathcal V^{\max}$ with a 
$k+2$ simplex $\Delta$.
	Let $\Omega$ be an interior point of $c$ and let 
$T=T_{\Omega}\mathcal V^{\max}$ be the tangent plane at $\Omega$. Let
$\Delta'$ be the intersection of $T$ with the interior of $\Delta$,
so $\Delta'$ is topologically an open $k$-dimensional ball. It can
be deformation retracted to $c$ along the orbits of the flow.

The easy part of the proof is to construct a map attaching the border
of a $k$-cell $c$ to the ($k-1$)-skeleton. Since points in the border have at least one separator less, they belong to the ($k-1$)-skeleton. 
\end{proof}

\section{The singularities of $\mathcal V^{\max}$}
\label{sec-sing}
\begin{figure}
	\centerline{
\includegraphics[width=0.7\textwidth,trim=1in 0in 1in 0in,clip]{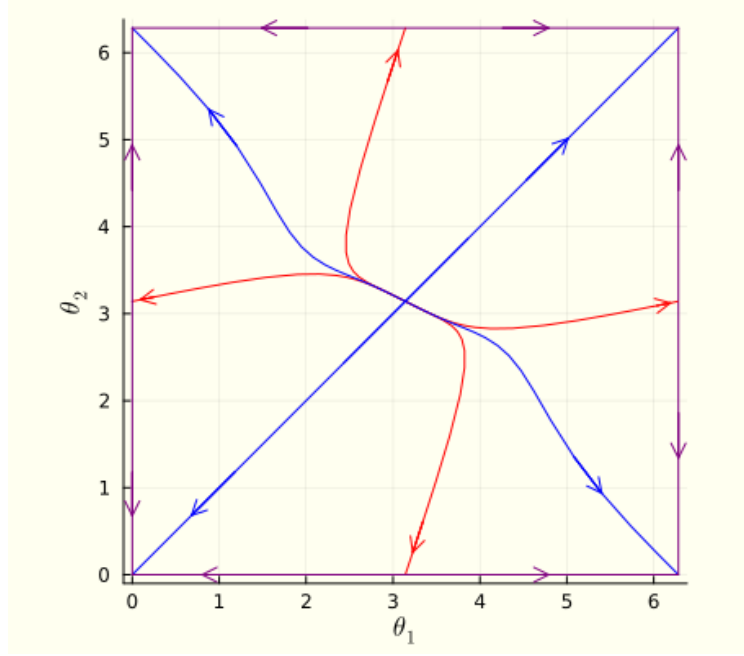}
}
	\caption{\label{fig-singularity}
	Reduced flow when $m=4$ restricted to
	the plane $\theta_3 = \theta_4 = 0$. The point $(\pi,\pi)$
	is the projection of $(\pi,\pi,0,0) \in V^{\sing}$.
    }
\end{figure}
If $m \ge 4$ is odd, the set $\mathcal V^{\max}$ is a smooth connected manifold. But if $m=2d \ge 4$ is even,
the set $\mathcal V^{\max}$ admits $\frac{1}{2}\binomial{m}{d}$ isolated singularities. Their  
counterdiagonal coordinates are permutations of 
$p=[-\pi/2: \dots : -\pi/2 : \pi/2 : \dots : \pi/2]$. The aim of this section is
to describe those singularities. Because singularities can be obtained
from each other by permuting coordinates, we will always consider the canonical singularity
$p$ with coordinates above. 

\medskip
\noindent \\
{\bf Trivial case.} Assume $m=4$. In the notation of the previous section, the singularity is the 0-cell $\cell{0a-bc}$.
The incident 1-cells are $\cell{0-a-b-c}$, $\cell{0-b-c-a}$, $\cell{0-a-c-b}$, $\cell{0-c-b-a}$. 
The tangent vectors of $\mathcal V^{\max}$ at the singularity $p$ are multiples of $[1,-1,1,-1]^T$ or of
$[1,-1,-1,1]^T$. The reader can quickly check that near $p$, $\mathcal V^{\max}$ is exactly the union of two orthogonal 
lines and that both are orthogonal to $[1,1,-1,-1]$.
The normal bundle collapses at $\mathcal V^{\sing}$, see Figure~\ref{fig-singularity} for what this means to the flow.
\medskip
\noindent \\
{\bf Easy case.} Assume now $m=6$. The canonical singularity is the 0-cell $p=\cell{0ab-cde}$. Its {\bf blowup}
is the set $S_p \subset S^{m-2} \subset T_p\mathcal Q$ of all vectors of the form $\gamma'(0)$, 
where $\gamma: [0,\delta) \rightarrow V^{\max}$ is
a curve with $\gamma(0)=p$ and $\|\gamma'(0)\|=1$. 
We assume
{\bf counterdiagonal coordinates} 
$\gamma(t)=[ -\pi/2 + t a_1, -\pi/2 + t a_2, -\pi/2 + t a_3, \pi/2 + t b_1, \pi/2 + t b_2,
\pi/2 + t b_3]+O(t^2)$ for $t$ small enough. It is assumed that
$\sum a_j + \sum b_j=0$.
Let $mZ$ be the map defined in the proof of Theorem~\ref{t.Vmax} in section~\ref{t.centroid}. Its derivative 
is the matrix
\[
	\begin{bmatrix}
		-\cos(t a_1) &
		-\cos(t a_2) &
		-\cos(t a_3) &
		\cos(t b_1) &
		\cos(t b_2) &
		\cos(t b_3) \\
		\sin(t a_1) &
		\sin(t a_2) &
		\sin(t a_3) &
		-\sin(t b_1) &
		-\sin(t b_2) &
		-\sin(t b_3) &
\end{bmatrix}.
\]
Direct calculation shows that $\gamma'(0)$ is normal to  
the vector  $[-1,-1,-1,1,1,1]^T$. Therefore,
$S_p$ is contained in the unit sphere of
the subspace $a_1+a_2+a_3=b_1+b_2+b_3=0$.
This sphere is 3-dimensional. But notice that 
\[
0=	\sum \sin(\gamma(t)) = \frac{1}{2} \left(\sum a_j^2 - \sum b_j^2\right) t^2 + O(t^3)
\]
so that we can normalize $\sum a_j^2 = \sum b_j^2 = \frac{1}{2}$. This implies that
$S_p \subset S^1 \times S^1 \subset S^3$, assuming the canonical embedding
$(x_1,x_2), (y_1, y_2) \rightarrow \frac{1}{\sqrt{2}} (x_1,x_2,y_1,y_2)$.
It turns out that $S_p$ is actually the product $S^1 \times S^1$.

Indeed, there are 36 possible 3-cells of $\mathcal V^{\max}$ that contain 
$p$ as a vertex. Those can be obtained
from $\cell{0-a-b-c-d-e}$ by permuting $0,a,b$ and $c,d,e$ separately, and then
rotating so that the cell description starts at $0$, e.g. 
$\cell{a-b-0-c-d-e}$ is rewritten as $\cell{0-c-d-e-a-b}$.
As all cells are isometric, we consider only $\cell{0-a-b-c-d-e}$.
Its edges (1-cells) incident to $p$ are $\cell{0-ab-c-de}$, $\cell{0-ab-cd-e}$,
$\cell{0a-b-c-de}$ and $\cell{0a-b-cd-e}$. The tangent vectors
to the edges are the columns of the matrix
\[
\frac{1}{\sqrt{12}}
\begin{bmatrix}
-2 & -2 & -1 & -1  \\
 1 &  1 & -1 & -1  \\
 1 &  1 &  2 &  2  \\
-2 & -1 & -2 & -1  \\
 1 & -1 &  1 & -1  \\
 1 &  2 &  1 &  2  \\
\end{bmatrix} .
\]
There is an obvious product structure on the matrix above, so we consider now
all the possible $(a_1, \dots, a_3)$. In addition to having $a_1+a_2+a_3=0$,
we can always find a permutation of the rows of 
\[
\begin{bmatrix}
-2 & -1 \\
 1 & -1 \\
 1 &  2 \\
\end{bmatrix}
\]
so that $(a_1, a_2, a_3)$ belongs to the cone generated by the columns. 
It follows that
the first factor $S^1$ is divided into six equal arcs, each corresponding to
a certain ordering of the $a_i$. 
Columns itself correspond to configurations with two of the $a_i$ equal. 
The same holds for the second factor. Thus
$S_p=S^1 \times S^1$, and one can think of each cell incident to $p$ as 
projecting in $S_p=S^1 \times S^1$ as a product of arcs of lenght $\pi/3$.

Among all possible row permutations, the columns of the matrix above maximize
the linear form $[-1,0,1]$. For each ordering of the $a_i$, the matrix obtained
by permutation will maximize a linear form obtained by permuting $[-1,0,1]$.
The set of possible linear forms so obtained is a regular hexagon.
In the language of fans, the set of incident cells and their intersections 
projects onto the dual fan to the product of two regular hexagons, one for 
each factor of $S^1$.
\smallskip
\noindent \\
{\bf General case.} At this point we can summarize the general case as follows,
the proof is the same as above.

\begin{theorem}\label{t-blow-up} Let $m=2d$, and consider the singularity
	$p=[-\pi/2: \dots: -\pi/2:\pi/2 : \dots : \pi/2]$
	with blowup $S_p$ as above. Then, $S_p=S^{d-2}\times S^{d-2} \subset S^{2d-3}$. 
	A typical incident $2d-3$-dimensional cells of $\mathcal V^{\max}$ is
	\cell{0-a-b-c\dots}. It projects on $S_p$ as
	the product $C_1 \cap S^{d-2} \times C_2 \cap S^{d-2}$, where each
	$C_i$ is the cone
	generated by the columns of the matrix
\[
\begin{bmatrix}
	-(d-1) & -(d-2) & -(d-3) & \dots & -2 & -1 \\
	     1 & -(d-2) & -(d-3) & \dots & -2 & -1 \\
	     1 &      2 & -(d-3) & \dots & -2 & -1 \\
	     1 &      2 &      3 & \dots & -2 & -1 \\
          \vdots & \vdots & \vdots & &\vdots & \vdots & \\
	     1 & 2 & 3 & \dots & (d-2) & (d-1)        
\end{bmatrix}
.
\]
     This cell is orthogonal to the vector $[-1,0, \dots,1,1,0, \dots, -1]^T$.
     The projection of each of the $d!^2$ incident cells can be constructed 
in the same way, by permuting the rows of the matrix above, independently for
each factor $C_1$ and $C_2$.
\end{theorem}

Let $\mathcal P \subset \mathbf 1^T \subset \mathbb R^d$ be the convex hull 
of all the points
obtained by permutation of $(-1,0,0,\dots,1)$. Then the set of all incident
cells of $\mathcal V^{\max}$ incident to $p$, of any dimension, projects into
$S_p$ as the dual fan of $\mathcal P \times \mathcal P$.

\section{The homology of $\mathcal V^{\max}$}
\label{sec-homology}

Rather than proceeding with cell decompositions of $\mathcal V^{\max}$
in high dimension,
we will resort to long exact sequences to compute the Betti numbers
of $\mathcal V^{\max}$. This section is more technical than the rest of the paper.
Recall that $\mathcal Q=q(\mathbb T^m)$ denotes the quotient of $\mathbb T^m$ with
respect to the diagonal action. Then $\mathcal V^{\max} \subset \mathcal Q$ is the maximal locus of
the Kuramoto potential in the quotient.

\begin{theorem}\label{th-homology}
	Let $d=\lfloor{\frac{m-1}{2}}\rfloor$.
	The Betti numbers 
	$\beta_k=\dim\left(H_k(\mathcal V^{\max}, \mathbb Z)\right)$ 
	of $\mathcal V^{\max}$ are:
\[
	\beta_k= \left\{
\begin{array}{cl}
	0 & \text{if $k \ge m-2$,}\\ \\
	\binomial{m-1}{k+2}&\text{if $m-d-1 \le k \le m-3$,}\\ \\
	\binomial{m-1}{k}+ 
	\binomial{m-1}{k+2}& 
	\text{if $k=m-d-2$, and}\\ \\
	\binomial{m-1}{k} & \text{if $0 \le k \le m-d-3$.}
\end{array}
		\right.
\]
\end{theorem}

The table of Betti numbers looks like Pascal's triangle, except for the coefficients in {\bf boldface} below. Those are the sum of two non-adjacent terms in Pascal's triangle,
and the term in between is erased.

\centerline{
\begin{tabular}{rr|rrrrrrrrr}
$m$ & $d$ & 
	$\beta_0$ &
	$\beta_1$ &
	$\beta_2$ &
	$\beta_3$ & $\beta_4 $ & $\beta_5$ & $\beta_6$ & $\beta_7$ & $\beta_8$
	\\ \hline 
	4 & 1 & 1 & {\bf 4} \\
	5 & 2 & 1 & {\bf 8} & 1 \\
	6 & 2 & 1 &5 &{\bf 15} &1 \\
	7 & 3 & 1 &6 &{\bf 30} &6 &1 \\
	8 & 3 & 1 &7 &21 &{\bf 56} &7 &1\\
	9 & 4 & 1 &8 &28 &{\bf 112} &28 &8 &1\\
	10 & 4 & 1 & 9 & 36 & 84 & {\bf 210} & 36 & 9 & 1\\ 
	11 & 5 & 1 & 10 & 45 & 120 & {\bf 420} & 120 & 45 & 10 & 1\\ 
\end{tabular}}

The lower part $\mathcal K$ of $\mathcal Q$ was defined as the union of all the cells $Q^I$, $|I|<m/2$.
Recall that for $\abs{I} < m/2$, the quotient flow admits an index $\abs{I}$ fixed point
$p_I=q(\Theta)$ where $\Theta$ is the exemplar
$\theta_i=0 \,\forall i\not \in I$,
$\theta_i=\pi$ otherwise. 
Therefore,
$\mathcal K$ is precisely  the union of all the fixed 
points outside $\mathcal V^{\max}$ and their unstable manifolds.
%
For instance if $m=5$, $\mathcal K$ is the union of $\binomial{5}{2}$ sub-tori of dimension 2 of the form
$\mathcal Q^{\{1,2\}}=q(\mathbb T_{\{1,2\}}) = \{[\theta_1, \theta_2, 0,0,0]\}$, and permutations. 

By construction, $\dim(\mathcal K) = d=\lfloor \frac{m-1}{2} \rfloor$.
For any $\abs{I} \le d$, we consider the trivial embedding of 
$\mathcal Q^I \rightarrow \mathbb T^I$, taking a subset of quotient space $\mathcal Q$
into the $m$-torus $\mathbb T^m$. This embedding can be extended to all of $\mathcal K$. 
Its image is precisely the {\bf $d$-skeleton} of the
canonical cell complex of the $m$-torus
$\mathbb T^m$.

The cohomology subcomplexes of the torus $\mathbb T^m$ are 
worked out in \ocite{Hatcher}, example 3.23.
Let $c^1, \dots, c^m$ the generators of $H^1(\mathbb T^m)$.
The cohomology ring of $\mathbb T^m$ is
the exterior algebra $\bigwedge_{\mathbb Z}[c^1, \dots,
c^m]$. The cohomology ring $H^*(\mathcal K)$ is obtained by
setting to zero all the monomials of degree $d+1$. For instance,
if $m=4$ and $m=5$,
$H^*(\mathcal K,\mathbb Z)$ is the $\mathbb Z$-span of the unit $1$, 
the $c^i$, and the expressions $c^i c^j = -c^j c^i$. 

The sets $\mathcal Q$, $\mathcal K$ and $\mathcal V^{\max}$ all have the structure of CW-complexes.
The main idea to prove Theorem~\ref{th-homology} is to deduce the homology of $\mathcal V^{\max}$ 
from the cohomology of $\mathcal K$. Here is the first step.

\begin{proposition}\label{deformation-retract} There is a neighborhood $\mathcal K \subset \mathcal U \subset \mathcal Q$
with the following properties:
	\begin{enumerate}
		\item The Kuramoto vector field $K=-\nabla V_{\mathcal Q}$ is transversal to $\partial \mathcal U$,
			and points always inwards.
		\item The set $\mathcal K$ is a deformation retract of $\bar {\mathcal U}$.
	\end{enumerate}
\end{proposition}

\begin{proof}
	{\bf Step 1.} Fix $\epsilon > 0$ small enough, and let $\mathcal U_{\infty}$ be the set of all $[\Theta] \in \mathcal Q$ with
	$d([\Theta], \mathcal K) < \epsilon$. For all $[\Theta] \in \partial \mathcal U_{\infty}$, we have
\[
	\epsilon = d([\Theta], \mathcal K) = \min_{|J|<m/2} d([\Theta], \mathcal Q^J). 
\]
	Assume that the minimum is attained at some $I \subset [m]$, $|I|=d$. 
	Let $[\eta]$ be the closest point,
	$\eta_i=\theta_i$ for $i \in I$ and $\eta_i = \bar \eta+ 2 k_i \pi$ for $i \not \in I$, $k_i \in \mathbb Z$.
	By $2\pi \mathbb Z^m$ action we may asssume that $k_i=0$. By diagonal action we assume   
	$\sum_{i \not \in I} \theta_i = 0$ and hence $\bar \eta=0$. In that case, 
	\[
		d ([\Theta], \mathcal Q^I)^2 = \sum_{i \not \in I} \theta_i^2.
	\]
	Define a vector field $N_I(\Theta) = -\frac{1}{2}\nabla d ([\Theta], \mathcal Q^I)^2 = -\sum_{i \not \in I} \theta_i \mathrm f_i=
	-\sum_{i \not \in I} \theta_i \mathrm e_i$ (See {\em counterdiagonal coordinates} 
	in Section \ref{sec-quotient}). The vector field $N_I$ extends to a neighborhood of $q^{-1}(\mathcal K)
	\subset \mathbb T^m$.
	We compute now the inner product
	\begin{eqnarray*}
		\left \langle K(\Theta), N_I(\Theta) \right \rangle &=&
		\sum_{i \not \in I} -\theta_i \sum_k \sin(\theta_k - \theta_i) \\
		&=& -\left(\sum_k \sin(\theta_k)\right) 
		\sum_{i \not \in I} \theta_i \cos(\theta_i) 
		+ \left(\sum_k \cos \theta_k \right) \sum_{i \not \in I} \theta_i \sin \theta_i.
	\end{eqnarray*}
	Since $|\theta_i|<\epsilon$ for all $i \not \in I$, $\sum_k(\cos \theta_k) > (m-d)(1-O(\epsilon^2)) 
	+ \sum_{k \in I} \cos(\theta_k) > 1/2$ once $\epsilon$ is small enough. On the other hand,
	$|\sum_{i \not \in I} \theta_i \cos(\theta_i)| < O(\epsilon^3)$.
	Therefore,
	\[
		\left \langle K(\Theta), N_I(\Theta) \right \rangle \ge \frac{1}{2}\sum_{i \not \in I} \theta_i^2 - O(\epsilon^3) > 0
	\]
	once $\epsilon>0$ is small enough.

{\bf Step 2}. From the previous step, the Kuramoto flow enters transversely the $\epsilon$-neighborhood of each $\mathcal Q^I$.
	Since $\mathcal U_{\infty}$ is the union of those neighborhoods, it is transversal to $\partial \mathcal U_{\infty}$ at all smooth points. We can now replace
	$\mathcal U_{\infty}$ by a slightly larger neighborhood $\mathcal U$ so that $\partial \mathcal U$ is smooth.

	Indeed, $x \in \mathcal U_{\infty}$ if and only if
	\[
		d(x,K) = \min_I d(x, \mathcal Q_I) < \epsilon.
	\]
Define instead
\[
	d_h(x,K) = \frac{1}{\left(\sum_I d(x, \mathcal Q_I)^{-h}\right)^{1/h}}
\]
for $h\ge 1$. We have always 
\[
	d_h(x,K) \le d(x,K)=d_{\infty}(x,K)=\lim_{h \rightarrow \infty} d_{h}(x,K) 
\]
	It follows that $\mathcal U_h =\{ x: d_h(x,K) < \epsilon \}$ contains $\mathcal U_{\infty}$. The border $\partial \mathcal U_h$
is smooth, and by taking $h$ large enough we can still guarantee that the flow enters $\mathcal U:=\mathcal U_h$.

{\bf Step 3}. Consider now the surfaces $S_t = \Psi_{t}(\partial \mathcal U)$ for all $t \ge 0$, where $\Psi_t$ is the 
Kuramoto flow on the quotient $\mathcal Q$. On the limit, $\lim_{t \rightarrow \infty} d_H(S_t, \mathcal K)=0$ where
$d_H$ stands for the Hausdorff distance, $d_H(X,Y)=\max \left(\sup_{x \in X} (d(x, Y)), \sup_{y \in Y} (d(y,X))\right)$. Let $N(q)$ be the outer normal vector of $S_t \ni q$ at $q$. This
defines a smooth vector field on $\mathcal U \setminus \mathcal K$. The map $f: \mathcal U \rightarrow \partial \mathcal U$ obtained
by integrating orbits of $\dot q = N(q)$ until they reach $S_0=\partial \mathcal U$ is a smooth map, and its derivative has rank $m-2$.
Therefore $f: \mathcal U \setminus \mathcal K \rightarrow \partial \mathcal U$ is a submersion, and the orbits of $N$ define a
one dimensional foliation of $\mathcal U \setminus \mathcal K$. Each leave has a limit point on $\mathcal K$. Each smooth
point of $K$ is the limit of infinitely many leaves, orthogonal to it. Define a deformation retract
from $\bar {\mathcal U}$ onto $\mathcal K$ by setting $r(q,0)=q$, and requiring for each fixed $q$ that $r(q,t)$ remains in the
same leave $\gamma$ as $r(q,0)$, that $r(q,1)$ is the limit point of $\gamma$ in $\mathcal K$ and that distances 
along $\gamma$ satisfy 
	\[
	d_{\gamma} (r(q,t), r(q,1)) = (1-t) d_{\gamma}(r(q,0), r(q,1)).
	\]
\end{proof}

\begin{proposition}\label{deformation-retract-2}
The high potential region
	$\mathcal V^{\high} \subset \mathcal Q$ contains
a neighborhood of $\mathcal V^{\max}$, and deformation retracts to it.
It admits a cellular decomposition as 
a subset of $\mathcal Q$.
\end{proposition}

\begin{proof}
	Corollary~\ref{c.alpha-retracts} shows that $\alpha: \mathcal V^{\high} \rightarrow \mathcal V^{\max}$
is a retraction. In order to produce a deformation retract, we define $r_t(q)$ as the time $t/(1-t)$ flow
	of the vector field $W$ from the proof of Theorem~\ref{t.alpha-Lipschitz}. The limit when $t \rightarrow 1$
	is $r_1(q)=\alpha(q)$ and continuity of $\alpha$ was already established. The cellular decomposition
	is obtained by cutting the cells of $\mathcal Q$ by the level sets $V_{\mathcal Q}=m^2/2$, actually $\mathcal V^{\max}$, 
	and $V_{\mathcal Q}=m^2/2-\epsilon$.
\end{proof}

\begin{proposition}\label{pi-injective}
Let $d= \dim(\mathcal K)=\lfloor \frac{m-1}{2}\rfloor$. 
Then,
\begin{equation}\label{eq-cohomology}
	\dim H_k(\mathcal Q,\mathcal V^{\high}; \mathbb Z) = 
\left\{
	\begin{array}{cl}
		\binomial{m}{k+1} &
		\text{for $k \ge m-d-1= \lceil \frac{m-1}{2} \rceil$} \\
		0 & \text{otherwise}
	\end{array}
	\right.
\end{equation}
	Moreover, if $k \ge m-d-1$ and $\pi: C_k (\mathcal Q) 
	\rightarrow C_k(\mathcal Q)/ C_k(\mathcal V^{\high})$ is the quotient
	map of cell complexes,
	then the induced homology map
	$\pi_*: H_k(\mathcal Q) \rightarrow H_k(\mathcal Q,\mathcal V^{\high})$
	is injective.
\end{proposition}

\begin{proof}
The first part of the proposition follows from 
	Alexander duality. The set $\mathcal Q \setminus \mathcal V^{\high}$ (low and mid potential region) 
	is compact and locally contractable.
It can be deformation retracted first to $\mathcal U$ along the Kuramoto flow,
and then to $\mathcal K$ as in Proposition \ref{deformation-retract}.
Hence $\mathcal Q \setminus \mathcal V^{\high}$ and $\mathcal K$ have isomorphic cohomology.
	In this specific case, Theorem 3.44 in \ocite{Hatcher} directly implies that
\[
	H_k(\mathcal Q , \mathcal Q \setminus \mathcal V^{\high}; \mathbb Z) \approx
	H^{m-k-1}(\mathcal K, \mathbb Z) 
.
\]
The cohomology ring of $\mathcal K$ has $m$ independent generators $c^i$, $1 \le i \le m$ and vanishes in dimension $>d$. 
This implies the first statement.

	For the second statement, we will refer to the proof of Theorem 3.44 in \ocite{Hatcher}. 
We first choose coordinates $\hat \Theta \mapsto
	[\hat \theta_1, \dots, \hat \theta_{m-1}, 0]$ in $\mathcal Q$.
This induces an ismorphism between $H_k(\mathbb T^{m-1})$
and $\mathcal Q$.
	A basis for the homology of $\mathbb T^{m-1}$ is given by
the homology class $c_I$ of $\mathbb T_I \subset \mathbb T^{m-1}$,
	where $\abs{I}=k$ and $1 \le i_1, \dots, i_k \le m-1$. This defines a 
basis for the homology of $\mathcal Q$.

A basis for the cohomology of $\mathcal Q$ is given by
$c^J$, with $|J|=m-k-1$, $1 \le j_{k+1}, \dots, j_m \le m-1$,
and $\langle c^J, c^I \rangle = \pm 1$ for $I \cap J = \emptyset$.
Signs can be adjusted so that the Poincar\'e duality isomorphism is
	\[
		c^J \mapsto c_{[m-1] \setminus J} .
	\]
Since $m-k-1 \le d$, each of those $c^J$ defines a non-zero cohomology class
in $\mathcal K$. This provides an injective map $H_k(\mathcal Q) \rightarrow
H^{m-1-k}(\mathcal K) \approx H_k(\mathcal Q,\mathcal V^{\high})$. The following diagram
	commutes\footnote{
We may follow the steps of the proof of Theorem 3.44, yet we may simplify some of the hypotheses. Let $\mathcal U \supset \mathcal K$ be an open neighborhood that deformation retracts onto $\mathcal K$, that
is assumed small enough. By taking a convenient subdivision of the chain complexes, the same construction as above produces a map 
$H_k(\mathcal Q;\mathbb Z) \rightarrow H^{m-k-1}(\mathcal U)$. By Poincar\'e duality,
$H^{m-k-1}(\mathcal U) \approx H_{k}(\mathcal U,\mathcal U \setminus \mathcal K;\mathbb Z) \approx
	H_k(\mathcal Q, \mathcal V^{\high};\mathbb Z)$. Since $\mathcal U$ deformation retracts onto $\mathcal K$, $H^{m-k-1}(\mathcal U)=H^{m-k-1}(k)$. This is the construction of the Alexander duality 
	$H^{m-k-1}(\mathcal K) \rightarrow H_k(\mathcal Q, \mathcal V^{\high}; \mathbb Z)$.}
\[ 
\minCDarrowwidth1em
\begin{CD}
H_k(\mathcal Q;\mathbb Z) @>\pi_*>> H_k(\mathcal Q,\mathcal V^{\high};\mathbb Z)
\\
	@A\text{Poincar\'e duality}AA @AA\text{Alexander duality}A \\
	H^{m-k-1}(\mathcal Q)
@>i^*>>
H^{m-k-1}(\mathcal K) 
\\
\end{CD}
	\]
and the vertical maps are isomorphisms. Therefore, $\pi_*$ is injective.
\end{proof}

\paragraph{Proof of the theorem} 
All the homology classes in this section are over the ring  
$\mathbb Z$ of integers.
Since $\mathcal Q \approx \mathbb T^{m-1}$,
\[
	\dim (H_k(\mathcal Q)) = 
\binomial{m-1}{m-k-1} = \binomial{m-1}{k}
\]

The long exact homology sequence for the pair $(\mathcal Q,\mathcal V^{\high})$ is
\begin{equation}\label{long}
\minCDarrowwidth1em
\begin{CD}
\dots 
@>\partial>>  
H_k(\mathcal V^{\high}) @>i_*>>
	H_k(\mathcal Q) @>\pi_*>> 
	H_k(\mathcal Q, \mathcal V^{\high}) 
@>\partial>>
H_{k-1}(\mathcal V^{\high})
@>>> 
\dots \\ 
@. @. @. \dots @>\pi_*>>
	H_{0}(\mathcal Q, \mathcal V^{\high})
@>>>
		0 .
\end{CD}
\end{equation}

The homology of $\mathcal V^{\high}$ for small dimensions is easy to compute.

\begin{lemma}\label{k-small} Let $d=\lfloor\frac{m-1}{2}\rfloor$.
	Under the assumptions above, if $k< m-d-2$, then $H_k(\mathcal V^{\high}) \approx
	H_k(\mathcal Q)$.
\end{lemma}
\begin{proof}
	From \eqref{eq-cohomology},
$H_{k+1}(\mathcal Q, \mathcal V^{\high})=0$ and $H_k(\mathcal Q, \mathcal V^{\high})=0$. Hence, 
\[
\minCDarrowwidth1em
\begin{CD}
	0 @>\partial>> H_k(\mathcal V^{\high}) @>i_*>> H_k(\mathcal Q )
@>\pi_*>> 0
\end{CD}
\]
meaning that $H_k(\mathcal V^{\high}) \approx
	H_k(\mathcal Q)$.
\end{proof}

\begin{proof}[Proof of Theorem~\ref{th-homology}]
	Since $\mathcal V^{\high}$ deformation retracts onto $\mathcal V^{\max}$ along the reverse Kuramoto flow, it
	is enough to compute the Betti numbers of $\mathcal V^{\high}$. 
The long exact sequence of homology is:
\[ 
\minCDarrowwidth1em
\begin{CD}
\dots 
@>\partial>>  
H_k(\mathcal V^{\high}) @>i_*>>
	H_k(\mathcal Q ) @>\pi_*>> 
	H_k(\mathcal Q, \mathcal V^{\high}) 
@>\partial>>
H_{k-1}(\mathcal V^{\high})
@>>> 
\dots  
\end{CD}
\]

If $k < m-d-2$, Lemma~\ref{k-small} 
guarantees that $\dim H_k(\mathcal V^{\high}) = \binomial{m-1}{k}$.

Suppose next that $k \ge m-d-1$. By Proposition~\ref{pi-injective}, $\pi_*$ is 
injective and therefore $i_*$ vanishes in dimension $\ge k$ at
the sequence below:
\[ 
\minCDarrowwidth1em
\begin{CD}
\dots 
@>\partial>>
	H_{k+1}(\mathcal V^{\high}) @>i_*>>  
	H_{k+1}(\mathcal Q ) @>\pi_*>> 
	H_{k+1}(\mathcal Q, \mathcal V^{\high}) 
@>\partial>>
H_{k}(\mathcal V^{\high})
@>i_*>> 0
.
\end{CD}
\]
It follows that  
$H_{k+1}(\mathcal Q, \mathcal V^{\high}) 
=H_{k+1}(\mathcal Q) \otimes H_{k}(\mathcal V^{\high})$.
This gives the value of $\beta_k = \binomial{m-1}{k+2}$ for $d <  k \le m-3$.

It remains to compute the Betti numbers for  $k=m-d-2=\lceil \frac{m-2}{2} \rceil$.
The exact sequence is
\[ 
\minCDarrowwidth1em
\begin{CD}
0
	@>i_*>>
	H_{k+1}(\mathcal Q) @>\pi_*>> 
	H_{k+1}(\mathcal Q, \mathcal V^{\high}) 
@>\partial>>
H_{k}(\mathcal V^{\high})
	@>i_*>>
	H_{k}(\mathcal Q ) @>\pi_*>> 
0
.
\end{CD}
\]
From Proposition~\ref{pi-injective},
\[
	\binomial{m}{k+2} = \dim (H_{k+1}(\mathcal Q, \mathcal V^{\high})) =
	\binomial{m-1}{k+1} + \mathrm{rank}(\partial)
\]
so $\mathrm{rank}(\partial) = \binomial{m-1}{k+2}$. Therefore,
\[
	\dim (H_k(\mathcal V^{\high})) = \binomial{m-1}{k+2} + \binomial{m-1}{k}
.
\]

\end{proof}

\section{Imprints}
\label{sec-imprints}


\begin{definition*} Let $p \in \mathcal Q$ be a saddle. The {\bf imprint} of $p$
	is the closure of the set of all $\alpha(\Theta)$ so that $\Theta \in W^s(p)$ but $\Theta \ne p$.
\end{definition*}

We consider first the imprints of minimal dimension in $\mathcal V^{\max}$ for $m$ odd. In that case
we can write $m=2d+1$ and the saddles have index at most $d$. Given $|I|=d$, we write
$p_I=q(\Theta)$ where $\theta_i = \pi$ for $i=I$ and $\theta_i=0$ otherwise. The stable manifold
$W^s(p_I)$ is $d$-dimensional.

Recall the notation: If $J \subset [m]$, then $\mathcal Q^J$ stands for the subtorus in $\mathcal Q$ with $\theta_i=\theta_j$ for all $i, j \not \in J$. In particular $\mathcal Q^{\emptyset}$ is the sink $[0, \dots, 0]$, and for $|J|<(m-1)/2$ the notation
$\mathcal Q^{J}$ stands for the {\bf template} $\overline{W^u(p_J)}$. For $1<|J|<m-1$, $\mathcal Q^{J}$ is $|J|$-dimensional while
$\mathcal Q^{[m]\setminus J}$ is $m-|J|$-dimensional, and their intersection is one-dimensional. 
If $|J| \ge m-1$ then $\mathcal Q^J = \mathcal Q$.

\begin{theorem}\label{t.imprint1} 
	If $m=2d+1$ and $\abs{I}=d$,
	then the imprint of $p_I$ is the set 
$\mathcal Q^{[m]\setminus I}\ \cap\ \mathcal V^{\max}$.  The imprint is a $d-1$-sphere and is not contractible to a point
in $\mathcal V^{\max}$. It is a totally geodesic surface.
\end{theorem}

\begin{proof} Recall that the maximum of the Kuramoto potential $V_{\mathcal Q}$
is equal to $m^2/2$ and is attained at the set $\mathcal V^{\max}$.
	Pick $\epsilon>0$ small enough, and consider the level set
$\mathcal V^{\epsilon} = V_{\mathcal Q}^{-1}(m^2/2-\epsilon)$ in $\mathcal Q$. Every orbit of the 
reduced Kuramoto flow $\psi_t$
	passing through the mid potential region $\mathcal V^{\mid}$
must necessarily cross $\mathcal V^{\epsilon}$ for $t$ negative enough. The level set can be only crossed once, 
and contains no fixed point. This defines
a smooth map $A_{\epsilon}: p \rightarrow \psi_t(p)$ from the mid potential region onto $\mathcal V^{\epsilon}$. By parameterizing orbits by the value of $V$, one actually obtains a deformation retract from the mid potential region onto $\mathcal V^{\epsilon}$.

In order to simplify notations, let $I=[d]$ and
assume cube coordinates with $\theta_m=0$. In that system of coordinates,
$p_I=(\pi, \dots, \pi, 0, \dots, 0)$ with $d$ copies of $\pi$ and
$m-d-1$ tailing zeros. The unstable manifold $W^u(p_I)$ is
	the set of all $\{ \theta_1, \dots, \theta_d, 0, \dots, 0\}$ with
	$\theta_i \ne 0 \mod 2\pi$. Its closure is the set $\mathcal Q^I$.
For any $\delta > 0$, let $\mathcal C_{\delta}$ be the set of points of the form
\[
	(\pi, \dots, \pi, \delta y_1, \dots, \delta y_{m-d-1}),\hspace{2em} \mathbf y \in S(E^s),
\]
	where $S(E^s)$ is the unit sphere on the negative eigenspace of the Hessian $H$ , which we identify to a submanifold in quotient space. 
	The set $\mathcal C_{\delta}$ is an
	$m-d-2$-dimensional smooth topological sphere
	with center $p_I$, and `winding around' the subtorus $\mathcal Q^I$. It is not an Euclidean sphere because we are not assuming an orthonormal basis, but it is contained in a $2 \delta$-neighborhood of $p_I$. The sphere $\mathcal C_{\delta}$ cannot be continuosly collapsed to a point in $\mathcal Q \setminus \mathcal Q^I$. Indeed, consider the projection of $\mathcal Q$ onto the last $m-d-1$ coordinates $\theta_{d+1}, \dots, \theta_{m-1}$, minus the point $[0, \dots, 0]$ origin. It is a retraction. If $\mathcal C_{\delta}$ could be collapsed into a point on $\mathcal Q \setminus \mathcal Q^I$ by a continous deformation, then its retract on $\mathbb T^{m-d-1} \setminus \{ [0, \dots, 0]\}$ could also be collapsed to a point. This retract is the geometrical sphere of radius $\delta$ around $[0, \dots, 0]$. Suppose that $H(y,s)$ is a homotopy with $H(y,0)=y$ and $H(y,1)$ constant. Then $H$ lifts to a continuous homotopy in $\mathbb R^{m-d-1} \setminus \{0\}$, which is impossible.
\medskip

	The flow $\psi_t$ deformation retacts $\mathcal C_{\delta}$ into a set $\mathcal D_{\epsilon,\delta}=A_{\epsilon}(\mathcal C_{\delta}) \subset \mathcal V^{\epsilon}$. Hence,
$\mathcal C_{\delta}$ and $\mathcal D_{\delta,\epsilon}$ have the same homotopy class
in the space $\mathcal Q \setminus \mathcal Q^I$. 
As $\delta \rightarrow 0$, $\mathcal D_{\epsilon,\delta}$ accumulates on 
$W^s(p_I) \cap \mathcal V^{\delta}$. This also induces a homotopy equivalence from
$\mathcal D_{\epsilon,\delta}$ to the limit set $\mathcal D_{\epsilon}$ in $\mathcal Q \setminus \mathcal Q^I$. 

	The imprint is $\mathcal D=\lim_{\epsilon \rightarrow 0} \mathcal D_{\epsilon}$. From Corollary~\ref{c.alpha-retracts}, the high potential set $\mathcal V^{\high}$ deformation retracts onto $\mathcal V^{\max}$. This induces a homotopy from $\mathcal D_{\epsilon}$ to $\mathcal D$, and $\mathcal D$ has the homotopy class of a $m-d-2$-dimensional sphere in $\mathcal V^{\max}$.

By construction, all the points in $\mathcal C_{\delta}$ belong to
$\mathcal Q^{[m] \setminus I}$. By the equality principle, this is also
true for points of $\mathcal D_{\delta,\epsilon}$. Passing to the limits,
both $\mathcal D_{\epsilon}$ and $\mathcal D$ are subsets of $\mathcal Q^{[m] \setminus I}$.

Reciprocally, $\mathcal D$ contains all the cells of the form $\theta_1=\dots=\theta_d$ in
	$\mathcal V^{\max}$. There is one $d-1$ cell for each strict ordering of 
	$0 < \theta_j < 2\pi$, $j=d+1, \dots, m$. Indeed, each of those cells contains
	the $\alpha$ limit of the orbits obtained by ordering the variables $y_j$ in the
	same way. For instance, the cell $\cell{0a-b-c-d}$ is obtained as the $\alpha$-limit
	of the points $(\pi,\pi,\delta y_1,\delta y_2,\delta y_3,\delta y_4)$ with $\delta \rightarrow 0$
	and $0<y_1<y_2<y_3<y_4<\pi$.

	Finally, we need to show that the imprint is totally geodesic. This follows from the
	fact that the imprint is fixed by the subgroup $S_{d} \subset S_m$ of permutations of the
	first $d$ variables. Any vector tangent to the imprint must be orthogonal to the
    dz-vector 
    \[
    \begin{pmatrix}
    z , 
    \dots ,
    z ,
    -d ,
    \dots,
    -d
    \end{pmatrix}^T
    \]
    with $z=d+1$, $d$ entries $z$, and $z$ entries $-d$ (See Section \ref{s.eigenstructure}).
    Since the imprint is codimension 1 in $\mathcal V^{\max}$, the tangent space to the imprint is precisely
    the space of vectors tangent to $\mathcal V^{\max}$ that are orthogonal to the dz-vector.

    Any curve $\gamma(t)$ tangent to the imprint
    is also permutation invariant. In particular, if $\gamma(t)$ is a curve in the imprint parameterized
    by arc length, the geodesic vector field $\gamma''(t)$ is permutation invariant. Hence it is tangent to the
    dz-vector and orthogonal to the imprint. Thus, geodesics of $\mathcal V^{\max}$ that are tangent to the imprint
    stay in the imprint.

\end{proof}

\begin{remark*}There is also a cohomological approach for the non-triviality of the imprint that may be more
	intuitive to many readers.
The {\bf winding number} of a closed plane curve $\gamma$ around the origin
is 
\[
	w(\gamma) = \frac{1}{2\pi} \oint_{\gamma} \frac{1}{x^2+y^2} (x\mathrm dy - y \mathrm dx).
\]
The winding number is not defined when the curve crosses the origin. Since 
the form $\frac{1}{x^2+y^2} (x \mathrm dy - y \mathrm dx)$ is closed, the winding number is a topological invariant with values in $\mathbb Z$. 

If $\gamma$ is a closed curve in $\mathbb R^4$, we can define the winding number
around the coordinate space $x_3=x_4=0$ in the same way,
\[
	w(\gamma) = \frac{1}{2\pi} \oint_{\gamma} \frac{1}{x_3^2+x_4^2} (x_3 \mathrm dx_4 - x_4 \mathrm dx_3).
\]
Again, the winding number is not defined when the curve crosses the space $x_3=x_4=0$. It is a topological invariant with values on $\mathbb Z$.

Extending the definition to the quotient set $\mathcal Q = q(\mathbb T^5)$, we can define the winding number around the template (subtorus) 
	$\mathcal Q^{\{3,4,5\}} = \overline{W^u(q(p_{1,2}))}$ as the sum
\[
	w(\gamma) = \sum_{a,b \in \mathbb Z}
	\frac{1}{2\pi} \int_{\eta} \frac{1}{(x_3-2\pi a)^2+(x_4-2\pi b)^2} (x_3 -2 \pi a) \mathrm dx_4 - (x_4 - 2 \pi b) \mathrm dx_3.
\]
In the integral above, $\eta$ is a lifting of the closed curve $\gamma: [0,T] \rightarrow \mathcal Q$ to the universal cover $\mathbb R^4$ of $\mathcal Q$. The curve $\eta$ is not necessarily closed. Because $w$ is biperiodic, if $\gamma_1$ and $\gamma_2$ are homologous,
then $w(\gamma_1) = w(\gamma_2)$.

	Suppose that $m=5$ so $d=2$. Take $\mathcal C_{\delta}$ as circle with center
	$p_{\{1,2\}}$ around $\mathcal Q^{\{3,4,5\}}$. The same limit procedure
	as before will keep $w(\mathcal C_{\delta})=w(\mathcal D_{\epsilon,\delta})=w(\mathcal D_{\epsilon})=w(\mathcal D)$ invariant and equal to one. Therefore, $\mathcal D$ has the homology of a circle and is a circle. 

The very same construction can be generalized to higher dimensions.
The operator $w$ defines a cohomology class, and this argument proves that
the homology of $\mathcal D$ is the homology of a $S^{m-d-2}$-sphere. This is
a slightly weaker result than the homotopy argument.
\end{remark*}

\begin{theorem}\label{t.imprint2}
	If $m=2d+2$ and $\abs{I}=d$,
	then the imprint of $p_I$ is the set 
$\mathcal Q^{[m]\setminus I}\ \cap\ \mathcal V^{\max}$.  The imprint is a $d+2$-pinched $d$-sphere, that is a $d$ sphere
	with $d+2$ distinguished pairs of antipodal points identified.
\end{theorem}

\begin{proof}
	The proof is exactly the same as the proof of Theorem~\ref{t.imprint1} with a 
	minor difference: some of the singularities of $\mathcal V^{\max}$ appear twice in the imprint.
	For instance, if $d=1$, the imprint for $I=\{1\}$ is all of $\mathcal V^{\max}$ and contains 3 
	double points $\cell{0a-bc}$, $\cell{0b-ac}$ and $\cell{0c-ab}$.
\end{proof}

\begin{theorem}\label{t.imprint3}
	Let $\abs{I}=1$, then the imprint of $p_I$ is all of $\mathcal V^{\max}$. 
	In general, the imprint of $p_I$ is the set $\mathcal Q^{[m]\setminus I}\ \cap\ \mathcal V^{\max}$. 
\end{theorem}

\begin{figure}
	\centerline{\resizebox{0.5\textwidth}{!}{\includegraphics{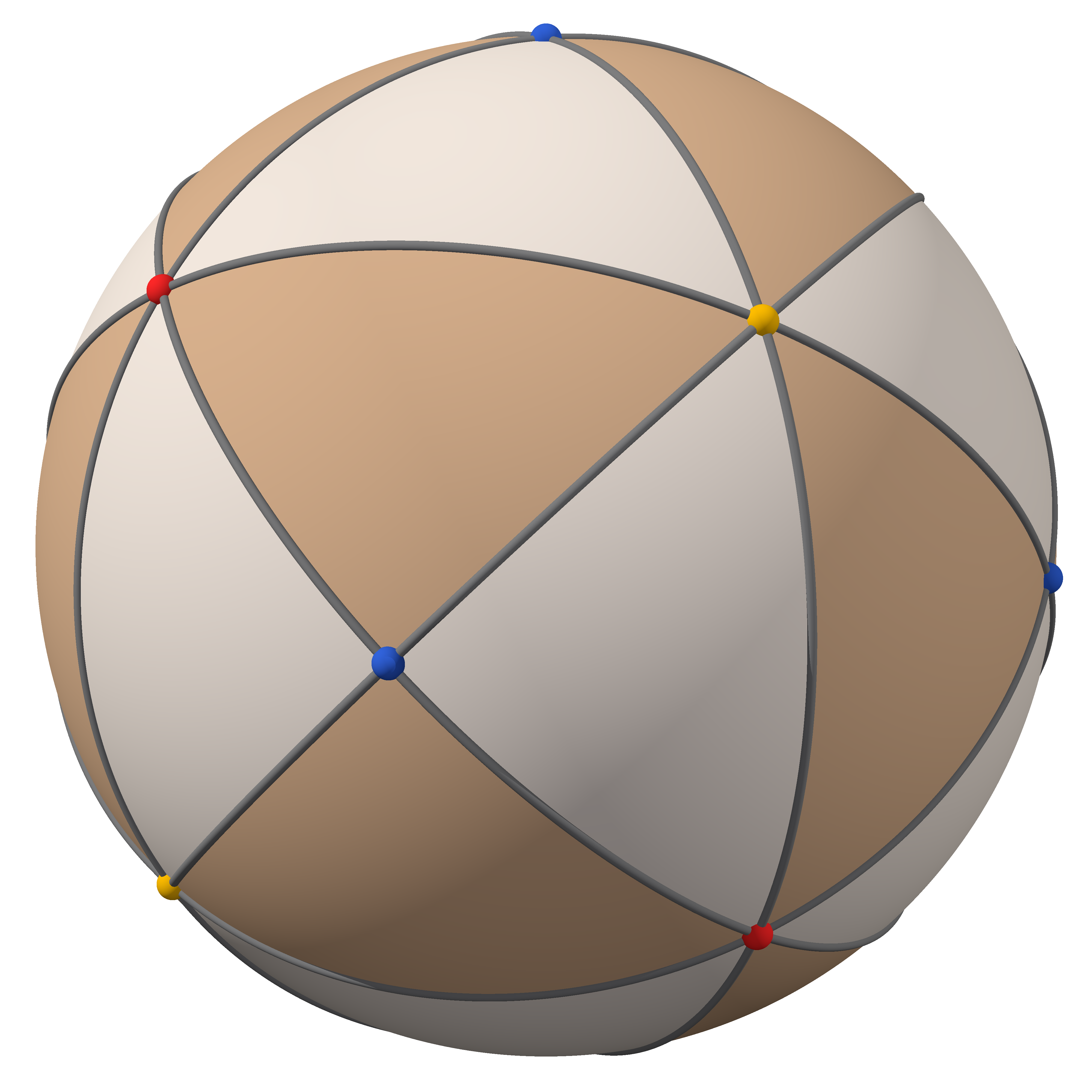}}}
	\caption{The sphere $S(E^s)$ for $m=5$, each triangle corresponds to a possible ordering of the
	$y_j$. Red and yellow dots correspond to directions $\pm \mathrm{f}_i$ in tangent space, where
	$\mathrm{f_i}$ is the projection onto $T\mathcal Q$ of the canonical basis vector $\mathrm e_i$.
	Blue dots represent orbits with two pairs of equal coordinates. Those orbits come directly from
	$\mathcal V^{\max}$.
	Figure by \ocite{Piesk}.
        \label{beachball}}
\end{figure}

\begin{proof}
We will only prove the case $m=5$ and explain how to generalize for larger values of $m$.
	Assume without loss of generality that $I=\{1\}$. Let $S(E^s)$ be the unit sphere in 
the negative eigenspace of the Hessian $H$ at $p_I$, write
\[
	S_I=\{	q(\pi, \delta y_1, \dots, \delta y_{4})),\hspace{2em} \mathbf \mathbf y \in S(E^s)\}.
\]
	(See figure \ref{beachball}).
	An orthogonal basis for $E^s$ was provided by Theorem~\ref{t.eigenstructure}. One can show by induction that
	all possible strict orderings for the variables $y_j$ appear in $S(E^s)$.
	One can decompose $S_I$ as follows: there are $24$ two-dimensional cells corresponding to the
strict orderings. The two-cells are separated by 12 circles of the form $y_i=y_j$. Circles may intersect
in straight angle, for instance $y_1=y_2$ and $y_3=y_4$. They may also intersect at angle of $\pi/3$ in a triple intersection,
	for instance $y_2=y_3=y_4$. 
	Triple intersections are special. They represent points in $S_I$ that belong to the unstable manifold of an
	index $2$ cell, in the previous example $(\pi,\pi,0,0,0)$. 

	A typical cell is $y_1<y_2<y_3<y_4$. Its border is the union of the arcs
	$y_1=y_2<y_3<y_4$, $y_1 < y_2 = y_3 < y_4$ and $y_1<y_2<y_3 \le y_4$. There is one straight angle intersection,
	$y_1=y_2 < y_3=y_4$ and two intersections of angle $pi/3$, viz. 
	$y_1< y_2=y_3=y_4$ and $y_1=y_2=y_3<y_4$ . 
	The $\alpha$-limit maps the two-dimensional and one-dimensional cells into respectively,
	$\cell{0-a-b-c-d}$, $\cell{0-ab-c-d}$, $\cell{0-a-bc-d}$ and $\cell{0-a-b-cd}$. The straight angle intersection
	in mapped onto $\cell{0-ab-cd}$. Orbits infinitesimally close to the intersections of angle $\pi/3$ pass arbitrarily
	close to the saddles $[\pi,\pi,0,0,0]$ and $[\pi,0,0,0,\pi]$ so their imprints accumulate into the imprints of 
	those saddles.  

	We claim that the cells $\cell{0d-a-b-c}$ and $\cell{0a-b-c-d}$ are actually contained in the imprint of $p_I$.
	To see this, consider an infinitesimal circle around $[\pi,\pi,0,0,0]$. It accumulates onto an infinitesimal
	circle on the stable eigenspace of $p_{\{1,2\}}$. Therefore, the imprint of $p_{\{1,2\}}$ is contained
	in the imprint of $p_{\{1\}}$. The same hold for the imprint of $p_{\{1,5\}}$.
	Since the same is true for all orderings of the $y_i$, the imprint of $p_i$ is all of $\mathcal V^{\max}$.

	For the general situation, one proceeds by induction on $d-|I|$ where $m=2d+1$ or $m=2d+2$. The initial case
	was done in Theorems~\ref{t.imprint1} and \ref{t.imprint2}. 
	For the induction step, it is sufficient to establish that the imprint of $p_{[k]}$ contains the imprint of $p_{[k+1]}$.
	As before, let $E^s$ be the negative Eigenspace of $H$ at $p_{[k]}$, and consider the $m-k-1$-infinitesimal sphere
	\[
		S_{[k]}=\{ q(\pi, \dots, \pi, \delta y_{m-k+1}, \dots, \delta y_m), y \in S(E^s)\}
	\]
	The restriction $y_{m-k-2}=\dots=y_{m}$ defines a pair of points in $S_{[k]}$, both in $W^u(p_{[k+1]})$.
	Let $q$ be one of those points. The backward orbits of an infinitesimal $m-k-2$-sphere around $q$
	pass infinitesimally close to the infinitesimal sphere $S{[k+1]}$ around $p_{[k+1]}$. Therefore,
	those orbits accumulate onto the imprint of $p_{[k+1]}$.
\end{proof}

\begin{figure}
\centerline{\includegraphics[width=\textwidth]{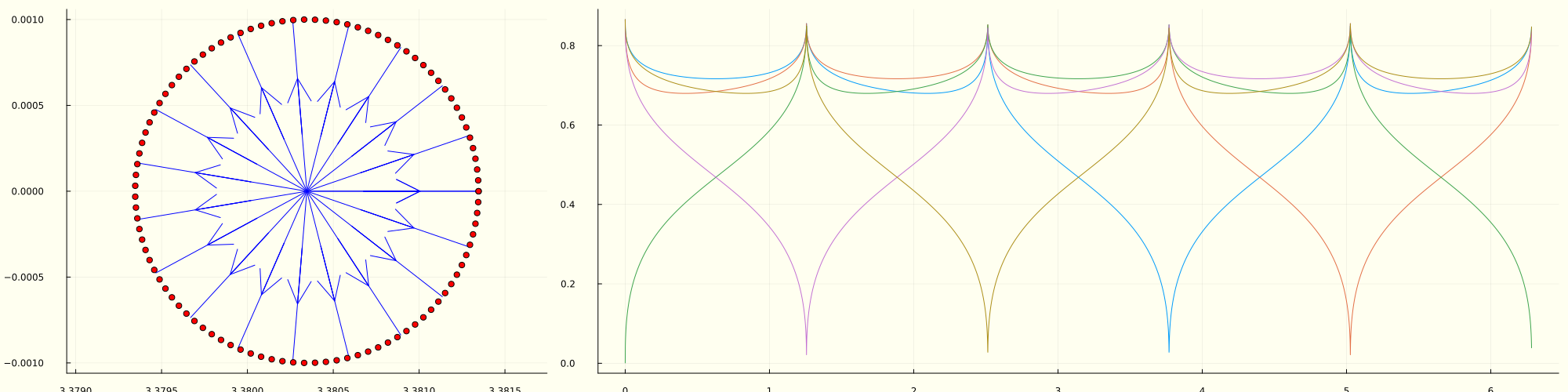}}
	\centerline{\includegraphics[width=\textwidth]{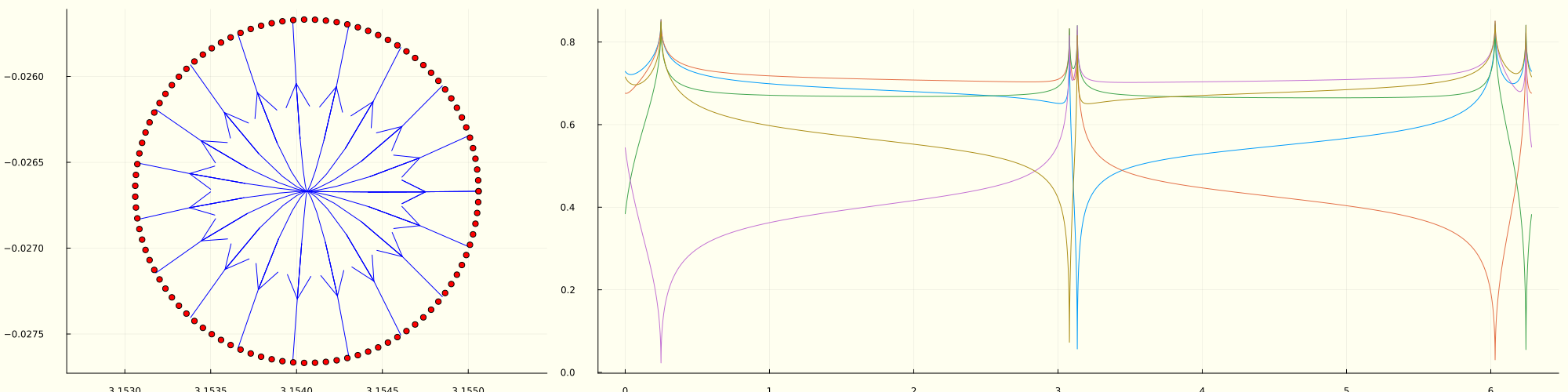}}

	\caption{Top left: orbits of the (quotient) Kuramoto flow crossing a circle of radius $1/100$ around the point
	$\left [0,\frac{2\pi}{5}, \frac{4\pi}{5},\frac{6\pi}{5},\frac{8\pi}{5}\right] \in \mathcal V^{\max}$.
	The circle is normal to $\mathcal V^{\max}$ and the orbits are linear, because the two eigenvalues of the Hessian $H$
	are the same. Top right: the intersection of the orbit of the circle with the 
	surface $V=m-2$ was computed numerically. The plots show the Euclidean distance to each saddle of index $1$. 
	Bottom left and right: same experiment
	with a generic (random) point in $\mathcal V^{\max}$. Orbits leaving $\mathcal V^{\max}$ are tangent to one of the
	eigenvectors of the Hessian $H$, most of them to the same eigenvector. Orbits connectiong $\mathcal V^{\max}$
	to different index $1$ saddle may be tangent for $t \rightarrow -\infty$.
	\label{fig-Vmax-saddle}}
\end{figure}

	One can also ask about the geometry of the orbits of the Kuramoto flow near $\mathcal V^{\max}$.
	Figure~\ref{fig-Vmax-saddle} pictures orbits in a normal slice around a couple of points in $\mathcal V^{\max}$. Most of
	those orbits go to the sink, and some of them are connected to the
	index $1$ saddles. The evidence suggests that
	the orbit of a typical infinitesimal circle around $\mathcal V^{\max}$ accumulates on a circuit 
	connecting the saddles $p_{\{i\}}$, $i=1, \dots, m$, through their respective templates. Up to sign, this circuit
	is homologous in the skeleton $K$ to the concatenation of the paths $t \mapsto 2 \pi t \mathrm f_i$,
	$t\in [0,1]$. Here, {\em typical}
	means that the center of the circle belongs to one of the open $m-3$-dimensional cells. The ordering in the
	concatenation depends on which cell and on the orientation of the circle. In this case, connections with
	index $\ge 2$ saddles are precluded. But when crossing boundaries of $m-3$-dimensional cells, connections
	to index $\ge 2$ saddles do occur as described in Theorems~\ref{t.imprint1} and \ref{t.imprint2}.

\section{Conclusions}
\label{s.conclusions}

In the present paper we have given a complete topological description of the dynamics of the simplified
Kuramoto model, viz.
\[
	\dot \theta_i = \sum_{j=1}^m \sin(\theta_i - \theta_j)
.
\]
For low values of the potential, the dynamics can be described by the structure of {\bf templates}. Those templates
are linear subtori, and they arise as the unstable manifolds of the saddles in quotient space. Inside each template,
the Kuramoto vector field is topologically conjugated to the gradient field of a Morse function.

The dynamics is trivial for intermediate values of the potential, and contains a source $\mathcal V^{\max}$ which is a possibly singular compact surface of codimension 2. 
In the case $m=5$, $\mathcal V^{\max}$ is a smooth two-dimensional surface of genus 4.
A general cell decomposition for $\mathcal V^{\max}$ was obtained for general $m$, and its homology was computed.
We obtained also a description of the {\bf imprints} of the saddles on $\mathcal V^{\max}$, that is the
intersection of $\mathcal V^{\max}$ with the closure of the stable manifold of the saddle.

The isolated singularities of the Kuramoto vector field on quotient space are finite and hyperbolic. There is only one
surface of non-isolated singularities that is $\mathcal V^{\max}$, which is normally hyperbolic. We claim
that most of our conclusions are valid under small perturbations of the model. However, we should restrict the
perturbations to those that are {\bf diagonally invariant}, so they induce a flow in quotient space $\mathcal Q$.
For instance, a more general formulation of Kuramoto's model is
\[
	\dot \theta_i = \omega_i - \sum_{j=1}^m a_{ij} \sin(\theta_i - \theta_j)
.
\]
A trivial change of coordinates reduces this system to the case $\sum_i \omega_i=0$. For coefficients $|\omega_i|$ small enough
and $|a_{ij}-1|$ small enough, this is a perturbation of the simplified Kuramoto model studied in this paper. Recall that $\psi_t$ is the (unperturbed) Kuramoto flow $K$ in $\mathcal Q$ for $\omega=0$ and $a_{ij}=1$. 

\begin{theorem}\label{t.robustness}
There is a $\mathcal C^2$-neighborhood of the Kuramoto vector field on $\mathcal Q$ so that,
for any $\tilde K$ in that neighborhood, the flow $\tilde \psi_t$ of $\tilde K$ is topologically
conjugate to $\psi_t$ on $\mathcal V \setminus \mathcal V^{\high}$, in the sense that there is a
homeomorphism $g: \mathcal Q \setminus \mathcal V^{\high} \rightarrow \mathcal Q \setminus \mathcal V^{\high}$
and a reparameterization $\gamma: \mathbb R \times  \mathcal Q \setminus \mathcal V^{\high}  \rightarrow \mathbb R$,
strictly increasing in the first parameter,
so that for all $t \ge 0$,
\begin{equation}\label{conj}
\begin{CD}
\mathcal Q \setminus \mathcal V^{\high}
 @>\text{\normalsize
$\qquad 
\psi_t(\Theta)\qquad$}>> 
\mathcal Q \setminus \mathcal V^{\high}
\\
@V\text{\normalsize$
g
$}VV @VV\text{\normalsize$
g
$}V
\\
\mathcal Q \setminus \mathcal V^{\high} @>\text{\normalsize$\qquad 
{\tilde\psi} _{\gamma(t,\Theta)}(\Theta)  
\qquad$}>>
\mathcal Q \setminus \mathcal V^{\high}
\end{CD}
\end{equation}

\end{theorem}

Theorem~\ref{t.robustness} is a direct consequence of the theorem on structural stability for flows on manifolds with boundary, due to \ocite{Robinson80}*{Theorem C}. 

Therefore, our conclusions
hold, up to topological conjugacy, for perturbations of the Kuramoto model. In particular, they hold for small
perturbations of the parameters $\omega_i \simeq 0$ and $a_{ij} \simeq 1$. 
The only exception is
$\mathcal V^{\max}$: normal hyperbolicity guarantees an invariant set close to $\mathcal V^{\max}$, but the perturbation
may introduce nontrivial dynamics inside that set.

\renewcommand{\MR}[1]{}


\begin{bibsection}

\begin{biblist}
\bib{DorflerBullo}{article}{
   author={D\"orfler, Florian},
   author={Bullo, Francesco},
   title={Synchronization in complex networks of phase oscillators: a
   survey},
   journal={Automatica J. IFAC},
   volume={50},
   date={2014},
   number={6},
   pages={1539--1564},
   issn={0005-1098},
   review={\MR{3214901}},
   doi={10.1016/j.automatica.2014.04.012},
}
\bib{ElderingKvalheimRevzen}{article}{
   author={Eldering, Jaap},
   author={Kvalheim, Matthew},
   author={Revzen, Shai},
   title={Global linearization and fiber bundle structure of invariant manifolds},
   journal={Nonlinearity},
   volume={31},
   date={2018},
   number={9},
   pages={4202--4245},
   issn={0951-7715},
   review={\MR{3841342}},
   doi={10.1088/1361-6544/aaca8d},
}

\bib{Hatcher}{book}{
   author={Hatcher, Allen},
   title={Algebraic topology},
   publisher={Cambridge University Press, Cambridge},
   date={2002},
   pages={xii+544},
   isbn={0-521-79160-X},
   isbn={0-521-79540-0},
   review={\MR{1867354}},
}
\bib{HirschPughShub}{book}{
   author={Hirsch, M. W.},
   author={Pugh, C. C.},
   author={Shub, M.},
   title={Invariant manifolds},
   series={Lecture Notes in Mathematics},
   volume={Vol. 583},
   publisher={Springer-Verlag, Berlin-New York},
   date={1977},
   pages={ii+149},
   review={\MR{0501173}},
}

	\bib{IzhikevichKuramoto}{article}{
   author={Izhikevich, E.M.},
   author={Kuramoto, Y.},
   title={Weakly coupled oscillators},
   booktitle={Encyclopedia of Matthematical Physics: Five-Volume Set},
   year={2006},
   pages={V448--453},
	doi={10.1016/B0-12-512666-2/00106-1},
   }


\bib{Kuramoto75}{article}{
   author={Kuramoto, Yoshiki},
   title={Self-entrainment of a population of coupled non-linear
   oscillators},
   conference={
      title={International Symposium on Mathematical Problems in Theoretical
      Physics},
      address={Kyoto Univ., Kyoto},
      date={1975},
   },
   book={
      series={Lecture Notes in Phys.},
      volume={39},
      publisher={Springer, Berlin-New York},
   },
   isbn={3-540-07174-1},
   date={1975},
   pages={420--422},
   review={\MR{0676492}},
}
\bib{Kuramoto}{book}{
   author={Kuramoto, Y.},
   title={Chemical oscillations, waves, and turbulence},
   series={Springer Series in Synergetics},
   volume={19},
   publisher={Springer-Verlag, Berlin},
   date={1984},
   pages={viii+156},
   isbn={3-540-13322-4},
   review={\MR{0762432}},
   doi={10.1007/978-3-642-69689-3},
}

\bib{Kuramoto2026}{misc}{
      title={Half a century of the theory of synchronization}, 
      author={Kuramoto,Yoshiki},
      year={2026},
      note={Preprint, arXiv, \url{https://arxiv.org/abs/2602.20505 }},
      date={2026},
}

\bib{Piesk}{misc}{
	author={Tilman Piesk},
	note={Disdyakis polyhedra, figure licensed under Common Contents CC BY-SA 4.0,
	and downloaded on Sept 24, 2025 from
	\url{https://commons.wikimedia.org/wiki/File:Disdyakis_6.png}},
	date={2020}
}
\bib{Robinson74}{article}{
   author={Robinson, Clark},
   title={Structural stability of vector fields},
   journal={Ann. of Math. (2)},
   volume={99},
   date={1974},
   pages={154--175},
   issn={0003-486X},
   review={\MR{0334283}},
   doi={10.2307/1971016},
}
\bib{Robinson80}{article}{
   author={Robinson, Clark},
   title={Structural stability on manifolds with boundary},
   journal={J. Differential Equations},
   volume={37},
   date={1980},
   number={1},
   pages={1--11},
   issn={0022-0396},
   review={\MR{0583334}},
   doi={10.1016/0022-0396(80)90083-2},
}
\bib{Smale1960}{article}{
   author={Smale, Stephen},
   title={Morse inequalities for a dynamical system},
   journal={Bull. Amer. Math. Soc.},
   volume={66},
   date={1960},
   pages={43--49},
   issn={0002-9904},
   review={\MR{0117745}},
   doi={10.1090/S0002-9904-1960-10386-2},
}
\bib{Strogatz}{article}{
   author={Strogatz, Steven H.},
   title={From Kuramoto to Crawford: exploring the onset of synchronization
   in populations of coupled oscillators},
   journal={Phys. D},
   volume={143},
   date={2000},
   number={1-4},
   pages={1--20},
   issn={0167-2789},
   review={\MR{1783382}},
   doi={10.1016/S0167-2789(00)00094-4},
}


\end{biblist}
\end{bibsection}
\end{document}